\documentclass[a4paper]{article}

\usepackage[dvipdfmx]{graphicx}
\usepackage{amsmath,amssymb,amsthm}
\usepackage{mathtools}

\usepackage[capitalize]{cleveref}
\usepackage[all]{xy}
\usepackage{enumitem}

\theoremstyle{plain}
\newtheorem{theorem}{Theorem}[section]
\newtheorem{proposition}[theorem]{Proposition}
\newtheorem{corollary}[theorem]{Corollary}
\newtheorem{lemma}[theorem]{Lemma}

\theoremstyle{definition}
\newtheorem*{definition}{Definition}

\theoremstyle{remark}
\newtheorem{example}[theorem]{Example}
\newtheorem{remark}[theorem]{Remark}
\newtheorem*{notation}{Notation}

\numberwithin{equation}{section}

\title{Categories of operators and actions of group operads}
\date{\today}
\author{Jun Yoshida}

\DeclareMathAlphabet{\mathpzc}{OT1}{pzc}{m}{it}
\DeclareMathAlphabet{\mathbfcal}{OMS}{cmsy}{b}{n}

\DeclareSymbolFont{bsletter}{OML}{cmbrm}{b}{it}
\DeclareMathSymbol{\bsrho}{\mathord}{bsletter}{"1A}
\DeclareMathSymbol{\bspsi}{\mathord}{bsletter}{"20}
\DeclareMathSymbol{\bsvarphi}{\mathord}{bsletter}{"27}

\newcommand{\llangle}{\mathopen{\langle\mkern-3mu\langle}}
\newcommand{\rrangle}{\mathclose{\rangle\mkern-3mu\rangle}}

\newcommand\objof{\operatorname{Ob}}
\newcommand\opposite{\mathrm{op}}

\newcommand{\overbar}[1]{\mkern 1.5mu\overline{\mkern-1.5mu#1\mkern-1.5mu}\mkern 1.5mu}

\newcommand{\blank}{\mkern1mu\mathchar"7B\mkern1mu}

\newcommand{\pbcorner}{\text{\LARGE{$\mathrlap{\cdot}\lrcorner$}}}
\newcommand{\bpcorner}{\text{\LARGE{$\llcorner\mathllap\cdot$}}}
\newcommand\dblslash{\mathbin{/\mkern-5mu/}}

\makeatletter
\newcommand*{\relrelbarsep}{.386ex}
\newcommand*{\@relrelbar}[2]{%
  \raise #2\hbox to 0pt{$\m@th#1\relbar$\hss}%
  \lower #2\hbox{$\m@th#1\relbar$}}

\newcommand*{\relrelbar}{%
  \mathrel{\mathpalette\@relrelbar\relrelbarsep}}

\makeatother

\newcommand\xrightrightarrows[1]{%
  \mathrel{\mathrlap\relrelbar\overset{#1\ }{\,\rightrightarrows}}}

\begin{document}
\maketitle

\tableofcontents

\section*{Introduction}
\label{sec:intro}
\addcontentsline{toc}{section}{Introduction}

Operads, introduced by Stasheff \cite{Stasheff1963} and May \cite{May1973}, play essential roles in the study of higher algebras.
For example, the little disks operads present the levels of higher commutativities.
On the other hand, May and Thomason introduced in their paper \cite{MayThomason1978} the notion of \emph{categories of operators}.
In contrast to operads, which are defined in an algebraic way, a category of operators is just a fibration in an appropriate sense over the opposite $\Gamma^\opposite$ of the Segal's category satisfying certain conditions.
They pointed out that every topological operad gives rise to a topological category of operators and proved the equivalence of algebras over them.
The algebraic operations in operads are presented by certain \emph{universal lifting properties} of categories of operators.
Hence, the relation of operads and categories of operators is compared with that of $G$-spaces and fiber bundles over the classifying space $BG$ for a group $G$.
This geometric nature of categories of operators was made use of in the definition of $\infty$-operads by Lurie \cite{Lur14}, which is one of the well-used models of \emph{homotopy operads}.

The original definition of categories of operators, however, only covers the \emph{symmetric} operads (or multicategories more generally).
On the other hand, Zhang introduced \emph{group operads} \cite{Zha11} as generalizations of the operad $\mathfrak S$ of symmetric groups.
As pointed out by Gurski \cite{Gur15}, a group operad $\mathcal G$ may have an action on multicategories; namely, for a non-symmetric multicategory $\mathcal M$, the action of $\mathcal G$ on $\mathcal M$ is defined to be a right action
\begin{equation}
\label{eq:intro:action}
\left(\coprod_{\sigma\in\mathfrak S_n}\mathcal M(a_{\sigma(1)}\dots a_{\sigma(n)};a)\right)\times\mathcal G(n)
\to \coprod_{\sigma\in\mathfrak S_n}\mathcal M(a_{\sigma(1)}\dots a_{\sigma(n)};a)
\end{equation}
on the multihom-set for $a,a_1,\dots, a_n\in\objof\mathcal M$ with an appropriate compatibility condition with the compositions.
In the author's previous paper \cite{YoshidaGrpop}, it is called a \emph{$\mathcal G$-symmetric structure} and defined in terms of a monoidal structures on the category $\mathbf{GrpOp}$ of group operads.
More precisely, a $2$-monad $(\blank)\rtimes\mathcal G$ on the $2$-category $\mathbf{MultCat}$ of multicategories is defined so that a $\mathcal G$-symmetric structure is nothing but a structure of a (strict) $2$-algebra over $(\blank)\rtimes\mathcal G$.
As a result, we obtain a $2$-category $\mathbf{MultCat}_{\mathcal G}$ of $\mathcal G$-symmetric multicategories.

It is then natural to ask for a $\mathcal G$-symmetric analogue of categories of operators.
This is exactly the main theme of this paper.
Note that, we already have a \emph{non-symmetric} analogue: let $\nabla$ be the category of intervals; i.e.
\begin{itemize}
  \item objects are totally ordered sets of the following form for $n\in\mathbb N$:
\[
\llangle n\rrangle := \{-\infty,1,\dots,n,\infty\}\ ;
\]
  \item morphisms are order-preserving maps sending $\pm\infty$ to $\pm\infty$ respectively.
\end{itemize}
Using the duality $\nabla\cong\Delta^\opposite$ due to Joyal, where $\Delta$ is the simplex category, for a non-symmetric multicategory $\mathcal M$, a multicategorical version of the \emph{categorical wreath product} appearing in \cite{Berger2007} and \cite{Rezk2010} gives rise to a fibered category
\[
\mathcal M^\triangledown := \Delta\wr\mathcal M \to \nabla\ .
\]
Note that the objects of $\mathcal M^\triangledown$ are finite sequences $\vec a=a_1\dots a_n$ of objects $a_1,\dots,a_n\in\objof\mathcal M$.
Moreover, if $\mu_n:\llangle n\rrangle\to\llangle 1\rrangle\in\nabla$ is the morphism which is the dual of the map $[1]\to[n]\in\Delta$ with $0\mapsto 0$ and $1\mapsto n$, we have a pullback square
\begin{equation}
\label{eq:intro:trigdown}
\vcenter{
  \xymatrix{
    \mathcal M(a_1\dots a_n;a) \ar[r] \ar[d] \ar@{}[dr]|(.4)\pbcorner & \mathcal M^\triangledown(a_1\dots a_n,a) \ar[d] \\
    \{\mu_n\} \ar[r] & \nabla(\llangle n\rrangle,\llangle 1\rrangle) }}
\quad.
\end{equation}
Actually, the functor $\mathcal M^\triangledown\to\nabla$ is characterized by \eqref{eq:intro:trigdown} with the universal lifting property with respect to a certain class of morphisms in $\nabla$ as well as \emph{Segal condition}; i.e. there is a canonical equivalence of categories
\[
\mathcal M^\triangledown_{\llangle n\rrangle} \simeq (\mathcal M^\triangledown_{\llangle 1\rrangle})^{\times n}
\]
on the fibers $\mathcal M^\triangledown_{\llangle k\rrangle}:=\mathcal M^\triangledown\times_\nabla\{\llangle k\rrangle\}$.
Hence, we are interested in the counterpart of the construction $(\blank)\rtimes\mathcal G$ for each group operad $\mathcal G$ in the side of categories of non-symmetric operators.

Unfortunately, the category $\nabla$ may not be enough for this purpose.
Indeed, the action \eqref{eq:intro:action} may change the domain of multimorphisms while, in view of \eqref{eq:intro:trigdown}, no fiberwise action on $\mathcal M^\triangledown$ can realize this phenomenon.
This is mainly because $\nabla$ has no non-trivial isomorphism, so we have to consider an extension of $\nabla$ so that the symmetry of $\mathcal G$ is taken into account.
We can make use of the results obtained in \cite{YoshidaGrpop}: the category $\mathbf{GrpOp}$ of group operads can be thought of as a reflective full subcategory of the slice category $\mathbf{CrsGrp}_\nabla^{/\mathfrak S}$ of the category of \emph{crossed interval groups}.
Hence, we can consider the \emph{total category} $\nabla_{\mathcal G}$.
Furthermore, quotient categories of $\nabla_{\mathcal G}$ is important.
For example, as pointed out by Segal \cite{Segal1968}, a monoid $M$ gives rise to a functor $\widehat M^\otimes:\nabla\to\mathbf{Set}$ given by $\widehat M^\otimes(\llangle n\rrangle):=M^{\times n}$ and, for $\varphi:\llangle m\rrangle\to\llangle n\rrangle\in\nabla$, $\varphi^{}_\ast:=\widehat M^\otimes(\varphi):M^{\times m}\to M^{\times n}$ with
\[
\varphi^{}_\ast(x_1,\dots,x_m)
:= \left(\prod_{\varphi(i)=1}x_i,\dots,\prod_{\varphi(i)=n} x_i\right)\ ,
\]
where the products are taken in the appropriate orders.
The \emph{Grothendieck construction} enables us to regard $\widehat M^\otimes$ as a discrete fibration $M^\otimes\to\nabla$, so it is a discrete version of a category of operators.
Notice that a functor $X:\nabla_{\mathcal G}\to\mathbf{Set}$ is equivalent to data
\begin{itemize}
  \item a functor $X:\nabla\to\mathbf{Set}$;
  \item for each $n\in\mathbb N$, a left $\mathcal G(n)$-action on the set $X_n:=(\llangle n\rrangle)$, say $\mathcal G(n)\times X_n\to X_n;\ (u,x)\mapsto x^u$;
\end{itemize}
such that, for each $\varphi:\llangle m\rrangle\to \llangle n\rrangle\in\nabla$, $x\in X(\llangle m\rrangle)$, and $v\in\mathcal G(n)$,
\[
\varphi_\ast(x)^v = (\varphi^v)_\ast(x^{\varphi^\ast(v)})\ .
\]
Now, each $\mathcal G(n)$ acts on $\widehat M^\nabla(\llangle n\rrangle)=M^{\times n}$ in the obvious way so that it gives rise to an extension $\widehat M^\nabla_{\mathcal G}:\nabla_{\mathcal G}\to\mathbf{Set}$.
In particular, when $\mathcal G=\mathfrak S$, we assert that the monoid $M$ is commutative if and only if the functor $\widehat M^\nabla_{\mathfrak S}:\nabla_{\mathfrak S}\to\mathbf{Set}$ factors through an appropriate quotient category of $\nabla_{\mathfrak S}$.
For each $(x_1,\dots,x_n)\in M^{\times n}$, and for each $\sigma\in\mathfrak S_n$, we have
\[
(\mu_n)_\ast\left((x_1,\dots,x_n)^\sigma\right)
= (\mu_n)_\ast(x_{\sigma^{-1}(1)},\dots,x_{\sigma^{-1}(n)})
= x_{\sigma^{-1}(1)}\dots x_{\sigma^{-1}(n)}\ .
\]
It follows that $M$ is commutative if and only if the two morphisms $(\mu_n,\sigma),(\mu_n,e_n):\llangle n\rrangle\to\llangle 1\rrangle\in\nabla_{\mathcal G}$, here $e_n\in\mathcal G(n)$ is the unit, induce the same map $\widehat M^\otimes_{\mathcal G}(\llangle n\rrangle)\to\widehat M^\otimes_{\mathcal G}(\llangle 1\rrangle)$ for every $n$ and $\sigma\in\mathcal G(n)$.
In other words, the associated discrete fibration $M^\nabla_{\mathfrak S}\to\nabla_{\mathcal G}$ is a pullback along a quotient $q:\nabla_{\mathfrak S}\to\mathcal Q$ such that $q(\mu_n,\sigma)=q(\mu_n,e_n)$.
More generally, for a group operad $\mathcal G$, each quotient of the total category $\nabla_{\mathcal G}$ may present a $\mathcal G$-symmetry on the fibrations over it.

In this point of view, we will establish a $\mathcal G$-symmetric analogue of categories of operators in the following way.
After reviewing the basic results of group operads and crossed interval groups in \cref{sec:grpop}, we will attempt to the classification of a sort of quotient categories of $\nabla_G$ for arbitrary crossed interval group $G$ in \cref{sec:quot-tot}.
It will be seen that the \emph{congruence families} on $G$ determine and are determined by quotient categories.
For a congruence family $K$ on $G$, we write $\mathcal Q_K$ the associated quotient of $\nabla_G$.
Basic constructions and some technical results on congruence families will be discussed.
In particular, in the case $G$ is a group operad $\mathcal G$, two special congruence families $\overbar{\operatorname{Kec}}^{\mathcal G}\subset\overbar{\operatorname{Dec}}^{\mathcal G}$ will be introduced.

We will see in \cref{sec:assoc-dblcat} that if we have a pair $(K,L)$ of congruence families satisfying a certain condition, the quotient category $\mathcal Q_L$ become a part of an internal category
\begin{equation}
\label{eq:intro:quot-dbl}
\mathcal Q_{L\dblslash K}\rightrightarrows\mathcal Q_L
\end{equation}
in the category $\mathbf{Cat}$ of small categories.
In other words, \eqref{eq:intro:quot-dbl} is a \emph{double category}.
For example, if $(K,L)=(\overbar{\operatorname{Dec}}^{\mathcal G},\overbar{\operatorname{Kec}}^{\mathcal G})$, we write $\mathbb G_{\mathcal G}:=\mathcal Q_{\overbar{\operatorname{Kec}}^{\mathcal G}\dblslash\overbar{\operatorname{Dec}}^{\mathcal G}}$ and $\mathbb E_{\mathcal G}:=\mathcal Q_{\overbar{\operatorname{Kec}}^{\mathcal G}}$, which are the key to establish $\mathcal G$-symmetric analogue of categories of operators.
Moreover, further quotient
\[
\widetilde{\mathcal Q}_{L\dblslash K}
\rightrightarrows\widetilde{\mathcal Q}_L
\]
will be discussed.

Having an internal category, we are interested in \emph{internal presheaves} over it.
In \cref{sec:intpresh-assocdbl}, we will see that a $\mathcal G$-symmetric multicategory $\mathcal M$ gives rise to an internal presheaf $\mathcal M\wr\widetilde{\mathbb E}_{\mathcal G}$ over the double category $\widetilde{\mathbb G}_{\mathcal G}\rightrightarrows\widetilde{\mathbb E}_{\mathcal G}$.
It will turn out that the construction extends to a $2$-functor
\begin{equation}
\label{eq:intro:wrE}
\mathbf{MultCat}
\to \mathbf{PSh}(\widetilde{\mathbb G}_{\mathcal G}\rightrightarrows\widetilde{\mathbb E}_{\mathcal G})
\end{equation}
from the $2$-category of (non-symmetric) multicategories to that of internal presheaves over $\widetilde{\mathbb G}_{\mathcal G}\rightrightarrows\widetilde{\mathbb E}_{\mathcal G}$.
It will proved that the fibered product $(\mathcal M\wr\widetilde{\mathbb E}_{\mathcal G})\times_{\widetilde{\mathbb E}_{\mathcal G}}\widetilde{\mathbb G}_{\mathcal G}$ is exactly the image of the \emph{free $\mathcal G$-symmetrization} $\mathcal M\rtimes\mathcal G$ under the $2$-functor \eqref{eq:intro:wrE}.
In other words, it is the required counterpart in the side of fibrations.

In the last two sections, we will compute the essential image of the $2$-functor \eqref{eq:intro:wrE}.
We will propose a notion of \emph{categories of algebraic $\mathcal G$-operators} as analogues of $\infty$-operads by Lurie \cite{Lur14}.
Namely, they form a $2$-subcategory $\mathbf{Oper}^{\mathsf{alg}}_{\mathcal G}$ of $\mathbf{PSh}(\widetilde{\mathbb G}_{\mathcal G}\rightrightarrows\widetilde{\mathbb E}_{\mathcal G})$ consisting of internal presheaves $\mathcal X$ over the double category such that the canonical functor $\mathcal X\to\widetilde{\mathbb E}_{\mathcal G}$ satisfies a universal lifting property and Segal condition.
Actually, they are alternative models of $\mathcal G$-symmetric multicategories; we will prove the biequivalence $\mathbf{MultCat}_{\mathcal G}\simeq\mathbf{Oper}^{\mathsf{alg}}_{\mathcal G}$ in \cref{sec:algopcat-equiv}.

We finally note that we say the models above are ``algebraic'' because the actions of $\mathcal G$ as in \eqref{eq:intro:action} are realized in an algebraic way.
It is also possible to realize them in a ``geometric'' way; for example, as fibrations over appropriate categories.
It will be established in the author's future work.

\subsection*{Acknowledgment}

I would like to first express my gratitude to my supervisor Prof. Kohno who provided valuable comments and advice.
Some important ideas were obtained during visit to Macquarie University.
In particular, the discussion with Prof. Street was really fruitful.
I am also indebted to Genki Sato whose opinion have helped me.

This work was supported by the Program for Leading Graduate Schools, MEXT, Japan.

\section{Group operads and crossed interval groups}
\label{sec:grpop}

We review the notion of group operads.
It was introduced by Zhang \cite{Zha11} though the axioms were already stated in 1.2.0.2 in the paper \cite{Wah01}.
The further theory was developed in \cite{CG13} and \cite{Gur15} while they use different terminology ``action operads.''
We give just sketches; for the details, we refer the reader to the literature above and the author's previous work \cite{YoshidaGrpop}.

For each $n\in\mathbb N$, set $\mathfrak S(n)=\mathfrak S_n$ to be the $n$-th symmetric group.
It turns out that the family $\mathfrak S=\{\mathfrak S(n)\}_n$ admits a structure of an operad so that, for $\sigma,\tau\in\mathfrak S(n)$, $\sigma_i,\tau_i\in\mathfrak S(k_i)$, we have
\[
\gamma_{\mathfrak S}(\sigma\tau;\sigma_1\tau_1,\dots,\sigma_n\tau_n)
= \gamma_{\mathfrak S}(\sigma;\sigma_{\tau^{-1}(1)},\dots,\sigma_{\tau^{-1}(n)})\gamma(\tau;\tau_1,\dots,\tau_n)\ ,
\]
where $\gamma$ is the composition in the operad structure.
The group operads are generalizations of this example.

\begin{definition}
A \emph{group operad} is an operad $\mathcal G$ together with data
\begin{itemize}
  \item a group structure on each $\mathcal G(n)$;
  \item a map $\mathcal G\to\mathfrak S$ of operads so that each $\mathcal G(n)\to\mathfrak S(n)$ is a group homomorphism, which gives rise to a left $\mathcal G(n)$-action on $\langle n\rangle$;
\end{itemize}
which satisfy the identity
\begin{equation}
\label{eq:grpop-interchange}
\gamma_{\mathcal G}(xy;x_1y_1,\dots,x_ny_n)
= \gamma_{\mathcal G}(x;x_{y^{-1}(1)},\dots,x_{y^{-1}(n)})\gamma_{\mathcal G}(y;y_1,\dots,y_n)
\end{equation}
for every $x,y\in\mathcal G(n)$ and $x_i,y_i\in\mathcal G(k_i)$ for $1\le i\le n$.
\end{definition}

\begin{example}
The operad $\mathfrak S$ is an example of group operads with the identity map $\mathfrak S\to\mathfrak S$.
\end{example}

\begin{example}
\label{ex:grpop-braid}
For each $n\in\mathbb N$, set $\mathcal B(n)$ to be the braid group on $n$ strands.
Then, in a similar manner to $\mathfrak S$, one can find an operad structure on the family $\mathcal B=\{\mathcal B(n)\}_n$ so that the canonical maps $\mathcal B(n)\to\mathfrak S(n)$ define a map of operads.
The canonical map $\mathcal B\to\mathfrak S$ together with the group structure on each $\mathcal B(n)$ exhibits $\mathcal B$ as a group operad.
\end{example}

\begin{definition}
A \emph{map of group operads} is a map $F:\mathcal G\to\mathcal H$ of operads satisfying the following conditions:
\begin{enumerate}[label={\upshape(\roman*)}]
  \item each map $F:\mathcal G(n)\to\mathcal H(n)$ is a group homomorphism;
  \item $F$ respects the maps into $\mathfrak S$; i.e. the diagram below commutes:
\[
\xymatrix@C=1.5em{
  \mathcal G \ar[rr]^F \ar[dr] && \mathcal H \ar[dl] \\
  & \mathfrak S & }
\]
\end{enumerate}
\end{definition}

We denote by $\mathbf{GrpOp}$ the category of group operads and maps of group operads.

One of the most important features of group operads is that they may act on multicategories.
Recall that a \emph{multicategory} $\mathcal M$ consists of a set $\objof\mathcal M$ and a set $\mathcal M(a_1\dots a_n;a)$ for each $a,a_1,\dots,a_n\in\mathcal M$ together with associative compositions and identities.
For a group operad $\mathcal G$, we define a multicategory $\mathcal M\rtimes\mathcal G$ as follows:
\begin{itemize}
  \item $\objof(\mathcal M\rtimes\mathcal G)=\objof\mathcal M$;
  \item for $a,a_1,\dots,a_n\in\objof\mathcal M$, set
\[
\begin{multlined}
(\mathcal M\rtimes\mathcal G)(a_1\dots a_n;a) \\
:=\left\{(f,x)\;\middle|\;x\in\mathcal G(n),\ f\in\mathcal M(a_{x^{-1}(1)}\dots a_{x^{-1}(n)};a)\right\}\ ;
\end{multlined}
\]
  \item the composition operation is given by
\[
\begin{multlined}
\gamma_{\mathcal M\rtimes\mathcal G}\bigl((f,x);(f_1,x_1),\dots,(f_n,x_n)\bigr) \\
= \left(\gamma_{\mathcal M}(f;f_{x^{-1}(1)},\dots,f_{x^{-1}(n)}),\gamma_{\mathcal G}(x;x_1,\dots,x_n)\right)\ .
\end{multlined}
\]
\end{itemize}
The assignment $\mathcal M\mapsto\mathcal M\rtimes\mathcal G$ extends in a canonical way to an endo-$2$-functor
\[
(\blank)\rtimes\mathcal G:\mathbf{MultCat}\to\mathbf{MultCat}
\]
on the $2$-category of multicategories, multifunctors, and multinatural transformations.
Moreover, there are two identity-on-objects multifunctors
\begin{equation}
\label{eq:MG-2monad}
\begin{gathered}
M:\mathcal M\rtimes\mathcal G\rtimes\mathcal G\to\mathcal M\rtimes\mathcal G
\ ;\quad (f,x,y)\mapsto (f,xy) \\
H:\mathcal M\to\mathcal M\rtimes\mathcal G
\ ;\quad f \mapsto (f,e)\ ,
\end{gathered}
\end{equation}
where $e$ is the unit of the group $\mathcal G(n)$ for an appropriate $n\in\mathbb N$.
Clearly, the multifunctors \eqref{eq:MG-2monad} are strictly natural with respect to $\mathcal M\in\mathbf{MultCat}$, and one can see the triple $\bigl((\blank)\rtimes\mathcal G,M,H\bigr)$ forms a (strict) $2$-monad on the $2$-category $\mathbf{MultCat}$ in the sense of \cite{BlackwellKellyPower1989}.
The following result is easy to verify.

\begin{lemma}
\label{lem:grpopmap-monadmap}
The assignment $\mathcal G\mapsto (\blank)\rtimes\mathcal G$ extends to a functor
\[
\mathbf{GrpOp}\to 2\mathchar`-\mathbf{Mnd}(\mathbf{MultCat})\ ,
\]
here the codomain is the category of (strict) $2$-monads on $\mathbf{MultCat}$ and $2$-monad transformations.
\end{lemma}

\begin{definition}[cf. Definition~5.1 in \cite{Gur15}]
Let $\mathcal G$ be a group operad.
Then, a \emph{$\mathcal G$-symmetric multicategory} is a multicategory $\mathcal M$ equipped with a structure of a (strict) algebra for the $2$-monad $(\blank)\rtimes\mathcal G$.
More explicitly, a $\mathcal G$-symmetric multicategory is a multicategory $\mathcal M$ together with a multifunctor $\mathpzc A:\mathcal M\rtimes\mathcal G\to\mathcal M$ such that the following diagrams of multifunctors are (strictly) commutative:
\[
\vcenter{
  \xymatrix{
    \mathcal M\rtimes\mathcal G\rtimes\mathcal G \ar[r]^-{\mathpzc A\rtimes\mathcal G} \ar[d]_M & \mathcal M\rtimes\mathcal G \ar[d]^{\mathpzc A} \\
    \mathcal M\rtimes\mathcal G \ar[r]^{\mathpzc A} & \mathcal M }}
\qquad
\vcenter{
  \xymatrix@C=.5em{
    \mathcal M \ar@{=}[rr] \ar[dr]_H && \mathcal M \\
    & \mathcal M\rtimes\mathcal G \ar[ur]_{\mathpzc A} }}
\]
\end{definition}

\begin{remark}
If $\mathcal M$ be a $\mathcal G$-symmetric multicategory with $\mathpzc A:\mathcal M\rtimes\mathcal G\to\mathcal G$, then $\mathpzc A$ is the identity on objects.
In this case, by abuse of notations, for $f\in\mathcal M(a_1\dots a_n;a)$ and for $x\in\mathcal G(n)$, we will write $f^x:=\mathpzc A(f,x)$.
Note that we have
\[
f^x \in \mathcal M(a_{x(1)}\dots a_{x(n)};a)\ .
\]
\end{remark}

\begin{example}
A \emph{symmetric multicategory} (resp. a \emph{symmetric operad}) in the usual sense is nothing but an $\mathfrak S$-symmetric multicategory (resp. an $\mathfrak S$-symmetric operad) in the sense above.
\end{example}

\begin{example}
Let $\ast$ be the trivial group operad.
Then $\ast$-symmetric multicategories are nothing but ordinary multicategories in our convention.
\end{example}

\begin{example}
\label{ex:braidmon-mulcatBsym}
Let $\mathcal C$ be a monoidal category, and put $\mathcal C^\otimes$ the associated multicategory; i.e. $\objof\mathcal C^\otimes=\objof\mathcal C$ and $\mathcal C^\otimes(X_1\dots X_n;X)=\mathcal C(X_1\otimes\dots\otimes X_n;X)$.
Hence, for each $X_1,\dots,X_n\in\mathcal C$, we have a multimorphism
\[
u_{X_1\dots X_n}\in\mathcal C^\otimes(X_1\dots X_n;X_1\otimes\dots\otimes X_n)
\]
corresponding to the identity.
If $\mathcal C^\otimes$ is endowed with a $\mathcal G$-symmetric structure, then for each $x\in\mathcal G(n)$, we define an isomorphism
\[
\Theta^x_{X_1\dots X_n}:X_{x(1)}\otimes\dots\otimes X_{x(n)}\to X_1\otimes\dots\otimes X_n
\]
to be the one corresponding to the multimorphism $u^x_{X_1\dots X_n}$.
It turns out that $\Theta^x$ is a natural isomorphism with $\Theta^x\Theta^y=\Theta^{xy}$.
For example, in the case $\mathcal G=\mathcal B$ (resp. $\mathfrak S$), the resulting structure on $\mathcal C$ is nothing but a braided (resp. symmetric) structure on the monoidal structure.
\end{example}

\begin{definition}
Let $\mathcal G$ be a group operad, and let $\mathcal M$ and $\mathcal N$ be $\mathcal G$-symmetric multicategories.
Then, a \emph{$\mathcal G$-symmetric multifunctor} $\mathcal M\to\mathcal N$ is a multifunctor which is a homomorphism of algebras for the $2$-monad $(\blank)\rtimes\mathcal G$.
\end{definition}

We denote by $\mathbf{MultCat}_{\mathcal G}$ the $2$-category of $\mathcal G$-symmetric multicategory, $\mathcal G$-symmetric multifunctors, and multinatural transformations.
In view of \cref{lem:grpopmap-monadmap}, a map $F:\mathcal G\to\mathcal H$ of group operads induces a $2$-functor
\[
F^\ast:\mathbf{MultCat}_{\mathcal H}\to\mathbf{MultCat}_{\mathcal G}\ .
\]
In particular, for every group operad $\mathcal G$, there are canonical $2$-functors
\[
\mathbf{MultCat}_{\mathfrak S}
\to\mathbf{MultCat}_{\mathcal G}
\to\mathbf{MultCat}\ .
\]

To understand group operads, as pointed out in \cite{YoshidaGrpop}, the notion of \emph{crossed interval groups} is convenient.
We consider the category $\nabla$ given as follows:
\begin{itemize}
  \item objects are totally ordered sets of the form
\[
\llangle n\rrangle:=\{-\infty,1,\dots,n,\infty\}
\]
for $n\in\mathbb N$;
  \item morphisms are order-preserving maps which send $\pm\infty$ to $\pm\infty$ respectively.
\end{itemize}
Then, a crossed interval group is a $\nabla$-set $G$ equipped with data
\begin{itemize}
  \item a group structure on $G_n=G(\llangle n\rrangle)$ for each $n\in\mathbb N$ and
  \item a left group action 
\[
G_n\times\nabla(\llangle m\rrangle,\llangle n\rrangle)
\to\nabla(\llangle m\rrangle,\llangle n\rrangle)
\ ;\quad (x,\varphi)\mapsto\varphi^x
\]
for each $m,n\in\mathbb N$
\end{itemize}
satisfying the equations
\[
\begin{gathered}
\varphi^\ast(xy)=(\varphi^y)^\ast(x)\varphi^\ast(y) \\
(\varphi\psi)^x = \varphi^x\psi^{\varphi^\ast(x)}
\end{gathered}
\]
for $x,y\in G(n)$, $\varphi:\llangle m\rrangle\to\llangle n\rrangle$, and $\psi:\llangle l\rrangle\to\llangle m\rrangle$.
In addition, for crossed interval groups $G$ and $H$, a map $G\to H$ of $\nabla$-sets is called a \emph{map of crossed interval groups} if it preserves the structures above.
We denote by $\mathbf{CrsGrp}_\nabla$ the category of crossed interval groups and maps of them.
By Theorem~2.4 in \cite{YoshidaLimcolim}, $\mathbf{CrsGrp}_\nabla$ is a locally presentable category and has all (small) limits and colimits.

\begin{remark}
The notion of crossed groups was originally introduced by Fiedorowicz and Loday \cite{FL91} and by Krasauskas \cite{Kra87} in the simplicial case.
Although it is easily generalized to arbitrary base categories, crossed interval groups was first studied by Batanin and Markl \cite{BataninMarkl2014}.
Actually, the terminology is due to them.
\end{remark}

\begin{example}
The sequence $\mathfrak S=\{\mathfrak S_n\}$ actually admits a structure of a crossed interval groups.
Actually, it is a subobject of the terminal object in $\mathbf{CrsGrp}_\nabla$ computed in \cite{YoshidaLimcolim}.
\end{example}

\begin{theorem}[Theorem~3.3 in \cite{YoshidaGrpop}]
\label{theo:grpop->crsint}
There is a fully faithful functor
\[
\widehat\Psi:\mathbf{GrpOp}\to\mathbf{CrsGrp}_\nabla^{/\mathfrak S}
\]
such that $\widehat\Psi(\mathcal G)_n=\mathcal G(n)$ for each $n\in\mathbb N$.
\end{theorem}

The $\nabla$-set structure on $\widehat\Psi(\mathcal G)$ is described as follows: note that morphisms $\varphi:\llangle m\rrangle\to\llangle n\rrangle\in\nabla$ correspond in one-to-one to $(n+2)$-tuples
\[
\vec k=(k_{-\infty},k_1,\dots,k_n,k_\infty)
\]
of non-negative integers with $\sum k_j=m$ via $\varphi\mapsto \vec k^{(\varphi)}$ with
\begin{equation}
\label{eq:nabla-tuple}
k^{(\varphi)}_j :=
\begin{cases}
\#\varphi^{-1}\{j\} & 1\le j\le n\ , \\
\#\varphi^{-1}\{j\}-1 & j=\pm\infty\ .
\end{cases}
\end{equation}
Then, for $\varphi:\llangle m\rrangle\to\llangle n\rrangle\in\nabla$, the induced map is given by
\[
\varphi^\ast:\mathcal G(n)\to\mathcal G(m)
\ ;\quad x\mapsto \gamma_{\mathcal G}\bigl(e_3;e^{(\varphi)}_{-\infty},\gamma_{\mathcal G}(x;e^{(\varphi)}_1,\dots,e^{(\varphi)}_n),e^{(\varphi)}_\infty\bigr)\ ,
\]
where $e^{(\varphi)}_j:=e_{k^{(\varphi)}_j}$.

\begin{example}
\Cref{theo:grpop->crsint} justifies the coincidence of the notation $\mathfrak S$.
Indeed, the functor $\widehat\Psi$ sends $\mathfrak S$ to $\mathfrak S$.
Identifying $\mathbf{GrpOp}$ with its image in $\mathbf{CrsGrp}_\nabla$, we can hence identify $\mathfrak S$.
\end{example}

\section{Quotients of the total category}
\label{sec:quot-tot}

As pointed out in Introduction, we can regard any quotients of $\nabla_G$ may present a kind of symmetries on monoids or higher variants.
Hence, the classification of the quotients is an important problem.
In particular, in this section, we focus on the quotients of the following form.

\begin{definition}
Let $G$ be a crossed interval group.
Then, a \emph{$G$-quotal category} is a category $\mathcal Q$ equipped with a functor $q:\nabla_G\to\mathcal Q$ satisfying the following conditions:
\begin{enumerate}[label={\upshape(\roman*)}]
  \item $q$ is full and bijective on objects, so we may assume $\objof\mathcal Q=\objof\mathcal\nabla$;
  \item for $\varphi,\varphi'\in\nabla(\llangle m\rrangle,\llangle n\rrangle)$ and $x,x'\in G_m$, the equality of morphisms
\[
q(\varphi,x)=q(\varphi',x'):\llangle m\rrangle\to\llangle n\rrangle\in\mathcal Q
\]
in $\mathcal Q$ implies $\varphi=\varphi'$.
\end{enumerate}
\end{definition}

\begin{lemma}
\label{lem:quotal-faithful}
Let $G$ be a crossed interval group, and let $\mathcal Q$ be a $G$-quotal category with $q:\nabla_G\to\mathcal Q$.
Then, the composition
\begin{equation}
\label{eq:quotal-faithful:comp}
\nabla
\hookrightarrow\nabla_G
\xrightarrow{q} \mathcal Q
\end{equation}
is faithful and conservative.
\end{lemma}
\begin{proof}
Let us denote by $e_m\in G_m$ the unit of the group.
Then, the composition \eqref{eq:quotal-faithful:comp} sends a morphism $\varphi:\llangle m\rrangle\to\llangle n\rrangle\in\nabla$ to $q(\varphi,e_m)$.
Hence, the condition on $G$-quotal categories directly implies \eqref{eq:quotal-faithful:comp} is faithful.

Next, suppose $\varphi:\llangle m\rrangle\to\llangle n\rrangle\in\nabla$ is a morphism such that $q(\varphi,e_m)$ is an isomorphism.
Since $q$ is full, the inverse of $q(\varphi,e_m)$ can be written in the form $q(\psi,y)$ with $\psi:\llangle n\rrangle\to\llangle m\rrangle\in\nabla$ and $y\in G_n$.
We have
\[
\begin{gathered}
\mathrm{id}_{\llangle m\rrangle}
= q(\psi,y)\circ q(\varphi,e_m)
= q(\psi\varphi^y,\varphi^\ast(y))
\\
\mathrm{id}_{\llangle n\rrangle}
= q(\varphi,e_m)\circ q(\psi,y)
= q(\varphi\psi,y)\ .
\end{gathered}
\]
By virtue of the condition on $G$-quotal categories, the former equation implies $\mathrm{id}_{\llangle m\rrangle}=\psi\varphi^y$ while the latter implies $\mathrm{id}_{\llangle n\rrangle}=\varphi\psi$.
It follows that $\psi$ is an inverse of $\varphi$ in $\nabla$, so \eqref{eq:quotal-faithful:comp} is conservative.
\end{proof}

Thanks to \cref{lem:quotal-faithful}, we can identify morphisms in $\nabla$ with their images in a $G$-quotal category; namely, if $q:\nabla_G\to\mathcal Q$ is a $G$-quotal category, then by abuse of notation, we write $\varphi=q(\varphi,e_m)$ for every $\varphi:\llangle m\rrangle\to\llangle n\rrangle\in\nabla$.

There  is a general recipe to construct $G$-quotal categories.
For this, we introduce some notions.

\begin{definition}
Let $G$ be a crossed interval group, and let $\varphi:\llangle m\rrangle\to\llangle n\rrangle\in\nabla$ be a morphism.
Then, an element $x\in G_m$ is called a \emph{right stabilizer of $\varphi$} if for every morphism $\psi:\llangle l\rrangle\to \llangle m\rrangle\in\nabla$, we have $\varphi\psi^\ast=\varphi\psi$.
We denote by $\mathrm{RSt}^G_\varphi\subset G_m$ the subset of right stabilizers of $\varphi$.
\end{definition}

It is obvious that, for $\varphi:\llangle m\rrangle\to\llangle n\rrangle\in\nabla$, the subset $\mathrm{RSt}^G_\varphi\subset G_m$ is a subgroup.

\begin{definition}
Let $G$ be a crossed interval group.
A \emph{congruence family} on $G$ is a family $K=\{K_\varphi\}_\varphi$ indexed by morphisms in $\nabla$ such that, for every $\varphi:\llangle m\rrangle\to\llangle n\rrangle\in\nabla$,
\begin{enumerate}[label={\upshape(\roman*)}]
  \item $K_\varphi$ is a subgroup of $\mathrm{RSt}^G_\varphi$;
  \item for every morphism $\chi:\llangle n\rrangle\to\llangle k\rrangle\in\nabla$, $K_\varphi\subset K_{\chi\varphi}$;
  \item for every morphism $\psi:\llangle l\rrangle\to\llangle m\rrangle\in\nabla$, the map $\psi^\ast:G_m\to G_l$ restricts to a map $K_\varphi\to K_{\varphi\psi}$;
  \item for every element $y\in G_n$, we have
\begin{equation}
\label{eq:congfam-conj}
\varphi^\ast(y)\cdot K_\varphi\cdot \varphi^\ast(y)^{-1}
= K_{\varphi^y}\ .
\end{equation}
\end{enumerate}
\end{definition}

\begin{remark}
\label{rem:congfam-crsgrp}
The first three conditions above implies $K=\{K_\varphi\}_\varphi$ forms a crossed group over $\operatorname{opTw}(\nabla):=\operatorname{Tw}(\nabla)^\opposite$ the opposite of the twisted arrow category of $\nabla$; $\operatorname{opTw}(\nabla)$ is the category such that
\begin{itemize}
  \item the objects are morphisms of $\nabla$;
  \item for morphisms $\varphi_i:\llangle m_i\rrangle\to\llangle n_i\rrangle\in\nabla$ for $i=1,2$, morphisms $\varphi_1\to\varphi_2$ in $\operatorname{Tw}(\nabla)$ are pairs $(\alpha,\beta)$ of morphisms in commutative squares of the form
\[
\vcenter{
  \xymatrix{
    \llangle m_1\rrangle \ar[r]^\alpha \ar[d]_{\varphi_1} & \llangle m_2\rrangle \ar[d]^{\varphi_2}  \\
    \llangle n_1\rrangle & \llangle n_2\rrangle \ar[l]_\beta }}
\quad;
\]
  \item the composition is given by
\[
(\gamma,\delta)\circ(\alpha,\beta)
= (\gamma\alpha,\beta\delta)\ .
\]
\end{itemize}
The second and the third conditions imply each morphism $(\alpha,\beta):\varphi_1\to\varphi_2$ in $\operatorname{opTw}(\nabla)$ induces a map
\[
(\alpha,\beta)^\ast:K_{\varphi_2}
\xrightarrow{\alpha^\ast} K_{\varphi_2\alpha}
\hookrightarrow K_{\beta\varphi_2\alpha} = K_{\varphi_1}\ .
\]
Moreover, for a morphism $(\alpha,\beta):\varphi_1\to\varphi_2\in\operatorname{opTw}(\nabla)$ as above, if $x\in G_{m_2}$ is a right stabilizer of $\varphi_2$, then the pair $(\alpha^x,\beta)$ is again a morphism $\varphi_1\to\varphi_2$ in $\operatorname{opTw}(\nabla)$.
Hence, by virtue of the first condition, this defines a left action of $K_{\varphi_2}$ on the set $\operatorname{opTw}(\nabla)(\varphi_1,\varphi_2)$.
It is easily verified that these data actually form a crossed $\operatorname{opTw}(\nabla)$-group structure on the family $K=\{K_\varphi\}_\varphi$.
\end{remark}

\begin{example}
\label{ex:congfam-rstab}
The family $\operatorname{RSt}^G=\{\operatorname{RSt}^G_\varphi\}_\varphi$ is itself a congruence family.
Indeed, the first three conditions for congruence families are obvious.
To verify \eqref{eq:congfam-conj}, observe that, for $\varphi:\llangle m\rrangle\to\llangle n\rrangle\in\nabla$, $\psi:\llangle l\rrangle\to\llangle l\rrangle\nabla$, $x\in\operatorname{RSt}^G_\varphi$, and $y\in G_n$, we have
\[
\varphi\psi^{\varphi^\ast(y)x\varphi^\ast(y)^{-1}}
= (\varphi\psi^{x\varphi^\ast(y)^{-1}})^y
= (\varphi\psi^{\varphi^\ast(y)^{-1}})^y
= \varphi\psi\ .
\]
Note that, in view of \cref{rem:congfam-crsgrp}, a congruence family on $G$ is nothing but a crossed $\operatorname{opTw}(\nabla)$-subgroup of $\operatorname{RSt}^G$ satisfying \eqref{eq:congfam-conj}.
\end{example}

\begin{example}
\label{ex:congfam-decomp}
Let $\mathcal G$ be a group operad.
Recall that morphisms $\varphi:\llangle m\rrangle\to\llangle n\rrangle\in\nabla$ correspond to $(n+2)$-tuples $\vec k=(k_{-\infty},k_1,\dots,k_n,k_\infty)$ by the formula \eqref{eq:nabla-tuple}.
We set a subset $\operatorname{Dec}^{\mathcal G}_\varphi\subset G_m$ to be the image of the map
\begin{equation}
\label{eq:congfam-decomp:map}
\begin{array}{ccc}
\mathcal G(k^{(\varphi)}_{-\infty})\times\mathcal G(k^{(\varphi)}_1)\times\dots\times\mathcal G(k^{(\varphi)}_n)\times\mathcal G(k^{(\varphi)}_\infty) &\mathclap\to& \mathcal G(m) \\
  (x_{-\infty},x_1,\dots,x_n,x_\infty) &\mathclap\mapsto& \gamma(e_{n+2};x_{-\infty},x_1,\dots,x_n,x_\infty)
\end{array}.
\end{equation}
As easily verified, the map \eqref{eq:congfam-decomp:map} is actually a group homomorphism, and $\operatorname{Dec}^G=\{\operatorname{Dec}^G_\varphi\}_\varphi$ forms a crossed $\operatorname{opTw}(\nabla)$-subgroup.
Moreover, for $y\in\mathcal G(n)$, we have
\[
\begin{split}
&\varphi^\ast(y)\cdot\gamma(e_{n+2};x_{-\infty},x_1,\dots,x_n,x_\infty) \\
&= \gamma(e_3;x_{-\infty},\gamma(y;e_{k^{(\varphi)}_1},\dots,e_{k^{(\varphi)}_n})\gamma(e_n;x_1,\dots,x_n),x_\infty) \\
&= \gamma(e_3;x_{-\infty},\gamma(e_n;x_{y^{-1}(1)},\dots,x_{y^{-1}(n)})\gamma(y;e_{k^{(\varphi)}_1},\dots,e_{k^{(\varphi)}_n}),x_\infty) \\
&= \gamma(e_{n+2};x_{-\infty},x_{y^{-1}(1)},\dots,x_{y^{-1}(n)})\cdot\varphi^\ast(y)
\ ,
\end{split}
\]
which implies
\[
\varphi^\ast(y)\cdot\operatorname{Dec}^G_\varphi\cdot\varphi^\ast(y)^{-1}
= \operatorname{Dec}^G_{\varphi^y}\ .
\]
Thus, $\operatorname{Dec}^G$ is a congruence family on $\mathcal G$.
\end{example}

\begin{example}
\label{ex:congfam-triv}
For each $\varphi:\llangle m\rrangle\to\llangle n\rrangle\in\nabla$, we set $\operatorname{Triv}_\varphi$ to consists of a single element $e_m$ which is seen as the unit of $G_m$ for any crossed interval group $G$.
Then, clearly $\operatorname{Triv}=\{\operatorname{Triv}_\varphi\}_\varphi$ is a congruence family on $G$.
In view of \cref{rem:congfam-crsgrp}, $\operatorname{Triv}$ is the trivial crossed $\operatorname{opTw}(\nabla)$-group.
\end{example}

\begin{example}
\label{ex:congfam-kernel}
Let $K$ be a congruence family on a crossed interval group $G$.
For each $\varphi:\llangle m\rrangle\to\llangle n\rrangle\in\nabla$, we have a canonical group homomorphism
\[
K_\varphi
\hookrightarrow\operatorname{RSt}^G_\varphi
\hookrightarrow G_m
\to\mathfrak W^\nabla_m\ .
\]
We put $K'_\varphi$ the kernel and claim that $K'=\{K'_\varphi\}_\varphi$ forms a congruence family.
Namely, it is an uncrossed $\operatorname{opTw}(\nabla)$-subgroup of $\operatorname{RSt}^G$ since it is the kernel of the map $\operatorname{RSt}^G\to\operatorname{RSt}^{\mathfrak W^\nabla}$ of crossed $\operatorname{opTw}(\nabla)$-groups.
This observation also leads to the equation \eqref{eq:congfam-conj}.
\end{example}

We see congruence families on a crossed interval group $G$ are associated to $G$-quotal categories.
We use the following lemma.

\begin{lemma}
\label{lem:congfam-pb}
Let $G$ be a crossed interval group, and let $K=\{K_\varphi\}_\varphi$ be a congruence family on $G$.
Suppose $\varphi:\llangle l\rrangle\to\llangle m\rrangle$ and $\psi:\llangle m\rrangle\to\llangle n\rrangle$ are morphisms in $\nabla$.
Then, for each $y\in G_n$, the composition
\[
\begin{array}{ccccl}
(K_\psi\cdot y)\times G_m &\hookrightarrow& G_n\times G_m &\to& G_m \\
                          && (y',x) &\mapsto& \varphi^\ast(y')\cdot x
\end{array}
\]
induces a maps
\[
\{K_\varphi\cdot y\}\times (K_\varphi\backslash G_m)
\to (K_{\psi\varphi^y}\backslash G_m)\ .
\]
\end{lemma}
\begin{proof}
Take $x\in G_m$ and $y\in G_n$.
Then, for $u\in K_\varphi$ and $v\in K_\psi$, we have
\begin{equation}
\label{eq:prf:congfam-pb}
\varphi^\ast(vy)\cdot ux
= (\varphi^y)^\ast(v)\cdot (\varphi^\ast(y)\cdot u\cdot \varphi^\ast(y)^{-1})\cdot \varphi^\ast(y)x\ .
\end{equation}
By virtue of the conditions on the congruence family $K$, the first term in the right hand side of \eqref{eq:prf:congfam-pb} belongs to $K_{\psi\varphi^y}$ while the second to $K_{\varphi^y}\subset K_{\psi\varphi^y}$.
Hence, the result follows.
\end{proof}

Now, for a congruence family $K$ on a crossed interval group $G$, we define a category $\mathcal Q_K$ as follows:
the objects are the same as $\nabla$, and for $m,n\in\mathbb N$,
\[
\mathcal Q_K(\llangle m\rrangle,\llangle n\rrangle)
=\left\{(\varphi,[x])\;\middle|\; \varphi\in\nabla(\llangle m\rrangle,\llangle n\rrangle), [x]\in K_{\varphi}\backslash G_m\right\}\ .
\]
There is an obvious map $q:\nabla_G(\llangle m\rrangle,\llangle n\rrangle)\to\mathcal Q(\llangle m\rrangle,\llangle n\rrangle)$.
Using \cref{lem:congfam-pb} and the inclusion $K_\varphi\subset\operatorname{RSt}^G_\varphi$, one can see the composition in the total category $\nabla_G$ induces a composition operation in $\mathcal Q_K$ so that $q$ is a functor.

\begin{proposition}
\label{prop:congfam-quotal}
For every congruence family $K$ on a crossed interval group $G$, the functor
\[
q:\nabla_G\to\mathcal Q_K
\]
given above exhibits $\mathcal Q_K$ as a $G$-quotal category.
Moreover, the assignment $K\mapsto\mathcal Q_K$ gives a one-to-one correspondence between congruence families on $G$ and (isomorphism classes of) $G$-quotal categories respecting the orders.
\end{proposition}
\begin{proof}
The first statement is obvious.
To see the assignment is one-to-one, suppose $\mathcal Q$ is a $G$-quotal category with $q:\nabla_G\to\mathcal Q$.
For each $\varphi:\llangle m\rrangle\to\llangle n\rrangle\in\nabla$, we put
\[
K^{\mathcal Q}_\varphi
:= \{x\in G_m\mid q(\varphi,x)=\varphi\}\ .
\]
We assert $K^{\mathcal Q}=\{K^{\mathcal Q}_\varphi\}_\varphi$ forms a congruence family on $G$.
First, clearly we have $K^{\mathcal Q}_\varphi\subset K^{\mathcal Q}_{\chi\varphi}$ and $\psi^\ast(K^{\mathcal Q}_\varphi)\subset K^{\mathcal Q}_{\varphi\psi}$ whenever the compositions make sense.
Next, for $\varphi:\llangle m\rrangle\to\llangle n\rrangle$ and $\psi:\llangle l\rrangle\to\llangle m\rrangle$, and for each $x\in K^{\mathcal Q}_\varphi$, we have
\begin{equation}
\label{eq:prf:congfam-quotal:Krst}
\varphi\psi
= q(\varphi,x)\psi
= q(\varphi\psi^x,\psi^\ast(x))\ ,
\end{equation}
here we identify morphisms in $\nabla$ with their images in $\mathcal Q$.
Since $\mathcal Q$ is $G$-quotal, \eqref{eq:prf:congfam-quotal:Krst} implies $\varphi\psi=\varphi\psi^x$, so we obtain $K^{\mathcal Q}_\varphi\subset\operatorname{RSt}^G_\varphi$.
In addition, for $y\in G_n$ and $x\in K^{\mathcal Q}_\varphi$, we have
\[
\begin{multlined}
q(\varphi^y,\varphi^\ast(y)x\varphi^\ast(y)^{-1}) \\
= q(\mathrm{id},y)q(\varphi,x)q(\mathrm{id},\varphi^\ast(y)^{-1})
= q(\mathrm{id},y) q(\varphi,\varphi^\ast(y)^{-1})
= \varphi^y\ ,
\end{multlined}
\]
which implies
\[
\varphi^\ast(y)\cdot K^{\mathcal Q}_\varphi\cdot\varphi^\ast(y)^{-1}
= K^{\mathcal Q}_{\varphi^y}\ .
\]
Therefore, $K^{\mathcal Q}$ is a congruence family.

It is straightforward that $\mathcal Q\mapsto K^{\mathcal Q}$ is an inverse assignment to $K\mapsto\mathcal Q_K$.
Furthermore, if we have an inclusion $K\subset K'$ between congruence families, i.e. $K_\varphi\subset K'_\varphi$ for each morphism $\varphi$ in $\nabla$, then there is a functor $\mathcal Q_K\to\mathcal Q_{K'}$ which makes the following diagram commutes:
\[
\xymatrix@C=1em{
  & \nabla_G \ar[dl] \ar[dr] & \\
  \mathcal Q_K \ar[rr] && \mathcal Q_{K'} }
\]
In other words, the assignment $K\mapsto K'$ respects the orders, and this completes the proof.
\end{proof}

To end the section, we mention a closure operator on the partially ordered set of congruence families.
We introduce two classes of morphisms in the category $\nabla$.

\begin{definition}
Let $\varphi:\llangle m\rrangle\to\llangle n\rrangle\in\nabla$ be a morphism.
\begin{enumerate}[label={\upshape(\arabic*)}]
  \item $\varphi$ is said to be \emph{active} if $\varphi^{-1}\{\pm\infty\}=\varphi^{-1}\{\pm\infty\}$; equivalently if $k^\varphi_{\pm\infty}=0$.
  \item $\varphi$ is said to be \emph{inert} if the restriction $\varphi^{-1}(\langle n\rangle)\to\langle n\rangle\subset\llangle n\rrangle$ is bijective.
\end{enumerate}
\end{definition}

\begin{lemma}
\label{lem:interval-factor}
\begin{enumerate}[label={\upshape(\arabic*)}]
  \item Every morphism $\varphi$ in $\nabla$ uniquely factors as $\varphi=\mu\rho$ with $\rho$ inert and $\mu$ active.
  \item Every inert morphism admits a unique section.
\end{enumerate}
\end{lemma}

\begin{remark}
\label{rem:ortho-fact}
In view of \cref{lem:interval-factor}, it turns out that the classes $\mathsf I$ and $\mathsf A$ of inert morphisms and active morphisms respectively form an \emph{orthogonal factorization system} $(\mathsf I,\mathsf A)$ on $\nabla$; i.e. the following two conditions are satisfied:
\begin{enumerate}[label={\upshape(\roman*)}]
  \item the classes $\mathsf I$ and $\mathsf A$ are closed under compositions and contains all the isomorphisms;
  \item every morphisms in $\nabla$ is of the form $\mu\rho$ with $\rho\in\mathsf I$ and $\mu\in\mathsf A$;
  \item for every commutative square
\[
\xymatrix{
  \llangle k\rrangle \ar[r]^\varphi \ar[d]_\rho & \llangle m\rrangle \ar[d]^\mu \\
  \llangle l\rrangle \ar[r]^\psi \ar@{-->}[ur]^{\exists!\chi} & \llangle n\rrangle}
\]
with $\rho\in\mathsf I$ and $\mu\in\mathsf A$, there is a \emph{unique} diagonal $\chi$ so that $\chi\rho=\varphi$ and $\mu\chi=\psi$.
\end{enumerate}
Note that the notion was first introduced by Freyd and Kelly in \cite{FreydKelly1972} under the name \emph{factorization}.
We instead use the name above to emphasize the \emph{unique} lifting property and to distinguish it from \emph{weak factorization systems}.
\end{remark}

Let $G$ be a crossed interval group, and let $K$ be a congruence family on $G$.
Using \cref{lem:interval-factor}, we construct another congruence family $\overbar K$ as follows: for each active morphism $\mu$ in $\nabla$, we put $\overbar K_\mu=K_\mu$.
For a general morphism $\varphi$ in $\nabla$, we set $\overbar K_\varphi\subset\operatorname{RSt}^G_\varphi$ to consist $x\in\operatorname{RSt}^G_\varphi$ such that, for every morphism $\psi$ with $\varphi\psi$ making sense and active, $\psi^\ast(x)\in\operatorname{RSt}^G_{\varphi\psi}$ belongs to $\overbar K_{\varphi\psi}$.
Thanks to \cref{lem:interval-factor}, this extension does not change $\overbar K_\mu$ for active $\mu$.

\begin{lemma}
\label{lem:congfam-bar}
In the situation above, the family $\overbar K=\{\overbar K_\varphi\}_\varphi$ forms a congruence family on $G$.
\end{lemma}
\begin{proof}
We first verify $\overbar K_\varphi\subset\operatorname{RSt}^G_\varphi$ forms a subgroup.
Suppose $\psi$ is a morphism in $\nabla$ with $\varphi\psi$ making sense and active.
For two elements $x,y\in\overbar K_\varphi$, we have
\begin{equation}
\label{eq:prf:congfam-bar:subg}
\psi^\ast(x^{-1}y)
= (\psi^{x^{-1}y})^\ast(x)^{-1}\psi^\ast(y)\ .
\end{equation}
Since $x,y\in\operatorname{RSt}^G_\varphi$, the composition $\varphi\psi^{x^{-1}y}$ equals to $\varphi\psi$, which is active.
Hence, both terms in the right hand side of \eqref{eq:prf:congfam-bar:subg} belongs to $K_\varphi$, which implies $x^{-1}y\in\overbar K_\varphi$.

Using \cref{lem:interval-factor}, one can verify $\overbar K=\{\overbar K_\varphi\}_\varphi$ forms a crossed $\operatorname{opTw}(\nabla)$-subgroup of $\operatorname{RSt}^G$.
It remains to verify the formula \eqref{eq:congfam-conj}.
Clearly, the inclusion in one direction will suffice, so we show
\begin{equation}
\label{eq:prf:congfam-bar:inc}
\varphi^\ast(y)\cdot\overbar K_\varphi\cdot\varphi^\ast(y)^{-1}
\subset \overbar K_{\varphi^y}
\end{equation}
for each $\varphi:\llangle m\rrangle\to\llangle n\rrangle\in\nabla$ and $y\in G_n$.
Suppose $\psi:\llangle l\rrangle\to\llangle m\rrangle\in\nabla$ is a morphism with $\varphi^y\psi$ active.
For $x\in\overbar K_\varphi$, we have
\begin{equation}
\label{eq:prf:congfam-bar:conj}
\begin{split}
\psi^\ast(\varphi^\ast(y)x\varphi^\ast(y)^{-1})
&= (\varphi\psi^{x\varphi^\ast(y)^{-1}})^\ast(y)\cdot (\psi^{\varphi^\ast(y)^{-1}})^\ast(x) \cdot \psi^\ast(\varphi^\ast(y)^{-1}) \\
&= (\varphi\psi^{\varphi^\ast(y)^{-1}})^\ast(y)\cdot (\psi^{\varphi^\ast(y)^{-1}})^\ast(x)\cdot (\varphi^y\psi)^\ast(y^{-1}) \\
&= (\varphi^y\psi)^\ast(y^{-1})^{-1}\cdot (\psi^{\varphi^\ast(y)^{-1}})^\ast(x)\cdot (\varphi^y\psi)^\ast(y^{-1})
\end{split}
\end{equation}
Note that, since $\varphi\psi^{\varphi^\ast(y)^{-1}}=(\varphi^y\psi)^{y^{-1}}$ is active, the middle term in the right hand side of \eqref{eq:prf:congfam-bar:conj} belongs to $K_{(\varphi^y\psi)^{y^{-1}}}$.
Using the formula \eqref{eq:congfam-conj} for the congruence family $K$, one gets
\[
(\varphi^y\psi)^\ast(y^{-1})^{-1}\cdot K_{(\varphi^y\psi)^{y^{-1}}}\cdot (\varphi^y\psi)^\ast(y^{-1})
= K_{\varphi^y\psi}\ .
\]
This implies that $\psi^\ast(\varphi^\ast(y)x\varphi^\ast(y)^{-1})\in K_{\varphi^y\psi}$, and the inclusion \eqref{eq:prf:congfam-bar:inc} follows.
\end{proof}

\begin{lemma}
\label{lem:congfam-barclop}
Let $G$ be a crossed interval group.
Then, the assignment $K\mapsto\overbar K$ defines a closure operator on the ordered set of congruence families on $G$.
\end{lemma}
\begin{proof}
The assignment $K\mapsto\overbar K$ clearly respects the inclusions.
Moreover, since $\overbar K_\mu=K_\mu$ for active morphisms $\mu$, we also have $\overbar{\overbar K}_\varphi=\overbar K_\varphi$ for every morphism $\varphi$.
On the other hand, we have $K_\varphi\subset\operatorname{RSt}^G_\varphi$, and for every morphism $\psi$ with $\varphi\psi$ making sense and active, $\psi^\ast(K_\varphi)\subset K_{\varphi\psi}$.
This implies $K_\varphi\subset\overbar K_\varphi$.
Thus, we obtain the result.
\end{proof}

\begin{definition}
A congruence family $K$ on a crossed interval group $G$ is said to be \emph{proper} if it is closed with respect to the closure operator $\overbar{(\blank)}$ defined above in the ordered set of congruence families on $G$; i.e. $\overbar K=K$.
\end{definition}

Every crossed interval group $G$ admits the minimum proper congruence family; namely the closure $\overbar{\operatorname{Triv}}$ of the trivial congruence family given in \cref{ex:congfam-triv}.
We write $\operatorname{Inr}^G:=\overbar{\operatorname{Triv}}$.
Hence, every proper congruence family on $G$ contains $\operatorname{Inr}^G$.
Moreover, it satisfies the following properties.

\begin{lemma}
\label{lem:conginr}
Let $G$ be a crossed interval group.
\begin{enumerate}[label={\upshape(\arabic*)}]
  \item\label{req:conginr:stab} If a composition $\varphi\psi$ in $\nabla$ is active, then the action of $\operatorname{Inr}^G_\varphi$ stabilizes $\psi$.
  \item\label{req:conginr:normal} For every morphism $\varphi$ in $\nabla$, the subgroup $\operatorname{Inr}^G_\varphi\subset\operatorname{RSt}^G_\varphi$ is normal.
  \item\label{req:conginr:inert} Let $K$ be a proper congruence family on $G$.
Then, for an active morphism $\mu$ and for an inert morphism $\rho$ with $\varphi\rho$ making sense, the composition
\begin{equation}
\label{eq:conginr:indinert}
K_\mu
\xrightarrow{\rho^\ast} K_{\mu\rho}
\twoheadrightarrow K_{\mu\rho}/\operatorname{Inr}^G_{\mu\rho}
\end{equation}
is bijective.
\end{enumerate}
\end{lemma}
\begin{proof}
We first show \ref{req:conginr:stab}.
Take the factorization $\varphi=\mu\rho$ with $\rho$ inert and $\mu$ active, and let $\delta$ be the unique section of $\rho$.
Notice that, in view of \cref{lem:interval-factor}, $\delta$ is characterized by the following two properties:
\begin{enumerate}[label={\upshape(\roman*)}]
  \item $\varphi\delta$ is active;
  \item every morphism $\psi$ with $\varphi\psi$ making sense and active uniquely factors as $\psi=\delta\psi'$ for a morphism $\psi'$.
\end{enumerate}
It follows that $\delta$ is fixed by the action of $\operatorname{RSt}^G_\varphi$.
Moreover, if $\varphi\psi$ is active, the property above implies there is a morphism $\psi'$ with $\psi=\delta\psi$.
Then, for each $x\in\operatorname{Inr}^G_\varphi$, we have
\[
\psi^x
= (\delta\psi')^x
= \delta^x\psi'^{\delta^\ast(x)}
= \delta\psi'
= \psi\ ,
\]
so that \ref{req:conginr:stab} follows.

Next, suppose $u\in\operatorname{Inr}^G_\varphi$ and $x\in\operatorname{RSt}^G_\varphi$.
For every morphism $\psi$ with $\varphi\psi$ making sense and active, we have
\[
\psi^\ast(xux^{-1})
= (\psi^{ux^{-1}})^\ast(x)(\psi^{x^{-1}})^\ast(u)\psi^\ast(x^{-1})
= (\psi^{ux^{-1}})^\ast(x)
\]
Note that $\varphi\psi^{x^{-1}}=\varphi\psi$ is active, so $(\psi^{x^{-1}})^\ast(u)$ is the unit.
Moreover, the part \ref{req:conginr:stab} implies $\psi^{ux^{-1}}=\psi^{x^{-1}}$.
It follows that $\psi^\ast(xux^{-1})$ vanishes, and we obtain \ref{req:conginr:normal}.

Finally, we show \ref{req:conginr:inert}.
Let $\delta:\llangle m\rrangle\to\llangle l\rrangle\in\nabla$ be the unique section of $\rho$.
We assert that the map $\delta^\ast:K_{\mu\rho}\to K_\mu$ induces the inverse of \eqref{eq:conginr:indinert}.
Indeed, for each $u\in\operatorname{Inv}^G_{\mu\rho}$, the definition of $\operatorname{Inv}^G_{\mu\rho}$ and the part \ref{req:conginr:stab} imply $\delta^\ast(u)=e$ and $\delta^u=\delta$.
Hence, for every $x\in K_{\mu\rho}$, $\delta^\ast(xu)=\delta^\ast(x)$.
In other words, $\delta^\ast$ is $\operatorname{Inv}^G_{\mu\rho}$-invariant so that it induces a map
\[
\delta^\dagger:K_{\mu\rho}/\operatorname{Inr}^G_{\mu\rho}\to K_\mu\ .
\]
Since $\delta$ is a section of $\rho$, the map is clearly a left inverse of the map \eqref{eq:conginr:indinert}.
To see it is also a right inverse, it is enough to see that, for each $x\in K_{\mu\rho}$, we have $\rho^\ast(\delta^\ast(x))x^{-1}\in\operatorname{Inr}^G_{\mu\rho}$.
Note that, by virtue of the characterization of $\delta$ above, this holds if and only if the map $\delta^\ast$ vanishes the element.
We have
\[
\delta^\ast(\rho^\ast(\delta^\ast(x))\cdot x^{-1})
= (\delta\rho\delta^{x^{-1}})^\ast(x)\cdot \delta^\ast(x^{-1})\ .
\]
As mentioned above, $\delta$ is fixed by the action of $\operatorname{RSt}^G_{\mu\rho}$, so we have $\delta\rho\delta^{x^{-1}}=\delta$ and $\delta^\ast(x^{-1})=\delta^\ast(x)^{-1}$.
Thus, $\delta^\ast(\rho^\ast(\delta^\ast(x))\cdot x^{-1})=e$, and we conclude $\delta^\dagger$ is a right inverse of \eqref{eq:conginr:indinert}, which completes the proof.
\end{proof}

\section{Associated double categories}
\label{sec:assoc-dblcat}

In the previous section, we see that congruence families are associated with quotal categories by taking the quotients of the total category.
Our problem is, on the other hand, higher categorical so we need ``higher categorical quotients'' in some sense.
We realize them in terms of \emph{double categories}.
Recall that a double category is a category internal to the category $\mathbf{Cat}$ of small categories; i.e. a diagram
\[
\mathfrak C\xrightrightarrows{s,t}\mathcal B
\]
in the category $\mathbf{Cat}$ of small categories together with functors
\[
\gamma:\mathfrak C\times_{\mathcal B}\mathfrak C
\to\mathfrak C
\quad\text{and}\qquad
\iota:\mathcal B\to\mathfrak C
\]
satisfying the appropriate conditions of categories, where the domain of $\gamma$ is the pullback of the cospan $\mathfrak C\xrightarrow{s}\mathcal B\xleftarrow{t}\mathfrak C$.

\begin{remark}
In what follows, we will often drop the structure functors $\gamma$ and $\iota$ from the notation and just say, for example, $\mathfrak C\rightrightarrows\mathcal B$ is a double category in the case above.
\end{remark}

Let $G$ be a crossed interval group.
We construct a double category from a pair $(K,L)$ of proper congruence families satisfying the following conditions:
\begin{enumerate}[label={\upshape($\spadesuit$\arabic*)}]
  \item\label{cond:congpair:normal} for each morphism $\varphi:\llangle m\rrangle\to\llangle n\rrangle\in\nabla$, the subgroup $K_\varphi\subset G_m$ is contained in the normalizer subgroup $N(L_\varphi)$ of $L_\varphi$; i.e.
\[
N(L_\varphi)
= \left\{x\in G_m\;\middle|\;xL_\varphi x^{-1}=L_\varphi\right\}\ ;
\]
  \item\label{cond:congpair:comm} if $\psi\varphi$ is a composition of morphisms in $\nabla$, for every $u\in L_\psi$ and for each $x\in K_\varphi$,
\[
[\varphi^\ast(u)x\varphi^\ast(u)^{-1}]=[x]\in\operatorname{Inr}^G_{\psi\varphi}\backslash K_{\psi\varphi}\ .
\]
\end{enumerate}
We first define a category $\mathcal Q_{L\dblslash K}$ as follows:
\begin{itemize}
  \item the objects are the same as $\nabla$;
  \item for $m,n\in\mathbb N$, morphisms $\llangle m\rrangle\to\llangle n\rrangle$ in $\mathcal Q_{L\dblslash K}$ are triples $(\varphi,[u],[x])$ with $\varphi\in\nabla(\llangle m\rrangle,\llangle n\rrangle)$, $[u]\in\operatorname{Inr}^G_\varphi\backslash K_\varphi$, and $[x]\in L_\varphi\backslash G_m$;
  \item the composition is given by
\[
(\psi,[v],[y])\circ(\varphi,[u],[x])
= (\psi\varphi^y, [\varphi^\ast(vy)u\varphi^\ast(y)^{-1}],[\varphi^\ast(y)x])\ .
\]
\end{itemize}
Note that we have
\[
\varphi^\ast(vy)u\varphi^\ast(y)^{-1}
= (\varphi^y)^\ast(v)\cdot\varphi^\ast(y)u\varphi^\ast(y)^{-1}\ ,
\]
so the conditions on congruence families imply the element belongs to $K_{\psi\varphi^y}$.
In addition, \ref{req:conginr:normal} in \cref{lem:conginr} and the condition \ref{cond:congpair:comm} guarantee that the composition does not depend on the choice of representatives.
If we are given another morphism $(\chi,[w],[z])$ postcomposable with $(\psi,[v],[y])$, the second component of the composition
\[
\left((\chi,[w],[z])\circ(\psi,[v],[y])\right)\circ(\varphi,[u],[x])
\]
is represented by the element
\[
\begin{split}
&\varphi^\ast(\psi^\ast(wz)v\psi^\ast(z)^{-1}\psi^\ast(z)y)u\varphi^\ast(\psi^\ast(z)y)^{-1} \\
&= (\psi\varphi^{vy})^\ast(wz)\varphi^\ast(vy)u\left((\psi\varphi^y)^\ast(z)\varphi^\ast(y)\right)^{-1} \\
&= (\psi\varphi^y)^\ast(wz)\varphi^\ast(vy)u\varphi^\ast(y)^{-1}(\psi\varphi^y)^\ast(z)^{-1}\ ,
\end{split}
\]
which also represents the second component of the other composition.
Thanks to this and the associativity of morphisms in $\mathbb E_G$, one obtains the associativity of the composition in $\mathbb G_{L\dblslash K}$ so that it is actually a category.

\begin{example}
\label{ex:L//Inr}
For every proper congruence family $L$, the pair $(\operatorname{Inr}^G,L)$ satisfies the conditions \ref{cond:congpair:normal} and \ref{cond:congpair:comm}.
One can verify that there is a canonical isomorphism $\mathcal Q_{L\dblslash \operatorname{Inr}^G}\cong\mathcal Q_L$.
\end{example}

\begin{example}
\label{ex:Kec//Dec}
Let $\mathcal G$ be a group operad, so we have the congruence family $\operatorname{Dec}^{\mathcal G}$ given in \cref{ex:congfam-decomp}.
For each morphism $\varphi:\llangle m\rrangle\to\llangle n\rrangle\in\nabla$, we set $\operatorname{Kec}^{\mathcal G}_\varphi\subset\operatorname{Dec}^{\mathcal G}$ to be the kernel of the composition
\[
\operatorname{Dec}^{\mathcal G}_\varphi
\hookrightarrow \mathcal G(m)
\to\mathfrak S(m)\ .
\]
In view of \cref{ex:congfam-kernel}, the family $\operatorname{Kec}^{\mathcal G}=\{\operatorname{Kec}^{\mathcal G}_\varphi\}_\varphi$ forms a congruence family on $\mathcal G$.
Taking the closure in the sense of \cref{lem:congfam-barclop}, we obtain proper congruence families $\overbar{\operatorname{Dec}}^{\mathcal G}$ and $\overbar{\operatorname{Kec}}^{\mathcal G}$.
One can verify that the pair $(\overbar{\operatorname{Dec}}^{\mathcal G},\overbar{\operatorname{Kec}}^{\mathcal G})$ satisfies the conditions \ref{cond:congpair:normal} and \ref{cond:congpair:comm} so that they give rise to a category $\mathbb G_{\overbar{\operatorname{Kec}}^{\mathcal G}\dblslash \overbar{\operatorname{Dec}}^{\mathcal G}}$.
\end{example}

The category $\mathcal Q_{L\dblslash K}$ comes equipped with two canonical functors
\begin{equation}
\label{eq:QL/K-doublecat}
s,t:\mathcal Q_{L\dblslash K}\rightrightarrows \mathcal Q_L\ ,
\end{equation}
here $\mathcal Q_L$ is the $G$-quotal category associated with $L$, such that
\begin{itemize}
  \item they are the identities on objects;
  \item for each morphism $(\varphi,[u],[x])\in\mathcal Q_{L\dblslash K}(\llangle m\rrangle,\llangle n\rrangle)$,
\[
s(\varphi,[u],[x]) = (\varphi,[x])
\ ,\quad
t(\varphi,[u],[x]) = (\varphi,[ux])\ .
\]
\end{itemize}
Note that the assignment $t$ does not depend on the choice of representatives by virtue of the condition \ref{cond:congpair:normal}.
Then, the functorialities are easily verified.
We assert that the diagram \eqref{eq:QL/K-doublecat} canonically admits a structure of a double category: define functors $\gamma:\mathcal Q_{L\dblslash K}\times_{\mathcal Q_L}\mathcal Q_{L\dblslash K}\to\mathcal Q_{L\dblslash K}$ and $\iota:\mathcal Q_L\to\mathcal Q_{L\dblslash K}$ by
\[
\gamma\left((\varphi,[u],[u'x]),(\varphi,[u'],[x])\right)
:= (\varphi,[uu'],[x])
\ ,\quad
\iota(\varphi,[x])
:= (\varphi,[e],[x])\ ,
\]
where $e$ is the unit in the group $K_\varphi$.
These actually define functors thanks to \ref{req:conginr:normal} in \cref{lem:conginr}, and the associativity and the unitality are obvious.
We call the double category \eqref{eq:QL/K-doublecat} the \emph{double category associated to the pair $(K,L)$}.

\begin{remark}
\label{rem:QL/K-2cat}
In the case $L\subset K$, the double category $\mathcal Q_{L\dblslash K}\rightrightarrows\mathcal Q_L$ looks like a ``homotopy quotient'' of the category $\mathcal Q_L$ with respect to the congruence family $K$ in the following sense:
since the functors \eqref{eq:QL/K-doublecat} are the identities on objects, one can see the double category as a $2$-category, say $\mathbf Q_{L\dblslash K}$.
For each $m,n\in\mathbb N$, the category $\mathbf Q_{L\dblslash K}(\llangle m\rrangle,\llangle n\rrangle)$ is a groupoid whose isomorphism classes corresponds in one-to-one to morphisms $\llangle m\rrangle\to\llangle n\rrangle$ in the $G$-quotal category $\mathcal Q_K$ associated with $K$.
\end{remark}

We further take a quotient of the double category $\mathcal Q_{L\dblslash K}\rightrightarrows\mathcal Q_L$ using the following general construction.

\begin{proposition}
\label{prop:lcancel-morcong}
Let $\mathcal A$ be a category, and let $\mathsf M$ be a left cancellative class of morphisms in $\mathcal A$; i.e. if a composition $\delta\varepsilon$ belongs to $\mathsf M$, so does $\delta$.
For each $a,b\in\mathcal A$, define a relation $\sim_{\mathsf M}$ on the set $\mathcal A(a,b)$ such that $\alpha\sim_{\mathsf M}\alpha'$ if and only if, for each morphism $\beta$ in $\mathcal A$ with codomain $a$, one has $\alpha\beta=\alpha'\beta$ as soon as either of the sides belongs to $\mathsf M$.
Then, the relation $\sim_{\mathsf M}$ is a congruence on $\mathcal A$ in the sense in \expandafter\uppercase\expandafter{\romannumeral 2\relax}.8 of \cite{McL98}.
Consequently, taking the quotient of each hom-set of $\mathcal A$ by $\sim_{\mathsf M}$, one gets a quotient category $\mathcal A\to\mathcal A/\mathord\sim_{\mathsf M}$.
\end{proposition}
\begin{proof}
It is obvious that $\sim_{\mathsf M}$ is an equivalence relation on each hom-set $\mathcal A(a,b)$.
We have to show, for compositions $\alpha\beta\gamma$ and $\alpha\beta'\gamma$ with $\beta\sim_{\mathsf M}\beta'$, we have $\alpha\beta\gamma\sim_{\mathsf M}\alpha\beta\gamma$.
Suppose a composition $\alpha\beta\gamma\delta$ belongs to $\mathsf M$.
By virtue of the observation above, we have $\beta\gamma\delta\in\mathsf M$, so $\beta\sim_{\mathsf M}\beta'$ implies $\beta\gamma\delta=\beta'\gamma\delta$.
Thus, we obtain $\alpha\beta\gamma\delta=\alpha\beta'\gamma\delta$ and conclude $\alpha\beta\gamma\sim_{\mathsf M}\alpha\beta'\gamma$.
\end{proof}

\begin{remark}
\label{rem:lcancel-factsys}
Typical examples of left cancellative class of morphisms come from orthogonal factorization systems (see \cref{rem:ortho-fact}).
Suppose $(\mathsf E,\mathsf M)$ is an orthogonal factorization system on a category $\mathcal A$.
One can see that the class $\mathsf M$ is left cancellative provided every morphisms in $\mathsf E$ is an epimorphism.
Indeed, if $\varepsilon=\mu\rho$ and $\delta\mu=\nu\sigma$ are factorizations with $\mu,\nu\in\mathsf M$ and $\rho,\sigma\in\mathsf E$, then the equation $\delta\varepsilon\circ\mathrm{id}=\nu\circ\sigma\rho$ and the unique lifting property implies $\sigma\rho$ is an isomorphism.
Since $\sigma$ is an epimorphism by the assumption on $\mathsf E$, $\rho$ is an isomorphism so that $\varepsilon\in\mathsf M$.
\end{remark}

We apply \cref{prop:lcancel-morcong} to the category $\mathcal Q_{L\dblslash K}$.
To obtain a left cancellative classes on it, in view of \cref{rem:lcancel-factsys}, we construct orthogonal factorization system.
We define two classes $\mathsf I_{L\dblslash K}$ and $\mathsf A_{L\dblslash}$ of morphisms in $\mathcal Q_{L\dblslash K}$ as follows:
\[
\begin{gathered}
\mathsf I_{L\dblslash K}
:= \left\{(\varphi,[u],[x])\;\middle|\;\text{$\varphi\,$:~inert}\right\}\ ,
\\
\mathsf A_{L\dblslash K}
:= \left\{(\varphi,[u],[x])\;\middle|\;\text{$\varphi\,$:~active}\right\}\ .
\end{gathered}
\]
We assert that $(\mathsf I_{L\dblslash K},\mathsf A_{L\dblslash K})$ forms an orthogonal factorization system on $\mathbb Q_{L\dblslash K}$.
In fact, if we have a composition $\mu\rho:\llangle m\rrangle\to\llangle n\rrangle\in\nabla$ with $\mu$ active and $\rho$ inert, then for each $u\in K_{\mu\rho}$ and $x\in G_m$, one can use \ref{req:conginr:inert} in \cref{lem:conginr} to find a unique element $\overbar u\in K_\mu$ so that
\[
(\mu\rho,[u],[x])
= (\mu,[\overbar u],[e])\circ(\rho,[e_m],[x])\ .
\]
As easily verified, the factorization is unique up to a unique isomorphism, and we conclude $(\mathsf I_{L\dblslash K},\mathsf A_{L\dblslash K})$ is an orthogonal factorization.
Since every member of $\mathsf I_{L\dblslash K}$ is a split epimorphism, \cref{lem:interval-factor} implies $\mathsf A_{L\dblslash K}$ is left cancellative.

We denote by $\widetilde{\mathcal Q}_{L\dblslash K}$ the quotient category of $\mathcal Q_{L\dblslash K}$ obtained by \cref{prop:lcancel-morcong} with the class $\mathsf A_{L\dblslash K}$.
In particular, as mentioned in \cref{ex:L//Inr}, there is an isomorphism $\mathcal Q_L\cong\mathcal Q_{L\dblslash\operatorname{Inr}^G}$.
We set $(\mathsf I_L,\mathsf A_L)$ the orthogonal factorization system corresponds to $(\mathsf I_{L\dblslash\operatorname{Inr}^G},\mathsf A_{L\dblslash\operatorname{Inr}^G})$, so we obtain a quotient category $\mathcal Q_L\twoheadrightarrow\widetilde{\mathcal Q}_L$, which corresponding to $\widetilde{\mathcal Q}_{L\dblslash\operatorname{Inr}^G}$ through the isomorphism.

\begin{remark}
\label{rem:Icong-crit}
We have a convenient criterion for the congruence $\sim_{\mathsf A_{L\dblslash K}}$.
Suppose $\varphi,\varphi':\llangle m\rrangle\to\llangle n\rrangle\in\nabla$, $x,x'\in G_m$, $u\in K_\varphi$, and $u'\in K_{\varphi'}$.
Then, we have $(\varphi,[u],[x])\sim_{\mathsf A_{L\dblslash K}}(\varphi',[u'],[x'])$ if and only if for every $\psi:\llangle l\rrangle\to\llangle m\rrangle$ with either $\varphi\psi^x$ or $\varphi'\psi^{x'}$ active, the following equations hold:
\[
\begin{gathered}
\varphi\psi^x = \varphi'\psi^{x'} \in \nabla(\llangle l\rrangle,\llangle n\rrangle) \\
[\psi^\ast(x)]=[\psi^\ast(x')]\in L_{\varphi\psi^x}\backslash G_l \\
[(\psi^x)^\ast(u)]=[(\psi^{x'})^\ast(u')]\in\operatorname{Inr}^G_{\varphi\psi^x}\!\backslash K_{\varphi\psi^x}\ .
\end{gathered}
\]
Furthermore, let $\varphi=\mu\rho$ and $\varphi'=\mu'\rho'$ be the factorization with $\mu,\mu'$ active and $\rho,\rho'$ inert, and say $\delta$ and $\delta'$ are the unique sections of $\rho$ and $\rho'$ respectively.
Then, it turns out that we only have to test the conditions above in the cases $\psi=\delta^{x^{-1}}$ and $\psi=\delta'^{x'^{-1}}$.
\end{remark}

\begin{example}
\label{ex:Icong-active}
Let $G$ be a crossed interval group, and let $(K,L)$ be a pair of proper congruence families satisfying \ref{cond:congpair:normal} and \ref{cond:congpair:comm}.
Then, for an active morphism $\mu:\llangle m\rrangle\to\llangle n\rrangle$, we have $(\mu,[u],[x])\sim_{\mathsf A_{L\dblslash K}}(\mu',[u'],[x'])$ for two morphisms in $\mathcal Q_{L\dblslash K}$ if and only if they are in fact equal.
Indeed, the composition $\mu\circ\mathrm{id}_{\llangle m\rrangle}^x$ equals to $\mu$ itself and is active, so we have
\[
\begin{gathered}
\mu
= \mu\circ\mathrm{id}_{\llangle m\rrangle}^x
= \mu'\circ\mathrm{id}_{\llangle m\rrangle}^{x'}
= \mu'
\\
[x]
= [\mathrm{id}_{\llangle m\rrangle}^\ast(x)]
= [\mathrm{id}_{\llangle m\rrangle}^\ast(x')]
= [x']
\\
[u]
= [(\mathrm{id}^x_{\llangle m\rrangle})^\ast(u)]
= [(\mathrm{id}^{x'}_{\llangle m\rrangle})^\ast(u')]
= [u']\ .
\end{gathered}
\]
\end{example}

\begin{example}
\label{ex:Icong-inert}
Let $G$ be a crossed interval group and $L$ a proper congruence family.
For $1\le i\le n$, define $\rho_i:\llangle n\rrangle\to\llangle 1\rrangle$ to be the inert morphism such that
\[
\rho_i(j) =
\begin{cases}
-\infty & j<i\ , \\
1 & j=i\ , \\
\infty & j>i\ ,
\end{cases}
\]
and put $\delta_i$ the unique section of $\rho_i$.
Then, for each $x\in G_n$, we have
\begin{equation}
\label{eq:Icong-inert}
(\rho_{x(i)},[x])
\sim_{\mathsf A_L} (\rho_i,[\rho_i^\ast\delta_i^\ast(x)])
\in\mathcal Q_L(\llangle n\rrangle,\llangle 1\rrangle)\ .
\end{equation}
Indeed, since $\rho_i^\ast\delta_i^\ast(x)\in G_n$ acts trivially on active morphisms, we have
\[
\begin{gathered}
\rho_{x(i)}\delta_{x(i)}
= \mathrm{id}_{\llangle 1\rrangle}
= \rho_i\delta_i
= \rho_i\delta_{x(i)}^{x^{-1}}
= \rho_i\delta_{x(i)}^{\rho_i^\ast(x)\delta_i^\ast(x)x^{-1}}
\\
(\delta^{x^{-1}}_{x(i)})^\ast(x)
= \delta_i^\ast(x)
= \delta_i^\ast\left((\rho_i\delta_i)^\ast(x)\right)\ .
\end{gathered}
\]
Hence, \eqref{eq:Icong-inert} follows from the argument in \cref{rem:Icong-crit}.
\end{example}

\begin{lemma}
\label{lem:QQ-quot}
Let $G$ be a crossed interval group, and let $(K,L)$ be a pair of proper congruence families on $G$ satisfying \ref{cond:congpair:normal} and \ref{cond:congpair:comm}.
Then, the quotient functors $\mathcal Q_{L\dblslash K}\to\widetilde{\mathcal Q}_{L\dblslash K}$ and $\mathcal Q_L\to\widetilde{\mathcal Q}_L$ derive functors $s,t:\widetilde{\mathcal Q}_{L\dblslash K}\rightrightarrows\widetilde{\mathcal Q}_L$ from the functors \eqref{eq:QL/K-doublecat} so that the diagram below is commutative:
\begin{equation}
\label{eq:QQ-quot:pbsq}
\vcenter{
  \xymatrix{
    \mathcal Q_{L\dblslash K} \ar[r]^s \ar[d] \ar@{}[dr]|(.4)\pbcorner & \mathcal Q_L \ar[d] & \mathcal Q_{L\dblslash K} \ar[l]_t \ar@{}[dl]|(.4)\bpcorner \ar[d] \\
    \widetilde{\mathcal Q}_{L\dblslash K} \ar[r]^s & \widetilde{\mathcal Q}_L & \widetilde{\mathcal Q}_{L\dblslash K} \ar[l]_t }}
\quad.
\end{equation}
Moreover, each square in \eqref{eq:QQ-quot:pbsq} is a pullback of categories.
\end{lemma}
\begin{proof}
The first statement is straightforward.
To see the last, we show that when we fix a morphism $\varphi:\llangle m\rrangle\to\llangle n\rrangle\in\nabla$ and an element $x\in G_m$, for two elements $u,u'\in K_\varphi$, the following three are all equivalent:
\begin{enumerate}[label={\upshape(\alph*)}]
  \item\label{cond:prf:QQ-quot:spb} $(\varphi,[u],[x])\sim_{\mathsf A_{L\dblslash K}} (\varphi,[u'],[x])\in\mathcal Q_{L\dblslash K}(\llangle m\rrangle,\llangle n\rrangle)$;
  \item\label{cond:prf:QQ-quot:tpb} $(\varphi,[u],[u^{-1}x])\sim_{\mathsf A_{L\dblslash K}} (\varphi,[u'],[u'^{-1}x])\in\mathcal Q_{L\dblslash K}(\llangle m\rrangle,\llangle n\rrangle)$;
  \item\label{cond:prf:QQ-quot:inr} $u^{-1}u'\in\operatorname{Inr}^G_\varphi$.
\end{enumerate}
Note the equivalence of \ref{cond:prf:QQ-quot:spb} and \ref{cond:prf:QQ-quot:inr} implies the left square in \eqref{eq:QQ-quot:pbsq} is a pullback while the equivalence of \ref{cond:prf:QQ-quot:tpb} and \ref{cond:prf:QQ-quot:inr} implies the other.

Let $\varphi=\mu\rho$ be the factorization with $\mu$ active and $\rho$ inert.
In view of \cref{rem:Icong-crit}, the condition \ref{cond:prf:QQ-quot:spb} is satisfied if and only if $\delta^\ast(u)=\delta^\ast(u')$ since $\operatorname{Inr}^G_{\varphi\delta}=\operatorname{Inr}^G_\mu$ is trivial.
The latter is equivalent to \ref{cond:prf:QQ-quot:inr} in view of \ref{req:conginr:inert} in \cref{lem:conginr}.
The same argument also completely goes well for the condition \ref{cond:prf:QQ-quot:tpb}, and this completes the proof.
\end{proof}

\begin{proposition}
\label{prop:tildeQL/K-doublecat}
Let $G$ be a crossed interval group, and let $K$ be a proper congruence family on $G$.
Then, the diagram
\begin{equation}
\label{eq:tildeQL/K-doublecat}
s,t:\widetilde{\mathcal Q}_{L\dblslash K}\rightrightarrows\widetilde{\mathcal Q}_L
\end{equation}
admits the structure of a double category inherited from \eqref{eq:QL/K-doublecat}.
\end{proposition}
\begin{proof}
It suffices to see that the vertical composition functor
\[
\gamma:\mathcal Q_{L\dblslash K}\times_{\mathcal Q_L}\mathcal Q_{L\dblslash K}\to\mathcal Q_{L\dblslash K}
\]
induces a functor
\[
\gamma:\widetilde{\mathcal Q}_{L\dblslash K}\times_{\widetilde{\mathcal Q}_L}\widetilde{\mathcal Q}_{L\dblslash K}\to\widetilde{\mathcal Q}_{L\dblslash K}\ .
\]
Note that, in view of the pullback squares in \eqref{eq:QQ-quot:pbsq}, we have a canonical isomorphism
\[
\mathcal Q_{L\dblslash K}\times_{\mathcal Q_L}\mathcal Q_{L\dblslash K}
\cong\mathcal Q_L\times_{\widetilde{\mathcal Q}_L}\left(\widetilde{\mathcal Q}_{L\dblslash K}\times_{\widetilde{\mathcal Q}_L}\widetilde{\mathcal Q}_{L\dblslash K}\right)
\]
of categories, where the right hand side is the limit of the diagram
\[
\vcenter{
  \xymatrix{
    & \mathcal Q_L \ar[d] & \\
    \widetilde{\mathcal Q}_{L\dblslash K} \ar[r]^s & \widetilde{\mathcal Q}_L & \widetilde{\mathcal Q}_{L\dblslash K} \ar[l]_t }}
\quad.
\]
Thus, we have to show the composition
\[
\widetilde\gamma:
\mathcal Q_L\times_{\widetilde{\mathcal Q}_L}\left(\widetilde{\mathcal Q}_{L\dblslash K}\times_{\widetilde{\mathcal Q}_L}\widetilde{\mathcal Q}_{L\dblslash K}\right)
\cong \mathcal Q_{L\dblslash K}\times_{\mathcal Q_L}\mathcal Q_{L\dblslash K}
\xrightarrow{\gamma} \mathcal Q_{L\dblslash K}
\to\widetilde{\mathcal Q}_{L\dblslash K}
\]
depends, with respect to the first parameter, only on the images under the functor $\mathcal Q_L\to\widetilde{\mathcal Q}_L$ .
Since $\widetilde\gamma$ is clearly the identity on objects, we concentrate on morphisms.
Note that, for a morphism $(\varphi,[x])$ of $\mathcal Q_L$ and for morphisms $\bsvarphi$ and $\bspsi$ of $\widetilde{\mathcal Q}_{L\dblslash K}$ with
\begin{equation}
\label{eq:prf:tildeQL/K-doublecat:comp}
t(\bsvarphi)=s(\bspsi)=[\varphi,x]\in\widetilde{\mathcal Q}_L\ ,
\end{equation}
the image $\widetilde\gamma((\varphi,[x]),\bspsi,\bsvarphi)$ is given as follows:
by virtue of \cref{lem:QQ-quot}, \eqref{eq:prf:tildeQL/K-doublecat:comp} implies there are elements $u,v\in K_\varphi$ so that
\[
\bsvarphi = [\varphi,u,u^{-1}x]
\ ,\quad \bspsi = [\varphi,v,x]\ .
\]
Then we have
\[
\widetilde\gamma((\varphi,[x]),\bspsi,\bsvarphi)
= [\varphi,vu,u^{-1}x]\ .
\]
Now, suppose $(\varphi,[x])\sim_{\mathsf A_L}(\varphi',[x'])$, and take $u',v'\in K_\varphi$ so that
\[
\bsvarphi = [\varphi',u',u'^{-1}x']
\ ,\quad \bspsi = [\varphi',v',x']\ .
\]
We show the congruence
\begin{equation}
\label{eq:prf:tildeQL/K-doublecat:eq}
(\varphi,[vu],[u^{-1}x])
\sim_{\mathsf A_{L\dblslash K}} (\varphi',[v'u'],[u'^{-1}x'])\ .
\end{equation}
Let $\psi$ be a morphism precomposable with $\varphi$ such that either $\varphi\psi^{u^{-1}x}$ or $\varphi'\psi^{u'^{-1}x'}$ is active.
Since $u$ and $u'$ are right stabilizers of $\varphi$, this implies either $\varphi\psi^x$ or $\varphi'\psi^{x'}$ is also active.
Then, in view of \cref{rem:Icong-crit}, the congruences
\[
(\varphi,[u],[u^{-1}x])
\sim_{\mathsf A_{L\dblslash K}} (\varphi',[u'],[u'^{-1}x'])
\ ,\quad
(\varphi,[v],[x])
\sim_{\mathsf A_{L\dblslash K}} (\varphi',[v'],[x'])
\]
imply
\[
\varphi\psi^{u^{-1}x} = \varphi\psi^{u'^{-1}x'}
\ ,\quad [\psi^\ast(u^{-1}x)]=[\psi^\ast(u'^{-1}x')]
\]
and
\[
\begin{split}
[(\psi^{u^{-1}x})^\ast(vu)]
&= [(\psi^x)^\ast(v)(\psi^{u^{-1}x})^\ast(u)] \\
&= [(\psi^x)^\ast(v')(\psi^{u'^{-1}x'})^\ast(u')] \\
&= [(\psi^{u'^{-1}x'})^\ast(v'u')]\ .
\end{split}
\]
Thus, \eqref{eq:prf:tildeQL/K-doublecat:eq} follows, and we obtain $\widetilde\gamma((\varphi,[x]),\bspsi,\bsvarphi)= \widetilde\gamma((\varphi',[x']),\bspsi,\bsvarphi)$ as required.
\end{proof}

\section{Internal presheaves over the associated double categories}
\label{sec:intpresh-assocdbl}

We constructed double categories $\mathcal Q_{L\dblslash K}\rightrightarrows\mathcal Q_L$ and $\widetilde{\mathcal Q}_{L\dblslash K}\rightrightarrows\widetilde{\mathcal Q}_L$ for a sort of pairs $(K,L)$ of proper congruence families on crossed interval groups $G$.
Recall that, as they are internal categories in the category $\mathbf{Cat}$ of small categories, we can consider the following notion on them.

\begin{definition}[cf. \cite{Johnstone1977}, Definition 2.14]
Let $\mathfrak C\xrightrightarrows{s,t}\mathcal B$ be a double category.
Then an \emph{internal presheaf} over it consists of the data
\begin{itemize}
  \item a category $(\mathcal X\to\mathcal B)\in\mathbf{Cat}^{/\mathcal B}$ over $\mathcal B$;
  \item a functor
\[
\mathpzc A_{\mathcal X}:\mathcal X\times_{\mathcal B}\mathfrak C\to\mathcal X
\]
in $\mathbf{Cat}^{/\mathcal B}$, where the domain is the pullback of the cospan $\mathcal X\to\mathcal B\xleftarrow{t}\mathfrak C$ and seen as a category over $\mathcal B$ with the composition
\[
\mathcal X\times_{\mathcal B}\mathfrak C
\xrightarrow{\mathrm{proj.}}\mathfrak C
\xrightarrow{s} \mathcal B\ ;
\]
\end{itemize}
such that the diagrams below are commutative:
\[
\vcenter{
  \xymatrix@C=3.5em{
    \mathcal X\times_{\mathcal B}\mathfrak C\times_{\mathcal B}\mathfrak C \ar[r]^-{\mathpzc A_{\mathcal X}\times\mathrm{Id}_{\mathfrak C}} \ar[d]_{\mathrm{Id}_{\mathcal X}\times\gamma_{\mathfrak C}} & \mathcal X\times_{\mathcal B}\mathfrak C \ar[d]^{\mathpzc A_{\mathcal X}} \\
    \mathcal X\times_{\mathcal B}\mathfrak C \ar[r]^-{\mathpzc A_{\mathcal X}} & \mathcal X }}
\quad,\quad
\vcenter{
  \xymatrix{
    \mathcal X\times_{\mathcal B}\mathcal B \ar[r]^{\mathrm{Id}_{\mathcal X}\times\iota} \ar@{=}[dr] & \mathcal X\times_{\mathcal B}\mathfrak C \ar[d]^{\mathpzc A_{\mathcal X}} \\
    & \mathcal X }}
\quad.
\]
\end{definition}

A double category $\mathfrak C\rightrightarrows\mathcal B$ gives rise to a $2$-monad
\[
\mathbf{Cat}^{/\mathcal B}\to\mathbf{Cat}^{/\mathcal B}
\ ;\quad \mathcal X\mapsto \mathcal X\times_{\mathcal B}\mathfrak C\ .
\]
Actually, internal presheaves over $\mathfrak C\rightrightarrows\mathcal B$ are precisely (strict) $2$-algebras on it.
In particular, they form a $2$-category, which we denote by $\mathbf{PSh}(\mathfrak C\rightrightarrows\mathcal B)$.
The $2$-morphisms in $\mathbf{PSh}(\mathfrak C\rightrightarrows\mathcal B)$ are, by definition, natural transformations $\alpha:H\to K:\mathcal X\to\mathcal Y$ over $\mathcal B$ such that the following two horizontal compositions coincide:
\[
\begin{gathered}
\vcenter{
  \xymatrix@C=6em{
    \mathcal X\times_{\mathcal B}\mathfrak C \ar@/^.7pc/[r]^{H\times\mathrm{Id}}_{}="a" \ar@/_.7pc/[r]_{K\times\mathrm{Id}}^{}="b" & \mathcal Y\times_{\mathcal B}\mathfrak C \ar[r]^-{\mathpzc A_{\mathcal Y}} & \mathcal Y \ar@{=>}^{\alpha\times\mathrm{id}} "a";"b" }}\quad,
\\
\vcenter{
  \xymatrix@C=6em{
    \mathcal X\times_{\mathcal B}\mathfrak C \ar[r]^-{\mathpzc A_{\mathcal X}} & \mathcal X \ar@/^.7pc/[r]^H_{}="a" \ar@/_.7pc/[r]_K^{}="b" & \mathcal Y \ar@{=>}^{\alpha} "a";"b" }}\quad.
\end{gathered}
\]
One can easily prove the following result.

\begin{lemma}
\label{lem:inrpresh-locfull}
Let $\mathfrak C\rightrightarrows\mathcal B$ be a double category such that the functors $s,t:\mathfrak C\rightrightarrows\mathcal B$ are the identity on objects.
Then, the forgetful functor
\[
\mathbf{PSh}(\mathfrak C\rightrightarrows\mathcal B)\to\mathbf{Cat}^{/\mathcal B}
\]
is locally fully faithful; i.e. for internal presheaves $\mathcal X$ and $\mathcal Y$, the functor
\[
\mathbf{PSh}(\mathfrak C\rightrightarrows\mathcal B)(\mathcal X,\mathcal Y)\to\mathbf{Cat}^{/\mathcal B}(\mathcal X,\mathcal Y)
\]
is fully faithful.
\end{lemma}

Suppose we are given a group operad $\mathcal G$ and a $\mathcal G$-symmetric multicategory $\mathcal M$.
In view of categories of operators of $\mathcal M$ with regard to $\mathcal G$, the pair $(\overbar{\operatorname{Dec}}^{\mathcal G},\overbar{\operatorname{Kec}}^{\mathcal G})$ plays the fundamental role.
We will write
\[
\mathbb G_{\mathcal G}
:= \mathcal Q_{\overbar{\operatorname{Kec}}^{\mathcal G}\dblslash\overbar{\operatorname{Dec}}^{\mathcal G}}
\ ,\quad
\mathbb E_{\mathcal G}
:= \mathcal Q_{\overbar{\operatorname{Kec}}^{\mathcal G}}
\ ,\quad
\widetilde{\mathbb G}_{\mathcal G}
:= \widetilde{\mathcal Q}_{\overbar{\operatorname{Kec}}^{\mathcal G}\dblslash\overbar{\operatorname{Dec}}^{\mathcal G}}
\ ,\quad
\widetilde{\mathbb E}_{\mathcal G}
:= \widetilde{\mathcal Q}_{\overbar{\operatorname{Kec}}^{\mathcal G}}\ .
\]
In this section, we see $\mathcal M$ gives rise to internal presheaves over the double categories $\mathbb G_{\mathcal G}\rightrightarrows\mathbb E_{\mathcal G}$ and $\widetilde{\mathbb G}_{\mathcal G}\rightrightarrows\widetilde{\mathbb E}_{\mathcal G}$ given in \eqref{eq:QL/K-doublecat} and \cref{prop:tildeQL/K-doublecat}.
We need some kinds of \emph{word calculus}, and the following notations are convenient.

\begin{notation}
Let $S$ be a set and $\vec a=a_1\dots a_n$ a word in $S$; i.e. $a_i\in S$.
\begin{enumerate}[label={\upshape(\arabic*)}]
  \item If $G$ is a crossed interval group, then for $x\in G_n$, we write
\[
x_\ast\vec a := a_{x^{-1}(1)}\dots a_{x^{-1}(n)}\ .
\]
Note that it coincides with the canonical left $G_n$-action on $S^{\times n}$ induced by the map $G_n\to\mathfrak W^\nabla_n\to\mathfrak S_n$.
  \item Suppose $\varphi:\llangle m\rrangle\to\llangle n\rrangle\in\nabla$ is an arbitrary morphism, and say $\varphi^{-1}\{j\}=\{i_1<\dots<i_{k^{(\varphi)}_j}\}$ for each $1\le j\le n$.
Then, we write
\[
\vec a^\varphi_j = a_{i_1}\dots a_{i_{k^{(\varphi)}_j}}\ .
\]
Hence, the concatenated word $\vec a^\varphi_1\dots\vec a^\varphi_n$ is a subword of the original $\vec a$.
\end{enumerate}
\end{notation}

\begin{lemma}
\label{lem:word-calc}
Let $S$ be a set and $\vec a=a_1\dots a_m$ a word in $S$.
\begin{enumerate}[label={\upshape(\arabic*)}]
  \item\label{req:word-calc:comp} Suppose we are given morphisms $\varphi:\llangle m\rrangle\to\llangle n\rrangle$ and $\psi:\llangle n\rrangle\to\llangle p\rrangle$ in $\nabla$, and say
\[
\psi^{-1}\{s\}
:=\{j^s_1<\dots<j^s_r\}\ .
\]
Then, we have
\[
\vec a^{\psi\varphi}_s = \vec a^\varphi_{j^s_1}\dots\vec a^\varphi_{j^s_r}\ .
\]
  \item\label{req:word-calc:perm} Let $G$ be a crossed interval group, and write the canonical map $G_n\to\mathfrak W^\nabla_n\cong(\mathfrak S_n\ltimes\mathbb Z/2\mathbb Z)\times\mathbb Z/2\mathbb Z$ in the form
\[
y\mapsto (\sigma^y;\varepsilon^y_1,\dots,\varepsilon^y_n;\theta^y)\ .
\]
Then, for every morphism $\varphi:\llangle m\rrangle\to\llangle n\rrangle\in\nabla$ and every $y\in G_n$,
\[
(\varphi^\ast(y)_\ast\vec a)^{\varphi^y}_j
= \beta^{\varepsilon^y_{y^{-1}(j)}}_\ast\vec a^\varphi_{y^{-1}(j)}\ .
\]
where $\beta$ is the order-reversing permutation.
  \item\label{req:word-calc:inert} Let $G$ be a crossed interval group.
Suppose we have two morphisms $\varphi,\varphi':\llangle m\rrangle\to\llangle n\rrangle\in\nabla$ and two elements $x,x'\in G_m$.
If two morphisms $[\varphi,x],[\varphi',x']$ in $\mathbb E_{\mathcal G}$ coincide with each other, then, for each $1\le j\le n$,
\[
(x_\ast\vec a)^\varphi_j = (x'_\ast\vec a)^{\varphi'}_j\ .
\]
\end{enumerate}
\end{lemma}
\begin{proof}
The parts \ref{req:word-calc:comp} is obvious.
On the other hand, the part \ref{req:word-calc:perm} follows from the following characterization of the permutation on $\llangle m\rrangle$ associated with $\varphi^\ast(y)$:
\begin{enumerate}[label={\upshape(\roman*)}]
  \item the square below is commutative
\[
\vcenter{
  \xymatrix{
    \llangle m\rrangle \ar[r]^\varphi \ar[d]_{\psi^\ast(y)} & \llangle n\rrangle \ar[d]^y \\
    \llangle m\rrangle \ar[r]^{\varphi^y} & \llangle n\rrangle}}
\quad;
\]
  \item for each $j\in\llangle n\rrangle$, the bijection
\[
\varphi^{-1}\{j\}\to(\varphi^y)^{-1}\{y(j)\}
\]
restricting the permutation $\varphi^\ast(y)$ either preserves or reverses the order depending on $\varepsilon^y_j$.
\end{enumerate}

We show \ref{req:word-calc:inert}.
Under the identification $\llangle k\rrangle\cong\nabla(\llangle 1\rrangle,\llangle k\rrangle)$, the data induces maps
\begin{equation}
\label{eq:prf:word-calc:indmap}
\begin{gathered}
\llangle m\rrangle
\xrightarrow{x} \llangle m\rrangle
\xrightarrow{\varphi} \llangle n\rrangle\ ,
\\
\llangle m\rrangle
\xrightarrow{x'} \llangle m\rrangle
\xrightarrow{\varphi'} \llangle n\rrangle\ .
\end{gathered}
\end{equation}
It is observed that if $[\varphi,x]=[\varphi',x']$, the two maps \eqref{eq:prf:word-calc:indmap} have the same inverse image of $\langle n\rangle=\{1,\dots,n\}\subset\llangle n\rrangle$ where they agree with each other.
Hence, the required equation $(x_\ast\vec a)^\varphi_j=(x'_\ast\vec a)^{\varphi'}_j$ follows for each $1\le j\le n$.
\end{proof}

We begin the main construction.
For a multicategory $\mathcal M$, we define a category $\mathcal M\wr\mathbb E_{\mathcal G}$ as follows:
\begin{itemize}
  \item objects are finite sequences $\vec a=a_1\dots a_n$ of objects of $\mathcal M$;
  \item for $\vec a=a_1\dots a_m$ and $\vec b=b_1\dots b_n$, the hom-set $(\mathcal M\wr\mathbb E_{\mathcal G})(\vec a,\vec b)$ consists of tuples $(\varphi;f_1,\dots,f_n;[x])$ of
\begin{itemize}
  \item $\varphi:\llangle m\rrangle\to\llangle n\rrangle\in\nabla$,
  \item $[x]\in\overbar{\operatorname{Kec}}^{\mathcal G}_\varphi\backslash\mathcal G(m)$ represented by $x\in\mathcal G(m)$, and
  \item $f_j\in\mathcal M((x_\ast\vec a)^\varphi_j;b_j)$ for each $1\le j\le n$ (see \ref{req:word-calc:inert} in \cref{lem:word-calc} and \cref{lem:QQ-quot});
\end{itemize}
  \item for morphisms $(\varphi;\vec f;[x]):a_1\dots a_l\to b_1\dots b_m$ and $(\psi;\vec g;[y]):b_1\dots b_m\to c_1\dots c_n$, the composition is given by
\[
(\psi;\vec g;[y])\circ(\varphi;\vec f;[x])
:= \left(\psi\varphi^y;\gamma(g_1;(y_\ast\vec f)^\psi_1),\dots,\gamma(g_n;(y_\ast\vec f)^\psi_n);[\varphi^\ast(y)x]\right)\ .
\]
\end{itemize}
The composition is in fact associative; indeed, suppose we have another morphism $(\chi;\vec h;[z]):\vec c\to d_1\dots d_p$, and consider the equation
\begin{equation}
\label{eq:MG-assoc}
\left((\chi;\vec h;[z])\circ (\psi;\vec g;[y])\right)\circ(\varphi;\vec f;[x])
= (\chi;\vec h;[z])\circ\left((\psi;\vec g;[y])\circ(\chi;\vec h;[z])\right)\ .
\end{equation}
In terms of the first and the third components of the tuples, the equation clearly holds.
If we put $\chi^{-1}\{s\}=\{j^s_1<\dots<j^s_r\}$, then, in view of \cref{lem:word-calc}, each term of the second component in the left hand side of \eqref{eq:MG-assoc} is given by
\[
\begin{split}
&\gamma\left(\gamma(h_s;(z_\ast\vec g)^\chi_s);(\psi^\ast(z)_\ast y_\ast\vec f)^{\chi\psi^z}_s\right) \\
&= \gamma\left(\gamma(h_s;g_{z^{-1}(j^s_1)},\dots,g_{z^{-1}(j^s_r)});\;(y_\ast\vec f)^\psi_{z^{-1}(j^s_1)}\dots(y_\ast\vec f)^\psi_{z^{-1}(j^s_r)}\right) \\
&= \gamma\left(h_s;\gamma(g_{z^{-1}(j^s_1)};(y_\ast\vec f)^\psi_{z^{-1}(j^s_1)}),\dots,\gamma(g_{z^{-1}(j^s_r)};(y_\ast\vec f)^\psi_{z^{-1}(j^s_r)})\right)
\end{split}
\]
Clearly, the last term is precisely the one appearing as a component in the left hand side of \eqref{eq:MG-assoc}, so that the composition is associative.
Note that the identity on the object $a_1\dots a_n\in\mathcal M\wr\mathbb E_{\mathcal G}$ is the tuple
\[
(\mathrm{id}_{\llangle n\rrangle};\mathrm{id}_{a_1},\dots,\mathrm{id}_{a_n};[e_n])\ .
\]

\begin{example}
\label{ex:term-total}
In the case $\mathcal M=\ast$ is the terminal operad, the resulting category $\ast\wr\mathbb E_{\mathcal G}$ is nothing but the category $\mathbb E_{\mathcal G}$ itself.
\end{example}

We extend the constructions $\mathcal M\to\mathcal M\wr\mathbb E_{\mathcal G}$ to $2$-functors.
If $F:\mathcal M\to\mathcal N$ be a $\mathcal G$-symmetric multifunctor, then we define a functor $F^{\mathcal G}:\mathcal M\wr\mathbb E_{\mathcal G}\to\mathcal N\wr\mathbb E_{\mathcal G}$ so that
\begin{itemize}
  \item for each objects $a_1\dots a_m\in\mathcal M\wr\mathbb E_{\mathcal G}$, we put
\[
F^{\mathcal G}(a_1\dots a_m) := F(a_1)\dots F(a_m)\ ;
\]
  \item for $\vec a=a_1\dots a_m,\vec b=b_1\dots b_n\in\mathcal M\wr\mathbb E_{\mathcal G}$, define
\[
\begin{array}{cccc}
  F^{\mathcal G}:&(\mathcal M\wr\mathbb E_{\mathcal G})(\vec a,\vec b) &\to& (\mathcal N\wr\mathbb E_{\mathcal G})(F^{\mathcal G}(\vec a),F^{\mathcal G}(\vec b)) \\[1ex]
    & (\varphi;f_1,\dots,f_n;[x]) &\mapsto& (\varphi;F(f_1),\dots,F(f_n);[x])\ .
\end{array}
\]
\end{itemize}
The functoriality is easily verified.
In addition, if $\alpha:F\to G:\mathcal M\to\mathcal N$ is a multinatural transformation of multinatural functors, then one can check that the morphisms
\[
\alpha^{\mathcal G}_{a_1\dots a_m}=(\mathrm{id}_{\llangle m\rrangle};\alpha_{a_1},\dots,\alpha_{a_m};e_m):F^{\mathcal G}(a_1\dots a_m)\to G^{\mathcal G}(a_1\dots a_m)
\]
for $a_1\dots a_m\in\mathcal M\wr\mathbb E_{\mathcal G}$ form a natural transformation $\alpha^{\mathcal G}:F^{\mathcal G}\to G^{\mathcal G}$.
Combining with \cref{ex:term-total}, we obtain a $2$-functor
\begin{equation}
\label{eq:2func-G1}
(\blank)\wr\mathbb E_{\mathcal G}:\mathbf{MultCat}\to\mathbf{Cat}^{/\mathbb E_{\mathcal G}}\ .
\end{equation}

We furthermore consider a quotient of the category $\mathcal M\wr\mathbb E_{\mathcal G}$.
For two morphisms
\[
(\varphi;f_1,\dots,f_n;[x]),(\varphi';f'_1,\dots,f'_n;[x'])
:a_1\dots a_m\to b_1\dots b_n\in\mathcal M\wr\mathbb E_{\mathcal G}\ ,
\]
we write $(\varphi;f_1,\dots,f_n;[x])\sim_{\mathsf A_{\mathcal G}}(\varphi';f'_1,\dots,f'_n;[x'])$ precisely when we have $[\varphi,x]=[\varphi',x']$ in $\widetilde{\mathbb E}_{\mathcal G}(\llangle m\rrangle,\llangle n\rrangle)$ and $f_j=f'_j$ for each $1\le j\le n$.
Note that, thanks to \ref{req:word-calc:inert} in \cref{lem:word-calc}, the first equation implies
\[
\mathcal M((x_\ast\vec a)^\varphi_j;b_j)
= \mathcal M((x'_\ast\vec a)^{\varphi'}_j;b_j)
\]
so that the latter comparison makes sense.
It is straightforward that the relation $\sim_{\mathsf A_{\mathcal G}}$ is a congruence on the category $\mathcal M\wr\mathbb E_{\mathcal G}$.
We denote by $\mathcal M\wr\widetilde{\mathbb E}_{\mathcal G}$ the resulting quotient category.
For each morphism $(\varphi;f_1,\dots,f_n;[x])$ of $\mathcal M\wr\mathbb E_{\mathcal G}$, we write $[\varphi;f_1,\dots,f_n;x]$ its image in $\mathcal M\wr\widetilde{\mathbb E}_{\mathcal G}$.
It is easily verified that the assignment $\mathcal M\mapsto\mathcal M\wr\widetilde{\mathbb E}_{\mathcal G}$ also extends to a $2$-functor so that the functor $\mathcal M\wr\mathbb E_{\mathcal G}\to\mathcal M\wr\widetilde{\mathbb E}_{\mathcal G}$ forms a (strict) $2$-natural transformation.
More explicitly, a multifunctor $F:\mathcal M\to\mathcal N$ induces a functor $\widetilde F^{\mathcal G}:\mathcal M\wr\widetilde{\mathbb E}_{\mathcal G}\to\mathcal N\wr\widetilde{\mathbb E}_{\mathcal G}$ such that
\begin{itemize}
  \item for each object $\vec a=a_1\dots a_m\in\mathcal M\wr\widetilde{\mathbb E}_{\mathcal G}$, $\widetilde F^{\mathcal G}(\vec a)=F(a_1)\dots F(a_m)$;
  \item as for morphisms, we have
\[
\widetilde F^{\mathcal G}([\varphi;f_1,\dots,f_n;x])
= [\varphi;F(f_1),\dots,F(f_n);x]\ .
\]
\end{itemize}
On the other hand, if $\alpha:F\to G:\mathcal M\to\mathcal N$ is a multinatural transformation, we have a natural transformation $\widetilde\alpha^{\mathcal G}:\widetilde F^{\mathcal G}\to\widetilde G^{\mathcal G}$ with
\[
\widetilde\alpha^{\mathcal G}_{a_1\dots a_m}
= [\mathrm{id}_{\llangle m\rrangle};\alpha_{a_1},\dots,\alpha_{a_m};e_m]
\]
for each $a_1\dots a_m\in\mathcal M\wr\widetilde{\mathbb E}_{\mathcal G}$.
Observing the canonical identification $\ast\wr\widetilde{\mathbb E}_{\mathcal G}\cong\widetilde{\mathbb E}_{\mathcal G}$, we obtain a $2$-functor
\begin{equation}
(\blank)\wr\widetilde{\mathbb E}_{\mathcal G}:\mathbf{MultCat}\to\mathbf{Cat}^{/\widetilde{\mathbb E}_{\mathcal G}}\ .
\end{equation}

\begin{lemma}
\label{lem:rtEG-pb}
Let $\mathcal M$ be a multicategory.
Then, for every group operad $\mathcal G$, the square below is a pullback:
\[
\vcenter{
  \xymatrix{
    \mathcal M\wr\mathbb E_{\mathcal G} \ar[r] \ar[d] \ar@{}[dr]|(.4)\pbcorner & \mathcal M\wr\widetilde{\mathbb E}_{\mathcal G} \ar[d] \\
    \mathbb E_{\mathcal G} \ar[r] & \widetilde{\mathbb E}_{\mathcal G} }}
\quad.
\]
\end{lemma}
\begin{proof}
The result is straightforward from the definition of the category $\widetilde{\mathbb E}_{\mathcal G}$.
\end{proof}

We now take $\mathcal G$-symmetries into account and see that they give rise to internal presheaf structures on $\mathcal M\wr\mathbb E_{\mathcal G}$ (resp. of $\mathcal M\wr\widetilde{\mathbb E}_{\mathcal G}$) over the double category $\mathbb G_{\mathcal G}\rightrightarrows\mathbb E_{\mathcal G}$ (resp. on $\widetilde{\mathbb G}_{\mathcal G}\rightrightarrows\widetilde{\mathbb E}_{\mathcal G}$).
To simplify the notation, we define categories $\mathcal M\wr\mathbb G_{\mathcal G}$ and $\mathcal M\wr\widetilde{\mathbb G}_{\mathcal G}$ by the pullback squares
\[
\vcenter{
  \xymatrix{
    \mathcal M\wr\mathbb G_{\mathcal G} \ar[r] \ar[d] \ar@{}[dr]|(.4)\pbcorner & \mathcal M\wr\mathbb E_{\mathcal G} \ar[d] \\
    \mathbb G_{\mathcal G} \ar[r]^t & \mathbb E_{\mathcal G} }}
\quad,\quad
\vcenter{
  \xymatrix{
    \mathcal M\wr\widetilde{\mathbb G}_{\mathcal G} \ar[r] \ar[d] \ar@{}[dr]|(.4)\pbcorner & \mathcal M\wr\widetilde{\mathbb E}_{\mathcal G} \ar[d] \\
    \widetilde{\mathbb G}_{\mathcal G} \ar[r]^t & \widetilde{\mathbb E}_{\mathcal G} }}
\qquad.
\]
Hence, the required internal presheaf structures are functors
\begin{equation}
\label{eq:ME-gamma}
\gamma:\mathcal M\wr\mathbb G_{\mathcal G}
\to \mathcal M\wr\mathbb E_{\mathcal G}
\ ,\quad
\gamma:\mathcal M\wr\widetilde{\mathbb G}_{\mathcal G}
\to \mathcal M\wr\widetilde{\mathbb E}_{\mathcal G}\ ,
\end{equation}
over $\mathbb E_{\mathcal G}$ and $\widetilde{\mathbb E}_{\mathcal G}$ respectively which satisfy appropriate conditions.
Since the latter may be induced from the first, we mainly discuss $\mathcal M\wr\mathbb E_{\mathcal G}$.
Note that the category $\mathcal M\wr\mathbb G_{\mathcal G}$ is described explicitly as follows: for each objects $\vec a=a_1\dots a_m, \vec b=b_1\dots b_n\in\mathcal M\wr\mathbb G_{\mathcal G}$, the hom-set $(\mathcal M\wr\mathbb G_{\mathcal G})(\vec a,\vec b)$ consists of tuples $(\varphi;f_1,\dots,f_n;[u],[x])$ such that
\begin{itemize}
  \item $[u]\in\operatorname{Inr}^{\mathcal G}_\varphi\backslash\overbar{\operatorname{Dec}}^{\mathcal G}_\varphi$ and $[x]\in\overbar{\operatorname{Kec}}^{\mathcal G}_\varphi\backslash G_m$;
  \item $(\varphi;f_1,\dots,f_n;[ux]):\vec a\to \vec b$ makes sense as a morphism in $\mathcal M\wr\mathbb E_{\mathcal G}$;
\end{itemize}
The composition is given by
\begin{equation}
\label{eq:MG-comp}
\begin{split}
&(\psi;g_1,\dots,g_l;[v],[y])\circ(\varphi;f_1,\dots,f_n;[u],[x]) \\
&= (\psi\varphi^y;\gamma(g_1;((vy)_\ast\vec f)^\psi_1),\dots,\gamma(g_l;(vy)_\ast\vec f)^\psi_l;[\varphi^\ast(vy)u\varphi^\ast(y)^{-1}],[\varphi^\ast(y)x])\ ,
\end{split}
\end{equation}
and the structure functor $\mathcal M\wr\mathbb G_{\mathcal G}\to\mathbb E_{\mathcal G}$ is an identity-on-object functor with
\[
\begin{array}{ccc}
  (\mathcal M\wr\mathbb G_{\mathcal G})(a_1\dots a_m, b_1\dots b_n) &\to& \mathbb E_{\mathcal G}(\llangle m\rrangle,\llangle n\rrangle) \\[1ex]
  (\varphi;f_1,\dots,f_n;[u],[x]) &\mapsto& (\varphi,[x])
\end{array}
\quad.
\]
For the construction of a internal presheaf structure, the key is a comparison of the category $\mathcal M\wr\mathbb G_{\mathcal G}$ with $(\mathcal M\rtimes\mathcal G)\wr\mathbb E_{\mathcal G}$ (see the construction in \cref{sec:grpop}).

\begin{notation}
For a morphism $\varphi:\llangle m\rrangle\to\llangle n\rrangle\in\nabla$, and for each $1\le j\le n$, suppose
\[
\varphi^{-1}\{j\}
= \left\{i_1<\dots<i_{k^{(\varphi)}_j}\right\}\ .
\]
In this case, we set
\begin{equation}
\label{eq:fundef-delta}
\delta^{(\varphi)}_j:\llangle k^{(\varphi)}_j\rrangle\to\llangle m\rrangle
\ ;\quad s \mapsto
\begin{cases}
-\infty & s=-\infty\ , \\
i_s & 1\le s\le k^{(\varphi)}_j\ , \\
\infty & s=\infty\ .
\end{cases}
\end{equation}
Hence, the composition $\varphi\delta^{(\varphi)}_j$ factors through the map $\llangle 1\rrangle\to\llangle n\rrangle$ corresponding to the element $j\in\llangle n\rrangle$.
\end{notation}

\begin{remark}
\label{rem:funchar-delta}
The morphism $\delta^{(\varphi)}_j$ defined above is characterized by the following two properties:
\begin{enumerate}[label={\upshape(\roman*)}]
  \item the composition $\varphi\delta^{(\varphi)}_j$ factors through the map $\llangle 1\rrangle\to\llangle n\rrangle$ corresponding to the element $j\in\llangle n\rrangle$;
  \item if $\psi$ is a morphism with the previous property, then there is a unique morphism $\psi'$ such that $\psi=\delta^{(\varphi)}_j\psi'$.
\end{enumerate}
\end{remark}

\begin{lemma}
\label{lem:rst-split}
Let $G$ be a crossed interval group.
Then, for every morphism $\varphi:\llangle m\rrangle\to\llangle n\rrangle$, the map
\[
\vec\delta^{(\varphi)\ast}: \operatorname{RSt}^G_\varphi\to G_{k^{(\varphi)}_1}\times\dots\times G_{k^{(\varphi)}_n}
\ ;\quad x\mapsto \left(\delta^{(\varphi)\ast}_1(x),\dots,\delta^{(\varphi)\ast}_n(x)\right)
\]
is a group homomorphism.
Moreover, its kernel contains the subgroup $\operatorname{Inr}^G_\varphi\subset\operatorname{RSt}^G_\varphi$.
\end{lemma}
\begin{proof}
To see each map $\delta^{(\varphi)\ast}_j:\operatorname{RSt}^G_\varphi\to G_{k^{(\varphi)}_j}$ is a group homomorphism, it suffices to show $\delta^{(\varphi)}$ is invariant under the left action of $\operatorname{RSt}^G_\varphi$.
This follows from the characterization in \cref{rem:funchar-delta}.
The last assertion is straightforward.
\end{proof}

In the case $G=\mathcal G$ is a group operad, if $\varphi=\mu\rho$ is the unique factorization with $\mu$ active and $\rho$ inert, then there are canonical identifications
\[
\mathcal G(k^{(\varphi)}_1)\times\dots\times\mathcal G(k^{(\varphi)}_n)
\cong \operatorname{Dec}^{\mathcal G}_\mu
= \overbar{\operatorname{Dec}}^{\mathcal G}_\mu\ .
\]
Put $\delta$ the unique section of $\rho$, then one can see the both squares in the diagram below are commutative:
\begin{equation}
\label{eq:decdelta}
\vcenter{
  \xymatrix{
    \overbar{\operatorname{Dec}}^{\mathcal G}_\varphi \ar@<.2pc>[r]^{\delta^\ast} \ar[d] & \overbar{\operatorname{Dec}}^{\mathcal G}_\mu \ar[d]^\cong \ar@<.2pc>[l]^{\rho^\ast} \\
    \operatorname{RSt}^{\mathcal G}_\varphi \ar[r]^-{\vec\delta^{(\varphi)\ast}} & \mathcal G(k^{(\varphi)}_1)\times\dots\times\mathcal G(k^{(\varphi)}_n)}}
\end{equation}
In other words, the composition of the left and the bottom arrows in \eqref{eq:decdelta} induces the inverse of the map \eqref{eq:conginr:indinert} in the case $K=\overbar{\operatorname{Dec}}^{\mathcal G}$.

\begin{theorem}
\label{theo:MG-compare}
Let $\mathcal G$ be a group operad, and let $\mathcal M$ be a multicategory.
Then, the family of maps
\[
\begin{array}{rccc}
  \Phi\mathrlap:&(\mathcal M\wr\mathbb G_{\mathcal G})(\vec a,\vec b) &\mathclap\to& ((\mathcal M\rtimes\mathcal G)\wr\mathbb E_{\mathcal G})(\vec a,\vec b) \\[1ex]
  &(\varphi;f_1,\dots,f_n;[u],[x]) &\mathclap\mapsto& \left(\varphi;(f_1,\delta^{(\varphi)\ast}_1(u)),\dots,(f_n,\delta^{(\varphi)\ast}_n(u));[x]\right)
\end{array}
\]
for $\vec a=a_1\dots a_m,\vec b=b_1\dots b_n\in\mathcal M\wr\mathbb G_{\mathcal G}$ form an identity-on-objects functor
\[
\Phi:\mathcal M\wr\mathbb G_{\mathcal G}\to(\mathcal M\rtimes\mathcal G)\wr\mathbb E_{\mathcal G}\ .
\]
Moreover, $\Phi$ is an isomorphism of categories which is $2$-natural with respect to $\mathcal M\in\mathbf{MultCat}$.
\end{theorem}
\begin{proof}
First notice that, by virtue of \cref{lem:rst-split}, for each class $[u]\in\operatorname{Inr}^{\mathcal G}_\varphi\backslash\overbar{\operatorname{Dec}}^{\mathcal G}_\varphi$, the element $\delta^{(\varphi)\ast}_j(u)$ does not depend on the choice of the representative $u\in\overbar{\operatorname{Dec}}^{\mathcal G}_\varphi$ for every $1\le j\le n$.
In particular, in view of \cref{lem:conginr}, we may take $u$ of the form
\begin{equation}
\label{eq:prf:MG-compare:u}
u
=\gamma(e_{m+2};e^{(\varphi)}_{-\infty},u_1,\dots,u_m,e^{(\varphi)}_\infty)
\in\overbar{\operatorname{Dec}}^{\mathcal G}_\varphi\subset \mathcal G(m)
\end{equation}
with $u_i\in\mathcal G(k^\varphi_i)$, where $e^{(\varphi)}_{\pm\infty}:=e_{k^{(\varphi)}_{\pm\infty}}$.
In this case, we have $\delta^{(\varphi)\ast}_i(u)=u_i$ so that, for each morphism in $\mathcal M\wr\mathbb G_{\mathcal G}$ of the form $(\varphi;f_1,\dots,f_m;[u],[x])$, we have
\[
\Phi(\varphi;f_1,\dots,f_m;[u],[x])
= (\varphi;(f_1,u_1),\dots,(f_m,u_m);[x])\ .
\]
Now, for every morphism $(\psi;g_1,\dots,g_n;[v],[y])$ in $\mathcal M\wr\mathbb G_{\mathcal G}$ postcomposable with $(\varphi;\vec f;[u],[x])$ above, the explicit formula of the composition operation in $\mathcal M\rtimes\mathcal G$ and the formulas in \cref{lem:word-calc} give the equation
\begin{equation}
\label{eq:prf:MG-compare:compPhi}
\begin{split}
&\Phi(\psi;\vec g;[v],[y])\circ\Phi(\varphi;\vec f;[u],[x]) \\
&
\begin{multlined}
=\left(\psi\varphi^y;\left(\gamma_{\mathcal M}(g_1;\delta^{(\psi)\ast}_1(v)_\ast(y_\ast\vec f)^\psi_1),\gamma_{\mathcal G}(\delta^{(\psi)\ast}_1(v);(y_\ast\vec u)^\psi_1)\right),\right. \\
\qquad\left.\dots,\left(\gamma_{\mathcal M}(g_n;\delta^{(\psi)\ast}_n(v)_\ast(y_\ast\vec f)^\psi_n),\gamma_{\mathcal G}(\delta^{(\psi)\ast}_n(v);(y_\ast\vec u)^\psi_n)\right);[\varphi^\ast(y)x]\right)
\end{multlined}
\\
&
\begin{multlined}
= \left(\psi\varphi^y;\left(\gamma_{\mathcal M}(g_1;((vy)_\ast\vec f)^\psi_1),\gamma_{\mathcal G}(\delta^{(\psi)\ast}_1(v);(y_\ast\vec u)^\psi_1)\right),\right. \\
\qquad\left.\dots,\left(\gamma_{\mathcal M}(g_n;((vy)_\ast\vec f)^\psi_n),\gamma_{\mathcal G}(\delta^{(\psi)\ast}_n(v);(y_\ast\vec u)^\psi_n)\right); [\varphi^\ast(y)x]\right)\ .
\end{multlined}
\end{split}
\end{equation}
The comparison of \eqref{eq:prf:MG-compare:compPhi} with the formula \eqref{eq:MG-comp} tells us that, in order to have the functoriality of $\Phi$, we only have to verify the equation
\begin{equation}
\label{eq:prf:MG-compare:vj}
\delta^{(\psi\varphi^y)\ast}_j(\varphi^\ast(vy)u\varphi^\ast(y)^{-1})
= \gamma(\delta^{(\psi)\ast}_j(v);(y_\ast\vec u)^\psi_j)
\end{equation}
for each $1\le j\le n$.
By virtue of \cref{lem:rst-split}, the left hand side of \eqref{eq:prf:MG-compare:vj} equals
\begin{equation}
\label{eq:prf:MG-compare:uvexpand}
(\varphi^y\delta^{(\psi\varphi^y)}_j)^\ast(v)
\cdot \delta^{(\psi\varphi^y)\ast}_j\left(\varphi^\ast(y)u\varphi^\ast(y)^{-1}\right)\ .
\end{equation}
Notice that there is a unique active morphism $\varphi'_j:\llangle k^{(\psi\varphi^y)}_j\rrangle\to\llangle k^{(\psi)}_j\rrangle$ which makes the square below commute:
\begin{equation}
\label{eq:prf:MG-compare:pb}
\vcenter{
  \xymatrix{
    \llangle k^{(\psi\varphi^y)}_j\rrangle \ar[r]^{\varphi'_j} \ar[d]_{\delta^{(\psi\varphi^y)}_j} & \llangle k^{(\psi)}_j\rrangle \ar[d]^{\delta^{(\psi)}_j} \ar[r] & \llangle 1\rrangle \ar[d]^{\{j\}} \\
    \llangle l\rrangle \ar[r]^{\varphi^y} & \llangle m\rrangle \ar[r]^\psi & \llangle n\rrangle }}
\quad.
\end{equation}
It turns out that each square in \eqref{eq:prf:MG-compare:pb} forms a pullback square of (ordinary) maps, so one has
\[
k^{(\varphi'_j)}_s
= k^{(\varphi^y)}_{\delta^{(\psi)}_j(s)}
= k^{(\varphi)}_{y^{-1}(\delta^{(\psi)}_j(s))}
\]
for each $1\le s\le k^{(\psi)}_j$.
Thus, we obtain
\begin{equation}
\label{eq:prf:MG-compare:deltav}
\begin{split}
(\varphi^y\delta^{(\psi\varphi^y)}_j)^\ast(v)
&= (\delta^{(\psi)}_j\varphi'_j)^\ast(v) \\
&= \gamma_{\mathcal G}\bigl(\delta^{(\psi)\ast}_j(v);e^{(\varphi)}_{y^{-1}(\delta^{(\psi)}_j(1))},\dots,e^{(\varphi)}_{y^{-1}(\delta^{(\psi)}_j(k^{(\psi)}_j))}\bigr) \\
&= \gamma_{\mathcal G}\left(\delta^{(\psi)\ast}_j(v);(y_\ast\vec e^{(\varphi)})^\psi_j\right) \\
\end{split}
\end{equation}
where $e^{(\varphi)}_i = e_{k^{(\varphi)}_i}$.
On the other hand, in view of the presentation \eqref{eq:prf:MG-compare:u}, we have
\begin{equation}
\label{eq:prf:MG-compare:deltau}
\begin{split}
\delta^{(\psi\varphi^y)\ast}_j\left(\varphi^\ast(y)u\varphi^\ast(y)^{-1}\right)
&= \delta^{(\psi\varphi^y)\ast}_j\left(\gamma_{\mathcal G}(e_{m+2};e^{(\varphi)}_{-\infty},u_{y^{-1}(1)},\dots,u_{y^{-1}(m)},e^{(\varphi)}_\infty)\right) \\
&= \gamma_{\mathcal G}\bigl(e^{(\psi)}_j;u_{y^{-1}(\delta^{(\psi)}_j(1))},\dots,u_{y^{-1}(\delta^{(\psi)}_j(k^{(\psi)}_j))}\bigr) \\
&= \gamma_{\mathcal G}\left(e^{(\psi)}_j;(y_\ast\vec u)^\psi_j\right)
\end{split}
\end{equation}
Substituting \eqref{eq:prf:MG-compare:deltav} and \eqref{eq:prf:MG-compare:deltau} into \eqref{eq:prf:MG-compare:uvexpand}, we obtain \eqref{eq:prf:MG-compare:vj}, which implies $\Phi$ is actually a functor.

The $2$-naturality of $\Phi$ immediately follows from definition.
We verify $\Phi$ is an isomorphism of categories.
Since it is the identity on objects, it suffices to show $\Phi$ is bijective on each hom-sets.
This is actually a consequence of \ref{req:conginr:inert} in \cref{lem:conginr}.
\end{proof}

\begin{corollary}
\label{cor:ME-intpresh}
For every group operad $\mathcal G$, the $2$-functor $(\blank)\wr\mathbb E_{\mathcal G}$ admits a lift depicted as the dashed arrow in the diagram below:
\[
\vcenter{
  \xymatrix{
    \mathbf{MultCat}_{\mathcal G} \ar@{-->}[r] \ar[d]_{\mathit{forget}} & \mathbf{PSh}(\mathbb G_{\mathcal G}\rightrightarrows\mathbb E_{\mathcal G}) \ar[d]^{\mathit{forget}} \\
    \mathbf{MultCat} \ar[r]^{(\blank)\wr\mathbb E_{\mathcal G}} & \mathbf{Cat}^{/\mathbb E_{\mathcal G}} }}
\quad.
\]
\end{corollary}
\begin{proof}
In view of \cref{theo:MG-compare}, each $\mathcal G$-symmetric multicategory $\mathcal M$ admits a canonical functor
\begin{equation}
\label{eq:prf:ME-intpresh:symstr}
(\mathcal M\wr\mathbb E_{\mathcal G})\times_{\mathbb E_{\mathcal G}}\mathbb G_{\mathcal G}
= \mathcal M\wr\mathbb G_{\mathcal G}
\xrightarrow[\cong]{\Phi} (\mathcal M\rtimes\mathcal G)\wr\mathbb E_{\mathcal G}
\to \mathcal M\wr\mathbb E_{\mathcal G}\ .
\end{equation}
\Cref{lem:rst-split} and the direct computation shows that it is in fact a structure of an internal presheaf over the double category $\mathbb G_{\mathcal G}\rightrightarrows\mathbb E_{\mathcal G}$.
Moreover, since the isomorphism $\Phi$ is $2$-natural, the structure functor \eqref{eq:prf:ME-intpresh:symstr} is also $2$-natural with respect to $\mathcal G$-symmetric multicategories $\mathcal M$.
Therefore, we obtain the result.
\end{proof}

We finally obtain an analogues on quotients.

\begin{theorem}
\label{theo:tildeMG-compare}
Let $\mathcal G$ be a group operad, and let $\mathcal M$ be a multicategory.
Then, there is an isomorphism $\widetilde\Phi:\mathcal M\wr\widetilde{\mathbb G}_{\mathcal G}\cong(\mathcal M\rtimes\mathcal G)\wr\widetilde{\mathbb E}_{\mathcal G}$ which is the identity on objects and, on each hom-set, described as
\begin{equation}
\label{eq:tildeMG-compare:def}
\widetilde\Phi([\varphi;f_1,\dots,f_n;u,x])
= \left[\varphi;(f_1,\delta^{(\varphi)\ast}_1(u)),\dots,(f_n,\delta^{(\varphi)\ast}_n(u));x\right]\ .
\end{equation}
Moreover, $\widetilde\Phi$ is a $2$-natural transformation with respect to $\mathcal M\in\mathbf{MultCat}$ such that the diagram below is commutative:
\[
\xymatrix{
  \mathcal M\wr\mathbb G_{\mathcal G} \ar[r]^-\Phi_-\cong \ar[d] & (\mathcal M\rtimes\mathcal G)\wr\mathbb E_{\mathcal G} \ar[d] \\
  \mathcal M\wr\widetilde{\mathbb G}_{\mathcal G} \ar[r]^-{\widetilde\Phi}_-\cong & (\mathcal M\rtimes\mathcal G)\wr\widetilde{\mathbb E}_{\mathcal G} }
\]
\end{theorem}
\begin{proof}
We have the following commutative diagram of functors:
\begin{equation}
\vcenter{
  \xymatrix@R=2ex@C=0em{
    & \mathcal M\wr\widetilde{\mathbb G}_{\mathcal G} \ar[rr] \ar[dd]|\hole && \mathcal M\wr\widetilde{\mathbb E}_{\mathcal G} \ar[dd] \\
    \mathcal M\wr\mathbb G_{\mathcal G} \ar[rr] \ar[dd] \ar[ur] && \mathcal M\wr\mathbb E_{\mathcal G} \ar[dd] \ar[ur] & \\
    & \widetilde{\mathbb G}_{\mathcal G} \ar[rr]|-\hole^(.7){t} && \widetilde{\mathbb E}_{\mathcal G} \\
    \mathbb G_{\mathcal G} \ar[rr]^{t} \ar[ur] && \mathbb E_{\mathcal G} \ar[ur] }}
\end{equation}
Note that, \cref{lem:QQ-quot,lem:rtEG-pb} assert that the bottom and the right faces, as well as the front and the back, are pullbacks.
Hence, the ``associativity property'' of pullbacks (e.g. see Proposition 2.5.9 in \cite{Bor94I}) implies the other faces are also pullbacks.
In particular, we obtain isomorphisms of categories:
\begin{equation}
\label{eq:prf:tildeMG-compare:isom}
(\mathcal M\wr\widetilde{\mathbb G}_{\mathcal G})\times_{\widetilde{\mathbb E}_{\mathcal G}}\mathbb E_{\mathcal G}
\cong \mathcal M\wr\mathbb G_{\overbar{\operatorname{Dec}}^{\mathcal G}}
\xrightarrow[\cong]{\Phi} (\mathcal M\rtimes\mathcal G)\wr\mathbb E_{\mathcal G}
\cong ((\mathcal M\rtimes\mathcal G)\wr\widetilde{\mathbb E}_{\mathcal G})\times_{\widetilde{\mathbb E}_{\mathcal G}}\mathbb E_{\mathcal G}
\end{equation}
The explicit computation shows that the isomorphism \eqref{eq:prf:tildeMG-compare:isom} is induced by the identity on $\mathbb E_{\mathcal G}$ and an identity-on-object functor $\widetilde\Phi:\mathcal M\wr\widetilde{\mathbb G}_{\mathcal G}\to(\mathcal M\rtimes\mathcal G)\wr\widetilde{\mathbb E}_{\mathcal G}$ described as \eqref{eq:tildeMG-compare:def}.
Moreover, since the functor $\mathbb E_{\mathcal G}\to\widetilde{\mathbb E}_{\mathcal G}$ is full and the identity on objects, the pullback along it preserves and reflects fully-faithfulness.
Thus, we conclude $\widetilde\Phi$ is an isomorphism of categories.
The $2$-naturality and the compatibility with $\Phi$ are obvious.
\end{proof}

\begin{corollary}
\label{cor:tildeME-intpresh}
For every group operad $\mathcal G$, the $2$-functor $(\blank)\wr\widetilde{\mathbb E}_{\mathcal G}$ admits a lift depicted as the dashed arrow in the diagram below:
\[
\vcenter{
  \xymatrix{
    \mathbf{MultCat}_{\mathcal G} \ar@{-->}[r] \ar[d]_{\mathit{forget}} & \mathbf{PSh}(\widetilde{\mathbb G}_{\mathcal G}\rightrightarrows\widetilde{\mathbb E}_{\mathcal G}) \ar[d]^{\mathit{forget}} \\
    \mathbf{MultCat} \ar[r]^{(\blank)\wr\widetilde{\mathbb E}_{\mathcal G}} & \mathbf{Cat}^{/\widetilde{\mathbb E}_{\mathcal G}} }}
\quad.
\]
\end{corollary}

\section{CoCartesian lifting properties}
\label{sec:cocart-lifts}

We investigate the image of the functor $\mathbf{MultCat}_{\mathcal G}\to\mathbf{PSh}(\widetilde{\mathbb G}_{\mathcal G}\rightrightarrows\widetilde{\mathbb E}_{\mathcal G})$ given in \cref{cor:tildeME-intpresh}.

\begin{definition}
For a crossed interval group $G$, a morphism in $\widetilde{\mathbb E}_G$ is called \emph{active} (resp. \emph{inert}) if it is of the form $[\mu,x]$ for $\mu:\llangle m\rrangle\to\llangle n\rrangle\in\nabla$ active (resp. \emph{inert}) and arbitrary $x\in G_m$.
\end{definition}

In particular, the functor $\nabla\to\widetilde{\mathbb E}_G$ preserves active morphisms and inert morphisms respectively.
Throughout the section, the following inert morphisms in $\nabla$ play important roles: for each $1\le i\le n$, we define a morphism $\rho_i:\llangle n\rrangle\to\llangle 1\rrangle$ by
\[
\rho_i(j) =
\begin{cases}
-\infty & j<i\ , \\
1 & j=i\ , \\
\infty & j>i\ .
\end{cases}
\]
By abuse of notation, we use the same notation $\rho_i$ to denote its image in $\widetilde{\mathbb E}_G$.

\begin{proposition}
\label{prop:inert-lift}
Let $\mathcal G$ be a group operad, and let $\mathcal M$ be a multicategory.
Then, the canonical functor $p_{\mathcal M}:\mathcal M\wr\widetilde{\mathbb E}_{\mathcal G}\to\widetilde{\mathbb E}_{\mathcal G}$ satisfies the following properties.
\begin{enumerate}[label={\upshape(\arabic*)}]
  \item\label{req:inert-lift:cocart} Every inert morphism $[\rho,x]:\llangle m\rrangle\to\llangle n\rrangle\in\widetilde{\mathbb E}_{\mathcal G}$ admits $p_{\mathcal M}$-coCartesian lifts along any object in the fiber $(\mathcal M\wr\widetilde{\mathbb E}_{\mathcal G})_{\llangle m\rrangle}:=p_{\mathcal M}^{-1}\{\llangle m\rrangle\}$.
More precisely, if $\delta:\llangle n\rrangle\to\llangle m\rrangle$ is the section of $\rho$, then for each $\vec a=a_1\dots a_m\in(\mathcal M\wr\widetilde{\mathbb E}_{\mathcal G})_{\llangle n\rrangle}$, the morphism
\begin{equation}
\label{eq:inert-lift:stdcocart}
\begin{multlined}
\widehat{[\rho,x]}_{\vec a} := [\rho;\mathrm{id}_{a_{x^{-1}(\delta(1))}},\dots,\mathrm{id}_{a_{x^{-1}(\delta(n))}};x] \\
\mkern100mu:a_1\dots a_m\to a_{x^{-1}(\delta(1))}\dots a_{x^{-1}(\delta(n))}
\end{multlined}
\end{equation}
is $p_{\mathcal M}$-coCartesian.
  \item\label{req:inert-lift:plim} For an object $\vec a=a_1\dots a_n\in(\mathcal M\wr\widetilde{\mathbb E}_{\mathcal G})_{\llangle n\rrangle}$, choose a $p_{\mathcal M}$-coCartesian lift $\widehat\rho_j:\vec a\to a'_j$ of $\rho_j$ along $\vec a$ for each $1\le i\le n$.
Then, for every object $\vec b\in\mathcal M\wr\widetilde{\mathbb E}_{\mathcal G}$, the square below is a pullback:
\begin{equation}
\label{eq:inert-lift:plim}
\vcenter{
  \xymatrix@C=6em{
    (\mathcal M\wr\widetilde{\mathbb E}_{\mathcal G})(\vec b,\vec a) \ar[r]^-{((\widehat\rho_1)_\ast,\dots,(\widehat\rho_n)_\ast)} \ar[d]_{p_{\mathcal M}} \ar@{}[dr]|(.4)\pbcorner & \prod_{i=1}^n(\mathcal M\wr\widetilde{\mathbb E}_{\mathcal G})(\vec b,a'_i) \ar[d]^{p_{\mathcal M}} \\
    \widetilde{\mathbb E}_{\mathcal G}(p_{\mathcal M}(\vec b),\llangle n\rrangle) \ar[r]^-{((\rho_1)_\ast,\dots,(\rho_n)_\ast)} & \widetilde{\mathbb E}_{\mathcal G}(p_{\mathcal M}(\vec b),\llangle 1\rrangle)^{\times n} }}
\quad.
\end{equation}
  \item\label{req:inert-lift:prod} For each $1\le i\le n$, take a functor $(\rho_i)_!:(\mathcal M\wr\widetilde{\mathbb E}_{\mathcal G})_{\llangle n\rrangle}\to(\mathcal M\wr\widetilde{\mathbb E}_{\mathcal G})_{\llangle 1\rrangle}$ together with a natural transformation $\widehat\rho_i:\vec a\to(\rho_i)_!\vec a$ which is (componentwisely) $p_{\mathcal M}$-coCartesian.
Then the functor
\[
((\rho_1)_!,\dots,(\rho_n)_!):(\mathcal M\wr\widetilde{\mathbb E}_{\mathcal G})_{\llangle n\rrangle} \to (\mathcal M\wr\widetilde{\mathbb E}_{\mathcal G})_{\llangle 1\rrangle}^{\times n}
\]
is an equivalence of categories:
\end{enumerate}
\end{proposition}

\begin{remark}
\label{rem:cocart-indep}
The condition \ref{req:inert-lift:plim} actually does not depend on the choice of coCartesian lifts $\widehat\rho_i$ of $\rho_i$.
Indeed, if one choose another coCartesian lift $\widehat\rho'_i:\vec a\to a''_i$, then the uniqueness of the coCartesian lifts implies there is a unique isomorphism $a'_i\cong a''_i$ so that $\widehat\rho'_i$ factors through $\widehat\rho_i$ followed by the isomorphism.
Moreover, it also gives rise to an isomorphism of squares \eqref{eq:inert-lift:plim}.
Thus, if \ref{req:inert-lift:plim} satisfied for one family of coCartesian lifts, then it is also for the other.

A similar argument shows that the condition \ref{req:inert-lift:prod} does not depend on the choice of the functors $(\rho_i)_!$.
\end{remark}

\begin{proof}[Proof of \cref{prop:inert-lift}]
In order to verify \ref{req:inert-lift:cocart}, it clearly suffices to consider only the case $x\in G_m$ is the unit.
For an inert morphism $\rho:\llangle m\rrangle\to\llangle n\rrangle\in\nabla$, set $\delta$ to be the unique section, and suppose we have a morphism in $\mathcal M\wr\widetilde{\mathbb E}_{\mathcal G}$ of the form
\[
[\varphi\rho;f_1,\dots,f_l;\rho^\ast(y)]:a_1\dots a_m\to\vec b\ .
\]
We show it uniquely factors through the morphism
\[
\widehat\rho_{\vec a}=[\rho;\mathrm{id}_{a_{\delta(1)}},\dots,\mathrm{id}_{a_{\delta(n)}};e_m]:\vec a\to a_{\delta(1)}\dots a_{\delta(n)}
\]
Thanks to the unique factorization in $\nabla$, we have
\begin{equation}
\label{eq:prf:inert-lift:fact}
[\varphi\rho;f_1,\dots,f_l;\rho^\ast(y)]
= [\varphi;f_1,\dots,f_l;y]\circ\widehat\rho_{\vec a}\ ,
\end{equation}
so there in fact exists a factorization.
Moreover, since the morphism $[\varphi;f_1,\dots,f_l;y]$ is uniquely determined by the underlying morphism $[\varphi,y]$ in $\widetilde{\mathbb E}_{\mathcal G}$ and the tuple $(f_1,\dots,f_l)$, which is determined by the left hand side.
This implies the factorization \eqref{eq:prf:inert-lift:fact} is unique, so $\widehat\rho_{\vec a}$ is $p_{\mathcal M}$-coCartesian.

We next see \ref{req:inert-lift:plim}.
For an object $\vec a=a_1\dots a_n\in(\mathcal M\wr\widetilde{\mathbb E}_{\mathcal G})_{\llangle n\rrangle}$, in view of \cref{rem:cocart-indep}, we may assume the lift $\widehat\rho_i=(\widehat\rho_i)_{\vec a}$ is the one given in the part \ref{req:inert-lift:cocart}.
Suppose $\vec b=b_1\dots b_m\in\mathcal M\wr\widetilde{\mathbb E}_{\mathcal G}$, and $[\varphi,x]:\llangle m\rrangle\to\llangle n\rrangle\in\widetilde{\mathbb E}_{\mathcal G}$.
Then, if one has a morphism of the form
\begin{equation}
\label{eq:prf:inert-lift:tuple}
[\varphi;f_1,\dots,f_n;x]:\vec b\to\vec a\ ,
\end{equation}
then $f_i\in\mathcal M((x_\ast\vec b)^\varphi_i;a_i)$.
On the other hand, we have $(x_\ast\vec b)^{\rho_i\varphi}_1=(x_\ast\vec b)^\varphi_i$ so that \eqref{eq:prf:inert-lift:tuple} makes sense if and only if we have morphisms
\begin{equation}
\label{eq:prf:inert-lift:component}
[\rho_i\varphi;f_i;x]:\vec b\to a_i
\end{equation}
for $1\le i\le n$.
When we fix a morphism $[\varphi,x]$ in $\widetilde{\mathbb E}_{\mathcal G}$, the two data \eqref{eq:prf:inert-lift:tuple} and \eqref{eq:prf:inert-lift:component} clearly correspond in one-to-one to each other.
It follows that the square \eqref{eq:inert-lift:plim} is a pullback.

We finally show \ref{req:inert-lift:prod}.
Note that, in view of \cref{ex:Icong-active}, every morphism in $(\mathcal M\wr\widetilde{\mathbb E}_{\mathcal G})_{\llangle n\rrangle}$ is of the form
\[
[\mathrm{id}_{\llangle n\rrangle};f_1,\dots,f_n;e_n]:a_1\dots a_n\to b_1\dots b_n
\]
with $f_i\in\mathcal M(a_i;b_i)=\underline{\mathcal M}(a_i,b_i)$, here $\underline{\mathcal M}$ is the underlying category of $\mathcal M$.
In other words, we have a canonical isomorphism
\begin{equation}
\label{eq:prf:inert-lift:prodisom}
(\mathcal M\wr\widetilde{\mathbb E}_{\mathcal G})_{\llangle n\rrangle}
\cong \underline{\mathcal M}^{\times n}
\cong (\mathcal M\wr\widetilde{\mathbb E}_{\mathcal G})_{\llangle 1\rrangle}^{\times n}\ .
\end{equation}
Hence, it suffices to show the functor $(\rho_i)_!:(\mathcal M\wr\widetilde{\mathbb E}_{\mathcal G})_{\llangle n\rrangle}\to(\mathcal M\wr\widetilde{\mathbb E}_{\mathcal G})_{\llangle 1\rrangle}$ coincides with the projection under the isomorphism \eqref{eq:prf:inert-lift:prodisom}.
If $(\rho_i)_!$ is the one induced by the $p_{\mathcal M}$-coCartesian lifts in the part \ref{req:inert-lift:cocart}, this follows from the correspondence of \eqref{eq:prf:inert-lift:tuple} to \eqref{eq:prf:inert-lift:component} and the unique factorization \eqref{eq:prf:inert-lift:fact}.
In view of \cref{rem:cocart-indep}, this completes the proof.
\end{proof}

We define a $2$-subcategory $\mathbf{Oper}'_{\mathcal G}\subset\mathbf{Cat}^{/\widetilde{\mathbb E}_{\mathcal G}}$ as follows:
\begin{itemize}
  \item objects of $\mathbf{Oper}'_{\mathcal G}$ are those categories $\mathcal C$ over $\widetilde{\mathbb E}_{\mathcal G}$ that satisfy three properties in \cref{prop:inert-lift};
  \item for $\mathcal C,\mathcal D\in\mathbf{Oper}'_{\mathcal G}$, the hom-category $\mathbf{Oper}'_{\mathcal G}(\mathcal C,\mathcal D)$ is the full subcategory of $\mathbf{Cat}^{/\widetilde{\mathbb E}_{\mathcal G}}$ spanned by functors $\mathcal C\to\mathcal D$ over $\widetilde{\mathbb E}_{\mathcal G}$ which preserve coCartesian lifts of inert morphisms in $\widetilde{\mathbb E}_{\mathcal G}$.
\end{itemize}
Furthermore, we put
\[
\mathbf{Oper}^{\mathsf{alg}}_{\mathcal G}
:= \mathbf{PSh}(\widetilde{\mathbb G}_{\mathcal G}\rightrightarrows\widetilde{\mathbb E}_{\mathcal G})\times_{\mathbf{Cat}^{/\widetilde{\mathbb E}_{\mathcal G}}}\mathbf{Oper}'_{\mathcal G}\ ,
\]
whose objects are called \emph{categories of algebraic $\mathcal G$-operators}, and whose morphisms \emph{maps of algebraic $\mathcal G$-operators}.
In other words, a category of operators is an internal presheaf $\mathcal X$ over the double category $\widetilde{\mathbb G}_{\mathcal G}\rightrightarrows\widetilde{\mathbb E}_{\mathcal G}$ with the functor $p:\mathcal X\to\widetilde{\mathbb E}_{\mathcal G}$ satisfying the following conditions:
\begin{enumerate}[label={\upshape(\roman*)}]
  \item\label{cond:catalgop:cocart} every inert morphism $\bsrho:\llangle m\rrangle\to\llangle n\rrangle\in\widetilde{\mathbb E}_{\mathcal G}$ admits $p$-coCartesian lifts along any object in the fiber $\mathcal X_{\llangle m\rrangle}:=p^{-1}\{\llangle m\rrangle\}$;
  \item\label{cond:catalgop:plim} if we are given a $p$-coCartesian morphism $\widehat\rho_j:X\to X_i$ covering the inert morphism $\rho_i:\llangle n\rrangle\to\llangle 1\rrangle\in\widetilde{\mathbb E}_{\mathcal G}$ for each $1\le i\le n$, for every object $W\in\mathcal X$, the square below is a pullback:
\begin{equation}
\label{eq:catalgop:plim:sq}
\vcenter{
  \xymatrix@C=6em{
    \mathcal X(W,X) \ar[r]^-{((\widehat\rho_1)_\ast,\dots,(\widehat\rho_n)_\ast)} \ar[d]_p \ar@{}[dr]|(.4)\pbcorner & \prod_{i=1}^n\mathcal X(W,X_i) \ar[d]^p \\
    \widetilde{\mathbb E}_{\mathcal G}(p(W),\llangle n\rrangle) \ar[r]^-{((\rho_1)_\ast,\dots,(\rho_n)_\ast)} & \widetilde{\mathbb E}_{\mathcal G}(p(W),\llangle 1\rrangle)^{\times n} }}
\quad;
\end{equation}
  \item\label{cond:catalgop:prod} if $(\rho_i)_!:\mathcal X_{\llangle n\rrangle}\to\mathcal X_{\llangle 1\rrangle}$ is a functor induced by the inert morphism $\rho_i:\llangle n\rrangle\to\llangle 1\rrangle$ for each $1\le i\le n$, then the functor
\[
((\rho_1)_!,\dots,(\rho_n)_!):(\mathcal M\wr\widetilde{\mathbb E}_{\mathcal G})_{\llangle n\rrangle} \to (\mathcal M\wr\widetilde{\mathbb E}_{\mathcal G})_{\llangle 1\rrangle}^{\times n}
\]
is an equivalence of categories.
\end{enumerate}
Thanks to \cref{cor:tildeME-intpresh} and \cref{prop:inert-lift}, the $2$-functor $(\blank)\wr\widetilde{\mathbb E}_{\mathcal G}$ induces a $2$-functor $\mathbf{MultCat}_{\mathcal G}\to\mathbf{Oper}^{\mathsf{alg}}_{\mathcal G}$, which we also denote by $(\blank)\wr\widetilde{\mathbb E}_{\mathcal G}$ by abuse of notation.
Thanks to \cref{lem:inrpresh-locfull}, the forgetful functor
\[
\mathbf{Oper}^{\mathsf{alg}}_{\mathcal G}
\to \mathbf{Oper}'_{\mathcal G}
\]
is locally fully faithful.

\begin{example}
\label{ex:ast-oper}
As we have $\ast\wr\widetilde{\mathbb E}_{\mathcal G}\cong\widetilde{\mathbb E}_{\mathcal G}$, the identity functor $\widetilde{\mathbb E}_{\mathcal G}\to\widetilde{\mathbb E}_{\mathcal G}$ exhibits $\widetilde{\mathbb E}_{\mathcal G}$ as a category of algebraic $\mathcal G$-operators.
\end{example}

\begin{example}
\label{ex:G-oper}
Recall that every group operad $\mathcal G$ is itself a $\mathcal G$-symmetric multicategory with the multiplication map $\mathcal G\rtimes\mathcal G\to\mathcal G$.
On the other hand, in view of \cref{theo:tildeMG-compare}, we have isomorphisms
\[
\widetilde{\mathbb G}_{\mathcal G}
\cong \ast\wr\widetilde{\mathbb G}_{\mathcal G}
\cong (\ast\rtimes\mathcal G)\wr\widetilde{\mathbb E}_{\mathcal G}
\cong \mathcal G\wr\widetilde{\mathbb E}_{\mathcal G}\ .
\]
It follows that the functor $s:\widetilde{\mathbb G}_{\mathcal G}\to\widetilde{\mathbb E}_{\mathcal G}$ exhibits $\widetilde{\mathbb G}_{\mathcal G}$ as a category of algebraic $\mathcal G$-operators.
\end{example}

It turns out that there are \emph{free $\mathcal G$-symmetrizations} of objects in $\mathbf{Oper}'_{\mathcal G}$.
Indeed, we have the following property on the free construction.

\begin{lemma}
\label{lem:free-cocart}
For every $\mathcal C\in\mathbf{Oper}'_{\mathcal G}$, the functor
\begin{equation}
\label{eq:free-cocart:unit}
\mathcal C
\cong\mathcal C\times_{\widetilde{\mathbb E}_{\mathcal G}}\widetilde{\mathbb E}_{\mathcal G}
\xrightarrow{\mathrm{Id}\times\iota}\mathcal C\times_{\widetilde{\mathbb E}_{\mathcal G}}\widetilde{\mathbb E}_{\mathcal G}
\end{equation}
preserves coCartesian lifts of inert morphisms.
\end{lemma}
\begin{proof}
Note that the category $\mathcal C\times_{\widetilde{\mathbb E}_{\mathcal G}}\widetilde{\mathbb G}_{\mathcal G}$ is described as follows:
\begin{itemize}
  \item objects are the same as $\mathcal C$;
  \item for $X,Y\in\mathcal C$, the hom-set $(\mathcal C\times_{\widetilde{\mathbb E}_{\mathcal G}}\widetilde{\mathbb G}_{\mathcal G})(X,Y)$ consists of tuples $[\varphi;f;u,x]$ so that $[\varphi,u,x]:q(X)\to q(Y)\in\widetilde{\mathbb G}_{\mathcal G}$ makes sense and $f:X\to Y\in\mathcal C$ with $q(f)=[\varphi,ux]$;
  \item the composition is given by
\[
[\psi;g;v,y]\circ[\varphi;f;u,x]
= [\psi\varphi^y;gf;\varphi^\ast(vy)u\varphi^\ast(y)^{-1},\varphi^\ast(y)x]\ ;
\]
  \item the structure functor $\mathcal C\times_{\widetilde{\mathbb E}_{\mathcal G}}\widetilde{\mathbb G}_{\mathcal G}\to\widetilde{\mathbb E}_{\mathcal G}$ is given by
\[
[\varphi;f;u,x] \mapsto [\varphi,x]\ .
\]
\end{itemize}
Suppose $[\rho,x]:\llangle m\rrangle\to\llangle n\rrangle\in\widetilde{\mathbb E}_{\mathcal G}$ is an inert morphism, and take a coCartesian lift
\[
\widehat{[\rho,x]}_X:X\to X'\in\mathcal C
\]
along $X\in\mathcal C$.
The functor \eqref{eq:free-cocart:unit} sends it to
\begin{equation}
\label{eq:prf:free-cocart:img}
[\rho;\widehat{[\rho,x]}_X;e,x]:X\to X'\in\mathcal C\times_{\widetilde{\mathbb E}_{\mathcal G}}\widetilde{\mathbb G}_{\mathcal G}\ .
\end{equation}
To see \eqref{eq:prf:free-cocart:img} is coCartesian, consider a morphism in $\mathcal C\times_{\widetilde{\mathbb E}_{\mathcal G}}\widetilde{\mathbb G}_{\mathcal G}$ of the form $[\varphi\rho;f;u,x]:X\to Y$.
In view of \ref{req:conginr:inert} in \cref{lem:conginr}, there is a unique element $\overbar u\in\overbar{\operatorname{Dec}}^{\mathcal G}_\varphi$ such that $[u]=[\rho^\ast(\overbar u)]\in\overbar{\operatorname{Inr}}^{\mathcal G}_{\varphi\rho}\backslash\overbar{\operatorname{Dec}}^{\mathcal G}_{\varphi\rho}$, which implies
\[
[\varphi\rho,ux] = [\varphi,\overbar u]\circ[\rho,x]\ .
\]
On the other hand, since $\widehat{[\rho,x]}$ is coCartesian, there is a unique factorization $f=f'\circ\widehat{[\rho,x]}_X$ with $f':X'\to Y$ covering $[\varphi,\overbar u]$.
One obtains
\begin{equation}
\label{eq:prf:free-cocart:fact}
[\varphi\rho;f;u,x]
= [\varphi;f';\overbar u,e]\circ[\rho;\widehat{[\rho,x]}_X;e,x]\ .
\end{equation}
Since the morphisms $f'$ and $\overbar u$ are uniquely determined by the other data, the factorization \eqref{eq:prf:free-cocart:fact} is unique.
It follows that the morphism \eqref{eq:prf:free-cocart:img} is coCartesian.
\end{proof}

\begin{proposition}
\label{prop:free-Gsym}
The free $2$-functor
\[
\mathbf{Cat}^{/\widetilde{\mathbb E}_{\mathcal G}}\to\mathbf{PSh}(\widetilde{\mathbb G}_{\mathcal G}\rightrightarrows\widetilde{\mathbb E}_{\mathcal G})
\ ;\quad \mathcal C\mapsto \mathcal C\times_{\widetilde{\mathbb E}_{\mathcal G}}\widetilde{\mathbb G}_{\mathcal G}
\]
associated to the $2$-monad of internal presheaves over the double category $\widetilde{\mathbb G}\rightrightarrows\widetilde{\mathbb E}_{\mathcal G}$ restricts to a $2$-functor
\[
\mathbf{Oper}'_{\mathcal G}\to\mathbf{Oper}^{\mathsf{alg}}_{\mathcal G}\ .
\]
\end{proposition}
\begin{proof}
It clearly suffices to show the composition
\[
\mathbf{Oper}'_{\mathcal G}
\hookrightarrow\mathbf{Cat}^{/\widetilde{\mathbb E}_{\mathcal G}}
\xrightarrow{(\blank)\times_{\widetilde{\mathbb E}_{\mathcal G}}{\widetilde{\mathbb G}_{\mathcal G}}} \mathbf{PSh}(\widetilde{\mathbb G}_{\mathcal G}\rightrightarrows\widetilde{\mathbb E}_{\mathcal G})
\xrightarrow{\mathit{forget}} \mathbf{Cat}^{/\widetilde{\mathbb E}_{\mathcal G}}
\]
factors through the subcategory $\mathbf{Oper}'_{\mathcal G}$ at the end.
We have to verify it regarding objects and $1$-morphisms.

Let $\mathcal C\in\mathbf{Oper}'_{\mathcal G}$ with $q:\mathcal C\to\widetilde{\mathbb E}_{\mathcal G}$.
We verify the three conditions on categories of algebraic $\mathcal G$-operators for $\mathcal C\times_{\widetilde{\mathbb E}_{\mathcal G}}\widetilde{\mathbb G}_{\mathcal G}$.
Since the unit $\mathcal C\to\mathcal C\times_{\widetilde{\mathbb E}_{\mathcal G}}\widetilde{\mathbb G}_{\mathcal G}$ is the identity on objects, \cref{lem:free-cocart} implies $\mathcal C\times_{\widetilde{\mathbb E}_{\mathcal G}}\widetilde{\mathbb G}_{\mathcal G}$ admits all the coCartesian lifts of inert morphisms.
On the other hand, according to the description of $\mathcal C\times_{\widetilde{\mathbb E}_{\mathcal G}}\widetilde{\mathbb G}_{\mathcal G}$ in the proof of \cref{lem:free-cocart}, one easily verify the property \ref{cond:catalgop:plim}.
To see the property \ref{cond:catalgop:prod}, observe that the category $(\widetilde{\mathbb G}_{\mathcal G})_{\llangle n\rrangle}$ consists of automorphisms on $\llangle n\rrangle$ in $\widetilde{\mathbb G}_{\mathcal G}$ of the form
\[
[\mathrm{id}_{\llangle n\rrangle},u,e_n]
\]
for $u\in\overbar{\operatorname{Dec}}^{\mathcal G}_{\mathrm{id}_{\llangle n\rrangle}}$.
It turns out that such morphisms vanished by the functor $t:\widetilde{\mathbb G}_{\mathcal G}\to\widetilde{\mathbb E}_{\mathcal G}$, so we obtain isomorphisms
\[
(\mathcal C\times_{\widetilde{\mathbb E}_{\mathcal G}}\widetilde{\mathbb G}_{\mathcal G})_{\llangle n\rrangle}
\cong\mathcal C\times_{\widetilde{\mathbb E}_{\mathcal G}} (\widetilde{\mathbb G}_{\mathcal G})_{\llangle n\rrangle}
\cong\mathcal C_{\llangle n\rrangle}\times(\widetilde{\mathbb G}_{\mathcal G})_{\llangle n\rrangle}
\]
Under the identification, it is easily verified that, for each inert morphism $\rho_i:\llangle n\rrangle\to\llangle 1\rrangle$, the induced functor
\[
(\rho_i)_!:(\mathcal C\times_{\widetilde{\mathbb E}_{\mathcal G}}\widetilde{\mathbb G}_{\mathcal G})_{\llangle n\rrangle}
\to(\mathcal C\times_{\widetilde{\mathbb E}_{\mathcal G}}\widetilde{\mathbb G}_{\mathcal G})_{\llangle 1\rrangle}
\]
coincides with the one induced by
\[
(\rho_i)_!:\mathcal C_{\llangle n\rrangle}
\to\mathcal C_{\llangle 1\rrangle}
\ ,\quad
(\rho_i)_!:(\widetilde{\mathbb G}_{\mathcal G})_{\llangle n\rrangle}
\to(\widetilde{\mathbb G}_{\mathcal G})_{\llangle 1\rrangle}\ .
\]
Thus, the functor
\[
((\rho_1)_!,\dots,(\rho_n)_!):
(\mathcal C\times_{\widetilde{\mathbb E}_{\mathcal G}}\widetilde{\mathbb G}_{\mathcal G})_{\llangle n\rrangle}
\to (\mathcal C\times_{\widetilde{\mathbb E}_{\mathcal G}}\widetilde{\mathbb G}_{\mathcal G})_{\llangle 1\rrangle}^{\times n}
\]
is an equivalence.

As for $1$-morphisms, suppose $F:\mathcal C\to\mathcal D$ is a functor over $\widetilde{\mathbb E}_{\mathcal G}$ for $\mathcal C,\mathcal D\in\mathbf{Oper}'_{\mathcal G}$.
We have the following commutative square:
\begin{equation}
\label{eq:prf:free-Gsym:unitsq}
\vcenter{
  \xymatrix{
    \mathcal C \ar[r]^F \ar[d]_{\eta} & \mathcal D \ar[d]^\eta \\
    \mathcal C\times_{\widetilde{\mathbb E}_{\mathcal G}}\widetilde{\mathbb G}_{\mathcal G} \ar[r]^{F\times\mathrm{Id}} & \mathcal D\times_{\widetilde{\mathbb E}_{\mathcal G}}\widetilde{\mathbb G}_{\mathcal G} }}
\quad.
\end{equation}
In view of \cref{lem:free-cocart}, all the coCartesian lifts of inert morphisms in $\mathcal C\times_{\widetilde{\mathbb E}_{\mathcal G}}\widetilde{\mathbb G}_{\mathcal G}$ and $\mathcal D\times_{\widetilde{\mathbb E}_{\mathcal G}}\widetilde{\mathbb G}_{\mathcal G}$ are isomorphic to the images of ones in $\mathcal C$ and in $\mathcal D$ respectively by the vertical functors.
It follows that the bottom arrow in \eqref{eq:prf:free-Gsym:unitsq} preserves coCartesian lifts of inert morphisms as soon as so does the top.
The required result now follows immediately.
\end{proof}

\begin{corollary}
\label{cor:Gsym-act}
Let $\mathcal G$ be a group operad, and let $\mathcal C$ be a category of algebraic $\mathcal G$-operators.
Then, the functor
\[
\mathpzc A_{\mathcal C}:\mathcal C\times_{\widetilde{\mathbb E}_{\mathcal G}}\widetilde{\mathbb G}_{\mathcal G}
\to\mathcal C
\]
in the internal presheaf structure on $\mathcal C$ is a map of algebraic $\mathcal G$-operators.
\end{corollary}
\begin{proof}
By virtue of \cref{lem:free-cocart} and \cref{prop:free-Gsym}, $\mathpzc A_{\mathcal C}$ is a $1$-morphism in the $2$-category $\mathbf{Oper}'_{\mathcal G}$.
In addition, it is straightforward from the definition of internal presheaves that $\mathpzc A_{\mathcal C}$ is a map of internal presheaves.
Combining them, one obtains the result.
\end{proof}

\section{The equivalence of notions}
\label{sec:algopcat-equiv}

The goal of this section is to prove the following result.

\begin{theorem}
\label{theo:multcat-oper-eq}
Let $\mathcal G$ be a group operad.
Then, the $2$-functor
\[
(\blank)\wr\widetilde{\mathbb E}_{\mathcal G}:\mathbf{MultCat}_{\mathcal G}
\to\mathbf{Oper}^{\mathsf{alg}}_{\mathcal G}
\]
is a biequivalence of $2$-categories.
In other words, the following hold.
\begin{enumerate}[label={\upshape(\arabic*)}]
  \item\label{req:multcat-oper-eq:ff} It is essentially fully faithful; i.e. for every pair $(\mathcal M,\mathcal N)$ of $\mathcal G$-symmetric multicategories, the functor
\begin{equation}
\label{eq:multcat-oper-eq:homcat}
\begin{array}{ccc}
  \mathbf{MultCat}_{\mathcal G}(\mathcal M,\mathcal N) &\mathclap\to& \mathbf{Oper}^{\mathsf{alg}}_{\mathcal G}(\mathcal M\wr\widetilde{\mathbb E}_{\mathcal G},\mathcal N\wr\widetilde{\mathbb E}_{\mathcal G}) \\
  F,\alpha &\mathclap\mapsto& \widetilde F^{\mathcal G},\widetilde\alpha^{\mathcal G}
\end{array}
\end{equation}
is an equivalence of categories.
  \item\label{req:multcat-oper-eq:surj} It is essentially surjective; i.e. for every category of algebraic $\mathcal G$-operators $\mathcal C$, there is a $\mathcal G$-symmetric multicategory $\mathcal M$ together with an equivalence $\mathcal M\wr\widetilde{\mathbb E}_{\mathcal G}\simeq\mathcal C$ in $\mathbf{Oper}^{\mathsf{alg}}_{\mathcal G}$.
\end{enumerate}
\end{theorem}

\begin{remark}
It is known that a $2$-functor $\mathcal K\to\mathcal L$ between $2$-categories admits a \emph{pseudoinverse}, i.e. a pseudofunctor $\mathcal L\to\mathcal K$ which is the inverse up to natural isomorphisms, provided it is essentially fully faithful and essentially surjective in the sense in \cref{theo:multcat-oper-eq}.
The reader can find a sketch in Section 3.2 in \cite{Lack2010}.
Note that, unlikely the case of equivalences of ordinary $1$-categories, a pseudoinverse might not be a strict one.
\end{remark}

In order to prove \cref{theo:multcat-oper-eq}, we need to observe that coCartesian lifts of inert morphisms in $\mathcal M\wr\widetilde{\mathbb E}_{\mathcal G}$ are preserved \emph{coherently} by arbitrary maps of algebraic $\mathcal G$-operators.
To simplify the notation, we use the following convention: let $[\rho,x]:\llangle m\rrangle\to\llangle n\rrangle\in\widetilde{\mathbb E}_{\mathcal G}$ be an inert morphism.
Although we may denote by $\widehat{[\rho,x]}_X$ an arbitrary coCartesian lift $\rho$ along an object $X$ in a general category of algebraic $\mathcal G$-operators $\mathcal C$, we always assume $\widehat{[\rho,x]}_{\vec a}$ is the one in \ref{req:inert-lift:cocart} in \cref{prop:inert-lift} in the special case $\mathcal C=\mathcal M\wr\widetilde{\mathbb E}_{\mathcal G}$.
In particular, we write
\[
\widehat\rho_{\vec a} := \widehat{[\rho,e_m]}_{\vec a}
\ ,\quad \widehat x_{\vec a} := \widehat{[\mathrm{id},x]}_{\vec a}\ .
\]
Hence, if $\delta$ is the section of $\rho$, the induced functor
\[
\rho^{}_!:(\mathcal M\wr\widetilde{\mathbb E}_{\mathcal G})_{\llangle m\rrangle}\to(\mathcal M\wr\widetilde{\mathbb E}_{\mathcal G})_{\llangle n\rrangle}
\]
coincides with the canonical projection so that $\rho^{}_!(a_1\dots a_m)=a_{\delta(1)}\dots a_{\delta(n)}$.

\begin{lemma}
\label{lem:adapter-isom}
Let $\mathcal G$ be a group operad, and let $\mathcal M$ and $\mathcal N$ be $\mathcal G$-symmetric multicategories.
Suppose $H:\mathcal M\wr\widetilde{\mathbb E}_{\mathcal G}\to\mathcal N\wr\widetilde{\mathbb E}_{\mathcal G}$ is a map of algebraic $\mathcal G$-operators.
Then, for each $\vec a=a_1\dots a_m$ and each $1\le i\le m$, there is a unique isomorphism
\[
\lambda_{\vec a,i}:H(a_i)\xrightarrow{\cong}(\rho_i)_!H(\vec a)\in(\mathcal M\wr\widetilde{\mathbb E}_{\mathcal G})_{\llangle 1\rrangle}=\underline{\mathcal M}
\]
such that $(\widehat\rho_i)_{H(\vec a)}=\lambda_{\vec a,i}\circ H((\widehat\rho_i)_{\vec a})$.
Moreover, the family
\[
\left\{\lambda_{\vec a}=[\mathrm{id}_{\llangle m\rrangle};\lambda_{\vec a,1},\dots,\lambda_{\vec a,m};e_m]\;\middle|\;\vec a=a_1\dots a_m\in\mathcal M\wr\widetilde{\mathbb E}_{\mathcal G}\right\}
\]
enjoys the property that, for every inert morphism $[\rho,x]:\llangle m\rrangle\to\llangle n\rrangle\in\widetilde{\mathbb E}_{\mathcal G}$, say $\delta$ is the section of $\rho$, and for each $\vec a=a_1\dots a_m\in\mathcal M\wr\widetilde{\mathbb E}_{\mathcal G}$, the square below commutes in $\mathcal M\wr\widetilde{\mathbb E}_{\mathcal G}$:
\begin{equation}
\label{eq:adapter-isom:sq}
\vcenter{
  \xymatrix{
    H(a_1)\dots H(a_m) \ar[r]^{\lambda_{\vec a}} \ar[d]_{\widehat{[\rho,x]}_{H(a_1)\dots H(a_m)}} & H(a_1\dots a_m) \ar[d]^{H(\widehat{[\rho,x]}_{\vec a})} \\
    H(a_{x^{-1}\delta(1)})\dots H(a_{x^{-1}\delta(n)}) \ar[r]^-{\lambda_{\rho^{}_!x_\ast\vec a}} & H(a_{x^{-1}\delta(1)}\dots a_{x^{-1}\delta(n)}) }}
\end{equation}
\end{lemma}
\begin{proof}
Since the functor $H:\mathcal M\wr\widetilde{\mathbb E}_{\mathcal G}\to\mathcal N\wr\widetilde{\mathbb E}_{\mathcal G}$ preserves coCartesian lifts of inert morphisms, the morphism $H((\widehat\rho_i)_{\vec a}):H(\vec a)\to H(a_i)$ is coCartesian.
Thus, the first statement is obvious.
As for the second, it turns out that we only have to verify the commutativity of \eqref{eq:adapter-isom:sq} for inert morphisms of the forms $[\rho,e_m]$ and $[\mathrm{id},x]$.
In the first case, for each $1\le j\le n$, we have
\[
\begin{split}
(\widehat\rho_j)_{H(\rho^{}_!\vec a)}\circ H(\widehat\rho_{\vec a})\circ\lambda_{\vec a}
&= \lambda_{\rho^{}_!\vec a,j}\circ H((\widehat\rho_j)_{\rho^{}_!\vec a}\circ\widehat\rho_{\vec a})\circ\lambda_{\vec a} \\
&= \lambda_{\rho^{}_!\vec a,j}\circ H((\widehat\rho^{}_{\delta(j)})_{\vec a})\circ \lambda_{\vec a} \\
&= \lambda_{\rho^{}_!\vec a,j}\circ\lambda_{\vec a,\delta(j)}^{-1}\circ(\widehat\rho^{}_{\delta(j)})_{H(\vec a)}\circ\lambda_{\vec a} \\
&= \lambda_{\rho^{}_!\vec a,j}\circ(\widehat\rho^{}_{\delta(j)})^{}_{H(a_1)\dots H(a_m)} \\
&= (\widehat\rho_j)_{H(\rho^{}_!\vec a)}\circ \lambda_{\rho^{}_!\vec a}\circ\widehat\rho^{}_{H(a_1)\dots H(a_m)}\ .
\end{split}
\]
In view of \ref{req:inert-lift:plim} in \cref{prop:inert-lift}, this implies $H(\widehat\rho_{\vec a})\lambda_{\vec a}=\lambda_{\rho_!\vec a}\widehat\rho^{}_{H(a_1)\dots H(a_m)}$, and \eqref{eq:adapter-isom:sq} is commutative.

It remains to show the commutativity of \eqref{eq:adapter-isom:sq} in the case $\rho$ is the identity.
Similarly to the case above, we have
\begin{equation}
\label{eq:prf:adapter-isom:hatx}
\begin{split}
(\widehat\rho_j)_{H(x_\ast\vec a)}\circ H(\widehat x_{\vec a})\circ\lambda_{\vec a}
&= \lambda_{x_\ast\vec a,j}\circ H((\widehat\rho_j)_{x_\ast\vec a}\circ\widehat x_{\vec a})\circ\lambda_{\vec a} \\
&= \lambda_{x_\ast\vec a,j}\circ H(\widehat{[\rho_j,x]}_{\vec a})\circ\lambda_{\vec a}\ .
\end{split}
\end{equation}
In view of \label{ex:Icong:inert}, setting $\delta_j$ to be the section of $\rho_j$ for each $1\le j\le m$, we have
\[
\begin{split}
[\rho_j,x]
&= [\rho_{x^{-1}(j)},\rho^\ast_{x^{-1}(j)}\delta_{x^{-1}(j)}^\ast(x)] \\
&= [\mathrm{id}_{\llangle n\rrangle},\delta_{x^{-1}(j)}^\ast(x)]\circ \rho_{x^{-1}(j)} \\
&= \rho_{x^{-1}(j)}
\end{split}
\]
as morphisms in $\widetilde{\mathbb E}_{\mathcal G}$ since $\overbar{\operatorname{Kec}}^{\mathcal G}_{\mathrm{id}_{\llangle 1\rrangle}}=\operatorname{Kec}^{\mathcal G}_{\mathrm{id}_{\llangle 1\rrangle}} = \mathcal G(1)$.
Substituting it to \eqref{eq:prf:adapter-isom:hatx}, we get
\[
\begin{split}
(\widehat\rho_j)_{H(x_\ast\vec a)}\circ H(\widehat x_{\vec a})\circ\lambda_{\vec a}
&= \lambda_{x_\ast\vec a,j}\circ H((\widehat\rho_{x^{-1}(j)})_{\vec a})\circ\lambda_{\vec a} \\
&= \lambda_{x_\ast\vec a,j}\circ\lambda_{\vec a,x^{-1}(j)}^{-1}\circ(\widehat\rho_{x^{-1}(j)})^{}_{H(\vec a)}\circ\lambda_{\vec a} \\
&= \lambda_{x_\ast\vec a,j}\circ(\widehat\rho_{x^{-1}(j)})^{}_{H(a_1)\dots H(a_m)} \\
&= \lambda_{x_\ast\vec a,j}\circ\widehat{[\rho_j,x]}_{H(a_1)\dots H(a_m)} \\
&= (\widehat\rho_j)^{}_{H(x_\ast\vec a)}\circ\lambda_{x_\ast\vec a}\circ \widehat x^{}_{H(a_1)\dots H(a_m)}\ .
\end{split}
\]
Hence, \ref{req:inert-lift:plim} in \cref{prop:inert-lift} again implies the commutativity of \eqref{eq:adapter-isom:sq}.
\end{proof}

On the other hand, on the construction of a $\mathcal G$-symmetric multicategories from a category of algebraic $\mathcal G$-operators, we need to observe how coCartesian lifts of inert morphisms determine composition operations.
Notice that, if we denote by $\mu_n:\llangle n\rrangle\to\llangle 1\rrangle$ the active morphism with $\mu_n(\pm\infty)=\pm\infty$ and $\mu_n(i)=1$ for $1\le i\le n$, then, for a multicategory $\mathcal M$, the multihom-set $\mathcal M(a_1\dots a_n;a)$ is recovered from the category $\mathcal M\wr\widetilde{\mathbb E}_{\mathcal G}$ as
\[
\mathcal M(a_1\dots a_n;a)
\cong (\mathcal M\wr\widetilde{\mathbb E}_{\mathcal G})(a_1\dots a_n,a)_{\mu_n}\ ,
\]
here the right hand side is the set of morphisms $a_1\dots a_n\to a$ in $\mathcal M\wr\widetilde{\mathbb E}_{\mathcal G}$ covering $\mu_n$.

\begin{notation}
Given active morphisms $\nu_1:\llangle k_i\rrangle\to\llangle l_i\rrangle\in\nabla$ for $1\le i\le n$, we define an active morphism $\nu_1\diamond\dots\diamond\nu_n:\llangle k_1+\dots+k_n\rrangle\to\llangle l_1+\dots+l_n\rrangle\in\nabla$ to be the map
\[
\begin{multlined}[t]
\llangle k_1+\dots+k_n\rrangle
\cong \{-\infty\}\star\langle k_1\rangle\star\dots\star\langle k_n\rangle\star\{\infty\}
\\[.5ex]
\xrightarrow{\mathrm{id}\amalg\nu_1\amalg\dots\amalg\nu_n\amalg\mathrm{id}} \{-\infty\}\star\langle l_1\rangle\star\dots\star\langle l_n\rangle\star\{\infty\}
\cong \llangle l_1+\dots+l_n\rrangle\ ,
\end{multlined}
\]
here $\star$ is the join of ordered sets and all maps are order-preserving.
In particular, we write
\[
\mu_{\vec k}
:= \mu_{k_1}\diamond\dots\diamond\mu_{k_n}\ .
\]
On the other hand, we set $\rho^{(\vec k)}_i:\llangle k_1+\dots+k_n\rrangle\to\llangle k_i\rrangle\in\nabla$ to be the inert morphism with
\[
\rho^{(\vec k)}_i(j) =
\begin{cases}
-\infty & j\le \sum_{s<i}k_s\ , \\[.5ex]
j- \sum_{s<i}k_s & \sum_{s<i}k_s<j\le\sum_{s\le i}k_s\ , \\[.5ex]
\infty & j>\sum_{s\le i}k_s\ .
\end{cases}
\]
We will identify the morphisms above with their images in $\widetilde{\mathbb E}_{\mathcal G}$.
\end{notation}

Using the notation above, one can immediately see
\begin{gather}
\label{eq:joinact-comp}
(\nu'_1\nu_1\diamond\dots\diamond\nu'_n\nu_n)
= (\nu'_1\diamond\dots\diamond\nu'_n)\circ(\nu_1\diamond\dots\diamond\nu_n)\ ,
\\
\label{eq:joinact-inert}
\rho^{(\vec l)}_i\circ(\nu_1\diamond\dots\diamond\nu_n)
= \nu_i\circ\rho^{(\vec k)}_i\ ,
\end{gather}
for active morphisms $\nu_i:\llangle k_i\rrangle\to\llangle l_i\rrangle$ and $\nu'_i:\llangle l_i\rrangle\to\llangle m_i\rrangle$ with $\vec k=(k_1,\dots,k_n)$ and $\vec l=(l_1,\dots,l_n)$.

\begin{lemma}
\label{lem:oper-concat}
Let $\mathcal C\in\mathbf{Oper}'_{\mathcal G}$, and suppose we are given active morphisms $\nu_i:\llangle k_i\rrangle\to\llangle l_i\rrangle\in\nabla$ for $1\le i\le n$.
Put $\vec k=(k_1,\dots,k_n)$ and $\vec l=(l_1,\dots,l_n)$, and suppose in addition $(\widehat\rho^{(\vec k)}_i)_X:X\to X_i$ and $(\widehat\rho^{(\vec l)}_i)_Y:Y\to Y_i$ are coCartesian morphisms in $\mathcal C$ covering $\rho^{(\vec k)}_i$ and $\rho^{(\vec l)}_i$ respectively.
Then, there is a unique bijection
\[
\varpi:\prod_{i=1}^n\mathcal C(X_i,Y_i)_{\nu_i}
\to \mathcal C(X,Y)_{\nu_1\diamond\dots\diamond\nu_n}
\]
such that $(\widehat\rho^{(\vec l)}_i)_Y\circ\varpi(f_1,\dots,f_n)=f_i\circ(\widehat\rho^{(\vec k)}_i)$, where $\mathcal C(V,W)_\nu$ is the set of morphisms $V\to W\in\mathcal C$ covering $\nu$.
Moreover, if other active morphisms $\nu'_i:\llangle l_i\rrangle\to\llangle m_i\rrangle\in\nabla$ and coCartesian morphisms $(\widehat\rho^{(\vec m)})_Z:Z\to Z_i\in\mathcal C$ covering $\rho^{(\vec m)}$ are given for $1\le i\le n$, with $\vec m=(m_1,\dots,m_n)$, then the square below is commutative:
\begin{equation}
\renewcommand{\objectstyle}{\displaystyle}
\label{eq:oper-concat:comp}
\vcenter{
  \xymatrix@C=4em{
    \prod_{i=1}^n\mathcal C(Y_i,Z_i)_{\nu'_i}\times\mathcal C(X_i,Y_i)_{\nu_i} \ar[r]^-{\prod\mathrm{comp}} \ar[d]_{\varpi\times\varpi} & \prod_{i=1}^n\mathcal C(X_i,Z_i)_{\nu'_i\nu_i} \ar[d]^\varpi \\
    \mathcal C(Y,Z)_{\nu'_1\diamond\dots\diamond\nu'_n}\times\mathcal C(X,Y)_{\nu_1\diamond\dots\diamond\nu_n} \ar[r]^-{\mathrm{comp}} & \mathcal C(X,Z)_{(\nu'_1\nu_1)\diamond\dots\diamond(\nu'_n\nu_n)} }}
.
\end{equation}
\end{lemma}
\begin{proof}
The existence of $\varpi$ immediately follows from the equation $\rho_i\mu_{\vec k}=\mu_{k_i}\rho^{(\vec k)}_i$ and the universal property of the coCartesian morphisms.
The uniqueness is guaranteed by the property \ref{cond:catalgop:plim} of objects of $\mathbf{Oper}'_{\mathcal G}$.
To see the last statement, take morphisms $f_i:X_i\to Y_i$ and $g_i:Y_i\to Z_i$ covering $\nu_i$ and $\nu'_i$ respectively for each $1\le i\le n$.
Then, we have
\[
\begin{split}
(\widehat\rho^{(\vec m)}_i)_Z\circ\varpi(g_1,\dots,g_n)\circ\varpi(f_1,\dots,f_n)
&= g_i\circ(\widehat\rho^{(\vec l)}_i)_Y\circ\varpi(f_1,\dots,f_n) \\
&= g_if_i\circ(\widehat\rho^{(\vec k)}_i)_X\ ,
\end{split}
\]
so the uniqueness of $\varpi$ implies
\[
\varpi(g_1f_1,\dots,g_nf_n) = \varpi(g_1,\dots,g_n)\circ\varpi(f_1,\dots,f_n)\ .
\]
Hence, the commutativity of \eqref{eq:oper-concat:comp} follows.
\end{proof}

\begin{lemma}
\label{lem:oper-Gsym}
Let $\mathcal C$ be a category of algebraic $\mathcal G$-operators, and let $X\in\mathcal C_{\llangle m\rrangle}$ is an object together with coCartesian morphisms
\[
(\widehat\rho_i)_X:X\to X_i
\]
covering the inert morphism $\rho_i:\llangle m\rrangle\to\llangle 1\rrangle\in\nabla$ for $1\le i\le m$.
For an inert morphism $[\rho,x]:\llangle n\rrangle\to\llangle n\rrangle\in\widetilde{\mathbb E}_{\mathcal G}$, say $\delta$ is the section of $\rho$ in $\nabla$, suppose we are given an object $X'\in\mathcal C_{\llangle n\rrangle}$ together with coCartesian morphisms
\[
(\widehat\rho_j)_{X'}:X'\to X_{x^{-1}(\delta(n))}
\]
covering the inert morphism $\rho_j:\llangle n\rrangle\to\llangle 1\rrangle$ for $1\le j\le n$.
Then, there is a unique isomorphism $\widehat{[\rho, x]}_X:X\to X'\in\mathcal C$ covering the morphism $[\rho,x]$ which makes the diagram
\begin{equation}
\label{eq:oper-Gsym:comp}
\vcenter{
  \xymatrix@C=1pc{
    X \ar[rr]^{\widehat{[\rho,x]}_X} \ar[dr]_{(\widehat\rho_{x^{-1}(\delta(j))})_X} && X' \ar[dl]^{(\widehat\rho_j)_{X'}} \\
    & X_{x^{-1}(\delta(j))} & }}
\end{equation}
commutes for each $1\le j\le n$.
Moreover, $\widehat{[\rho,x]}_X$ is coCartesian.
\end{lemma}
\begin{proof}
According to the computation in \cref{ex:Icong-inert}, we have
\[
\rho_j\circ[\rho,x]
= [\rho_{\delta(j)},x]
= \rho_{x^{-1}(\delta(j))}
\]
so that the first statement directly follows from the property \ref{cond:catalgop:plim} of categories of algebraic $\mathcal G$-operators.
To prove the last, take coCartesian lifts $\widehat{[\rho,x]\mathrlap'}_X:X\to X''$ of $[\rho,x]$ along $X$ and $(\widehat\rho_j)_{X''}:X''\to X''_j$ of $\rho_j$ along $X''$ for $1\le j\le n$.
The computation above shows the composition $(\widehat\rho_j)_{X''}\circ\widehat{[\rho,x]\mathrlap'}_X$ is a coCartesian lift of $\rho_{x^{-1}(\delta(j))}$, so the uniqueness of coCartesian lifts enables us to assume $X''_j = X_{x^{-1}(\delta(j))}$ and the following diagram commutes:
\begin{equation}
\label{eq:prf:oper-Gsym}
\vcenter{
  \xymatrix@C=1pc{
    X \ar[rr]^{\widehat{[\rho,x]}'_X} \ar[dr]_{(\widehat\rho_{x^{-1}(\delta(j))})_X} && X'' \ar[dl]^{(\widehat\rho_j)_{X''}} \\
    & X_{x^{-1}(\delta(j))} & }}
\quad.
\end{equation}
We have two cones in $\mathcal C$ below
\[
\begin{gathered}
\left\{(\widehat\rho_j)_{X'}:X'\to X_{x^{-1}(\delta(j))}\right\}_{j=1}^n\ ,
\\
\left\{(\widehat\rho_j)_{X''}:X''\to X_{x^{-1}(\delta(j))}\right\}_{j=1}^n\ ,
\end{gathered}
\]
both of which consist of coCartesian morphisms and lie over the cone
\[
\left\{\rho_j:\llangle n\rrangle\to\llangle 1\rrangle\right\}_{j=1}^n
\]
in $\widetilde{\mathbb E}_{\mathcal G}$.
Then, the property \ref{cond:catalgop:plim} of categories of algebraic $\mathcal G$-operators implies there is a unique isomorphism $\theta:X''\to X'\in\mathcal C_{\llangle n\rrangle}$ such that $(\widehat\rho_j)_{X'}\theta=(\widehat\rho_j)_{X''}$.
In view of the uniqueness of the morphism $\widehat{[\rho,x]}_X$, we obtain
\[
\theta\circ\widehat{[\rho,x]}_X = \widehat{[\rho,x]\mathrlap'}_X\ .
\]
In particular, $\widehat{[\rho,x]}_X$ is isomorphic to a coCartesian morphism, so it is itself coCartesian.
\end{proof}

When we endow a $\mathcal G$-symmetric structure, it is good to have transfer.

\begin{lemma}
\label{lem:Gsym-transfer}
Let $\mathcal G$ be a group operad.
Suppose we are given a multicategory $\mathcal M$ and a category of algebraic $\mathcal G$-operators $\mathcal C$ together with an equivalence
\[
H:\mathcal M\wr\widetilde{\mathbb E}_{\mathcal G}
\xrightarrow\sim \mathcal C
\]
in the $2$-category $\mathbf{Oper}'_{\mathcal G}$.
Then, there is a unique $\mathcal G$-symmetric structure on $\mathcal M$ which makes $H$ into an equivalence in $\mathbf{Oper}^{\mathsf{alg}}_{\mathcal G}$.
\end{lemma}
\begin{proof}
For each $\mathcal D\in\mathbf{Oper}'_{\mathcal G}$ with $q':\mathcal D\to\widetilde{\mathbb E}_{\mathcal G}$, for $X,Y\in\mathcal D$, and for $[\varphi,x]:q'(X)\to q'(Y)\in\widetilde{\mathbb E}_{\mathcal G}$, we write $\mathcal D(X,Y)_{[\varphi,x]}$ the set of morphisms of $\mathcal D$ lying over $[\varphi,x]$.
Hence, in view of \cref{theo:tildeMG-compare}, we have canonical bijections
\begin{equation}
\label{eq:prf:Gsym-transfer:MC}
\begin{split}
\mathcal M(a_1\dots a_n;a)
&\cong (\mathcal M\wr\widetilde{\mathbb E}_{\mathcal G})(a_1\dots a_n,a)_{\mu_n} \\
&\xrightarrow[\cong]{H} \mathcal C(H(a_1\dots a_n),H(a))_{\mu_n}\ .
\end{split}
\end{equation}
for objects $a,a_i\in\mathcal M$.
On the other hand, since $\widetilde{\mathbb G}_{\mathcal G}\to\widetilde{\mathbb E}_{\mathcal G}$ is full and the identity on objects, the induced functor $H_{\widetilde{\mathbb G}_{\mathcal G}}:\mathcal M\wr\widetilde{\mathbb G}_{\mathcal G}\to\mathcal C\times_{\widetilde{\mathbb E}_{\mathcal G}}\widetilde{\mathbb G}_{\mathcal G}$ is also an equivalence in $\mathbf{Oper}'_{\mathcal G}$ in view of \cref{prop:free-Gsym}.
Thus, we also have bijections
\begin{equation}
\label{eq:prf:Gsym-transfer:MGC}
\begin{split}
(\mathcal M\rtimes\mathcal G)(a_1\dots a_n;a)
&\cong (\mathcal M\wr\widetilde{\mathbb G}_{\mathcal G})(a_1\dots a_n,a)_{\mu_n} \\
&\xrightarrow[\cong]{H_{{\widetilde{\mathbb G}}_{\mathcal G}}} (\mathcal C\times_{\widetilde{\mathbb E}_{\mathcal G}}\widetilde{\mathbb G}_{\mathcal G})(H(a_1\dots a_n),H(a))_{\mu_n}\ .
\end{split}
\end{equation}
Note that the internal presheaf structure $\mathpzc A_{\mathcal C}:\mathcal C\times_{\widetilde{\mathbb E}_{\mathcal G}}\widetilde{\mathbb G}_{\mathcal G}\to\mathcal C$ is the identity on objects; indeed, the following diagram commutes:
\begin{equation}
\label{eq:prf:Gsym-transfer:Cunit}
\vcenter{
  \xymatrix@C=1em{
    \mathcal C\times_{\widetilde{\mathbb E}_{\mathcal G}}\widetilde{\mathbb E}_{\mathcal G} \ar@{=}[dr] \ar[rr]^{\mathrm{Id}\times\iota} && \mathcal C\times_{\widetilde{\mathbb E}_{\mathcal G}}\widetilde{\mathbb G}_{\mathcal G} \ar[dl]^{\mathpzc A_{\mathcal C}} \\
    & \mathcal C & }}
\quad.
\end{equation}
Combining with the isomorphisms in \eqref{eq:prf:Gsym-transfer:MC} and \eqref{eq:prf:Gsym-transfer:MGC}, we now obtain a map
\begin{equation}
\label{eq:prf:Gsym-transfer:candidate}
\begin{multlined}
\mathpzc A_{\mathcal M}:(\mathcal M\rtimes\mathcal G)(a_1\dots a_n;a)
\cong (\mathcal C\times_{\widetilde{\mathbb E}_{\mathcal G}}\widetilde{\mathbb G}_{\mathcal G})(H(a_1\dots a_n),H(a))_{\mu_n} \\
\xrightarrow{\mathpzc A_{\mathcal C}} \mathcal C(H(a_1\dots a_n),H(a))_{\mu_n}
\cong \mathcal M(a_1\dots a_n;a)\ .
\end{multlined}
\end{equation}

We assert that the map \eqref{eq:prf:Gsym-transfer:candidate} gives a $\mathcal G$-symmetric structure on $\mathcal M$.
Notice that, the composition operation in $\mathcal M$ is recovered from $\mathcal C$ as follows: for each $a,a_i\in\objof\mathcal M$ and $\vec a^{(i)}=a^{(i)}_1\dots a^{(i)}_{k_i}$, the composition operation is given by
\begin{equation}
\label{eq:prf:Gsym-transfer:multcomp}
\begin{split}
&\mathcal M(a_1\dots a_n;a)\times\prod_{i=1}^n\mathcal M(\vec a^{(i)};a_i) \\
&\cong \mathcal C(H(a_1\dots a_n),H(a))_{\mu_n}\times\prod_{i=1}^n\mathcal C(H(\vec a^{(i)}),H(a_i))_{\mu_{k_i}} \\
&\xrightarrow[\cong]{\mathrm{id}\times\varpi} \mathcal C(H(a_1\dots a_n),H(a))_{\mu_n}\times \mathcal C(H(\vec a^{(1)}\dots\vec a^{(n)}),H(a_1\dots a_n))_{\mu_{\vec k}} \\
&\xrightarrow{\mathrm{comp.}} \mathcal C(H(\vec a^{(1)}\dots\vec a^{(n)}),H(a))_{\mu_{k_1+\dots+k_n}} \\
&\cong\mathcal M(\vec a^{(1)}\dots\vec a^{(n)};a)\ ,
\end{split}
\end{equation}
where $\varpi$ is the bijection in \cref{lem:oper-concat} with respect to the image by $H$ of the standard coCartesian lifts
\[
(\widehat\rho_i)_{\vec a}:\vec a\to a_i
\ ,\quad
(\widehat\rho^{(\vec k)}_i)_{\vec a^{(1)}\dots\vec a^{(n)}}:\vec a^{(1)}\dots\vec a^{(n)}\to \vec a^{(i)}
\ \in\mathcal M\wr\widetilde{\mathbb E}_{\mathcal G}\ .
\]
Similarly, the composition in $\mathcal M\rtimes\mathcal G$ is also recovered from $\mathcal C\times_{\widetilde{\mathbb E}_{\mathcal G}}\widetilde{\mathbb G}_{\mathcal G}$.
Moreover, thanks to the choice of the coCartesian lifts of $\rho_i$ and $\rho^{(\vec k)}_i$, the square below is commutative:
\[
\vcenter{
  \xymatrix{
    \prod_{i=1}^n(\mathcal C\times_{\widetilde{\mathbb E}_{\mathcal G}}\widetilde{\mathbb G}_{\mathcal G})(H(\vec a^{(i)}),H(a_i))_{\mu_{k_i}}  \ar[r]^-{\prod\mathpzc A_{\mathcal C}} \ar[d]_{\varpi} & \prod_{i=1}^n\mathcal C(H(\vec a^{(i)}),H(a_i))_{\mu_{k_i}} \ar[d]^\varpi \\
    (\mathcal C\times_{\widetilde{\mathbb E}_{\mathcal G}}\widetilde{\mathbb G}_{\mathcal G})(H(\vec a^{(1)}\dots\vec a^{(n)}),H(\vec a))_{\mu_{\vec k}} \ar[r]^-{\mathpzc A_{\mathcal C}} & \mathcal C(H(\vec a^{(1)}\dots\vec a^{(n)}),H(\vec a))_{\mu_{\vec k}} }}
\]
This together with the functoriality of $\mathpzc A_{\mathcal C}$ implies that the map \eqref{eq:prf:Gsym-transfer:candidate} actually defines a multifunctor $\mathpzc A_{\mathcal M}:\mathcal M\rtimes\mathcal G\to\mathcal M$.
It furthermore turns out that $\mathpzc A_{\mathcal M}$ is actually a $\mathcal G$-symmetric structure on $\mathcal M$; the unitality and the associativity follow from the corresponding axioms for the internal presheaf structure on $\mathcal C$.

Finally, we see $H$ respects the internal presheaf structures over $\widetilde{\mathbb G}_{\mathcal G}\rightrightarrows\widetilde{\mathbb E}_{\mathcal G}$.
In view of \cref{theo:tildeMG-compare} and the property \ref{cond:catalgop:plim} of categories of algebraic $\mathcal G$-symmetric operators, it suffices to show that, for each morphism $[\varphi;f_1,\dots,f_n;u,x]:\vec a\to\vec b=b_1\dots b_n$ in $\mathcal M\wr\widetilde{\mathbb E}_{\mathcal G}$, we have
\begin{equation}
\label{eq:prf:Gsym-transfer:Gequiv}
\begin{multlined}
H((\widehat\rho_j)_{\vec b})\circ H([\varphi;\mathpzc A_{\mathcal M}(f_1,\delta^{(\varphi)\ast}_1(u)),\dots\mathpzc A_{\mathcal M}(f_n,\delta^{(\varphi)\ast}_n(u));x]) \\
= H((\widehat\rho_j)_{\vec b})\circ \mathpzc A_{\mathcal C}H_{\widetilde{\mathbb G}_{\mathcal G}}([\varphi;f_1,\dots,f_n;u,x])
\end{multlined}
\end{equation}
for each $1\le j\le n$.
Let $\varphi=\mu\rho$ be the factorization with $\mu$ active and $\rho$ inert, so \ref{req:conginr:inert} in \cref{lem:conginr} allows us to assume $u=\rho^\ast(\overbar u)$ with $\overbar u\in\overbar{\operatorname{Dec}}^{\mathcal G}_{\mu}$.
If we put $\mu=\mu_{\vec k}$, then the left hand side of \eqref{eq:prf:Gsym-transfer:Gequiv} is computed as
\begin{equation}
\label{eq:prf:Gsym-transfer:GequivL}
\begin{split}
&H((\widehat\rho_j)_{\vec b})\circ H([\varphi;\mathpzc A_{\mathcal M}(f_1,\delta^{(\varphi)\ast}_1(u)),\dots\mathpzc A_{\mathcal M}(f_n,\delta^{(\varphi)\ast}_n(u));x]) \\
&= H([\mu_{k_j};\mathpzc A_{\mathcal M}(f_j,\delta^{(\mu)\ast}_j(\overbar u));e])\circ H(\widehat\rho^{(\vec k)}_j\widehat{[\rho, x]}_{\vec a}) \\
&= \mathpzc A_{\mathcal C}H_{\widetilde{\mathbb G}_{\mathcal G}}([\mu_{k_j};f_j;\delta^{(\mu)\ast}_j(\overbar u),e])\circ H(\widehat\rho^{(\vec k)}_j\widehat{[\rho, x]}_{\vec a})\ .
\end{split}
\end{equation}
On the other hand, according to \eqref{eq:prf:Gsym-transfer:Cunit}, for every standard coCartesian lift $\widehat{[\rho',x']}$ of an inert morphism in $\mathcal M\wr\widetilde{\mathbb E}_{\mathcal G}$, one has
\[
H(\widehat{[\rho',x']})
= \mathpzc A_{\mathcal C} H_{\widetilde{\mathbb G}_{\mathcal G}}(\widehat{[\rho',x']})\ ,
\]
here we identify $\widehat{[\rho',x']}$ with its image in $\mathcal M\wr\widetilde{\mathbb E}_{\mathcal G}$ using \cref{lem:free-cocart}.
Thus, the right hand side of \eqref{eq:prf:Gsym-transfer:Gequiv} is given by
\begin{equation}
\label{eq:prf:Gsym-transfer:GequivR}
\begin{split}
& H((\widehat\rho_j)_{\vec b})\circ \mathpzc A_{\mathcal C}H_{\widetilde{\mathbb G}_{\mathcal G}}([\varphi;f_1,\dots,f_n;u,x]) \\
&= \mathpzc A_{\mathcal C} H_{\widetilde{\mathbb G}_{\mathcal G}}(\widehat{\rho_j}_{\vec b}\circ[\mu_{\vec k};f_1,\dots, f_n;\overbar u,e]\circ \widehat{[\rho,x]}_{\vec a}) \\
&= \mathpzc A_{\mathcal C} H_{\widetilde{\mathbb G}_{\mathcal G}}([\mu_{k_j};f_j;\delta^{(\mu)\ast}_j(\overbar u),e])\circ H(\widehat\rho^{(\vec k)}_j\widehat{[\rho,x]}_{\vec a})\ .
\end{split}
\end{equation}
Now, \eqref{eq:prf:Gsym-transfer:GequivL} and \eqref{eq:prf:Gsym-transfer:GequivR} give rise tot the equation \eqref{eq:prf:Gsym-transfer:Gequiv}, and it shows $H$ is a $1$-morphism in $\mathbf{Oper}^{\mathsf{alg}}_{\mathcal G}$.
The uniqueness is obvious by construction.
\end{proof}

\begin{proof}[Proof of \cref{theo:multcat-oper-eq}]
In order to show \ref{req:multcat-oper-eq:ff}, we construct an inverse $\Theta$ of \eqref{eq:multcat-oper-eq:homcat}.
Fix $\mathcal G$-symmetric multicategories $\mathcal M$ and $\mathcal N$.
Suppose $H:\mathcal M\wr\widetilde{\mathbb E}_{\mathcal G}\to\mathcal N\wr\widetilde{\mathbb E}_{\mathcal G}$ is a map of algebraic $\mathcal G$-operators, and take the family
\[
\{\lambda_{\vec a}=[\mathrm{id};\lambda_{\vec a,1},\dots,\lambda_{\vec a,m};e_m]\}_{\vec a=a_1\dots a_m}
\]
of morphisms in $\mathcal N\wr\widetilde{\mathbb E}_{\mathcal G}$ as in \cref{lem:adapter-isom}.
Note that $H$ induces a functor $\underline H:\underline{\mathcal M}\to\underline{\mathcal N}$ between underlying categories via the restriction to the fibers over $\llangle 1\rrangle\in\widetilde{\mathbb E}_{\mathcal G}$.
On the other hand, if $f\in\mathcal M(\vec a;b)$ is a multimorphism, we can take a unique multimorphism $H^\circ(f)\in\mathcal N(H(\vec a);H(b))$ so that
\[
H([\mu_m;f;e_m]) = [\mu_m;H^\circ(f);e_m]\ .
\]

We define a multifunctor $\Theta(H):\mathcal M\to\mathcal N$ as follows: for each $a\in\operatorname{Ob}\mathcal M$, we set $\Theta(H)(a):=\underline H(a)$.
For each $\vec a=a_1\dots a_m$ and each $b$, define
\[
\begin{array}{rccc}
  \Theta(H)\ \mathrlap:& \mathcal M(\vec a;b) &\mathclap\to& \mathcal N(\underline H(a_1)\dots\underline H(a_m);\underline H(b)) \\[1ex]
  & f &\mathclap\mapsto& \gamma_{\mathcal N}(H^\circ(f);\lambda_{\vec a,1},\dots,\lambda_{\vec a,m})
\end{array}
\quad.
\]
Since $\lambda_{a,1}$ is the identity for $a\in\underline{\mathcal M}$, $\Theta(H)$ preserves the identities, so we show the multifunctoriality.
For $f_i\in\mathcal M(a^{(i)}_1\dots a^{(i)}_{k_i};a_i)$ for $1\le i\le m$, thanks to the commutative square \eqref{eq:adapter-isom:sq}, we have
\[
\begin{split}
& (\widehat\rho_i)_{H(\vec a)}\circ H([\mu_{\vec k};f_1,\dots,f_m;e])\circ\lambda_{\vec a^{(1)}\dots\vec a^{(m)}} \\
&= \lambda_{\vec a,i}\circ H((\widehat\rho_i)_{\vec a}\circ[\mu_{\vec k};f_1,\dots,f_m;e])\circ\lambda_{\vec a^{(1)}\dots\vec a^{(m)}} \\
&= \lambda_{\vec a,i}\circ H([\mu_{k_i};f_i;e]\circ(\widehat\rho^{(\vec k)}_i)_{\vec a^{(1)}\dots\vec a^{(m)}})\circ\lambda_{\vec a^{(1)}\dots\vec a^{(m)}} \\
&= \lambda_{\vec a,i}\circ [\mu_{k_i};H^\circ(f_i);e]\circ \lambda_{\vec a^{(i)}}\circ(\widehat\rho^{(\vec k)}_i)_{H(\vec a^{(1)})\dots H(\vec a^{(m)})} \\
&= \lambda_{\vec a,i}\circ[\mu_{k_i};\Theta(H)(f_i);e] \\
&= (\widehat\rho_i)_{H(\vec a)}\circ\lambda_{\vec a}\circ[\mu_{\vec k};\Theta(H)(f_1),\dots,\Theta(H)(f_m);e]\ ,
\end{split}
\]
which, by virtue of the property \ref{req:inert-lift:plim} in \cref{prop:inert-lift}, implies the square below is commutative:
\begin{equation}
\label{eq:prf:multcat-oper-eq:natlambda}
\vcenter{
  \xymatrix@C=4em{
    H(a^{(1)}_1)\dots H(a^{(1)}_{k_1})\dots H(a^{(m)}_{k_m}) \ar[r]^-{\lambda_{\vec a^{(1)}\dots\vec a^{(m)}}} \ar[d]_{[\mu_{\vec k};\Theta(H)(f_1),\dots,\Theta(H)(f_m);e]} & H(\vec a^{(1)}\dots\vec a^{(m)}) \ar[d]^{H([\mu_{\vec k};f_1,\dots,f_m;e])} \\
    H(a_1)\dots H(a_m) \ar[r]^{\lambda_{\vec a}} & H(a_1\dots a_m) }}
\end{equation}
Therefore we obtain
\[
\begin{split}
& [\mu_{k_1+\dots+k_m};\gamma_{\mathcal N}(\Theta(H)(f);\Theta(H)(f_1),\dots,\Theta(H)(f_m));e] \\
&= [\mu_m;\Theta(H)(f);e]\circ[\mu_{\vec k};\Theta(H)(f_1),\dots,\Theta(H)(f_m);e] \\
&= [\mu_m;H^\circ(f);e]\circ\lambda_{\vec a}\circ[\mu_{\vec k};\Theta(H)(f_1),\dots,\Theta(H)(f_m);e] \\
&= H([\mu_m;f;e]\circ[\mu_{\vec k};f_1,\dots,f_m;e])\circ\lambda_{\vec a^{(1)}\dots\vec a^{(m)}} \\
&= H([\mu_{k_1+\dots+k_m};\gamma_{\mathcal M}(f;f_1,\dots,f_m);e]\circ\lambda_{\vec a^{(1)}\dots\vec a^{(m)}} \\
&= [\mu_{k_1+\dots+k_m};H^\circ(\gamma_{\mathcal M}(f;f_1,\dots,f_m));e]\circ\lambda_{\vec a^{(1)}\dots\vec a^{(m)}} \\
&= [\mu_{k_1+\dots+k_m};\Theta(H)(\gamma_{\mathcal M}(f;f_1,\dots,f_m));e]\ ,
\end{split}
\]
and the multifunctoriality of $\Theta(H)$ follows.
Furthermore, $\Theta(H):\mathcal M\to\mathcal N$ is $\mathcal G$-symmetric.
To see this, notice that, for $f\in\mathcal M(x_\ast\vec a,b)$, we have
\[
\begin{split}
H([\mu_m;f;x])
&= [\mu_m;H^\circ(f);e]\circ H(\widehat x_{\vec a}) \\
&= [\mu_m;H^\circ(f);e]\circ \lambda_{x_\ast\vec a}\circ\widehat x_{H(\vec a)}\circ\lambda_{\vec a}^{-1} \\
&= [\mu_m;\Theta(H)(f);x]\circ\lambda_{\vec a}^{-1} \\
&= \left[\mu_m;\gamma_{\mathcal N}\left(\Theta(H)(f);\lambda_{\vec a,x^{-1}(1)}^{-1},\dots,\lambda_{\vec a,x^{-1}(m)}\right);x\right]\ .
\end{split}
\]
It follows that, for the induced functor $H_{\mathbb G}:\mathcal M\wr\mathbb G_{\mathcal G}\to\mathcal N\wr\mathbb G_{\mathcal G}$, we have
\[
H_{\mathbb G}([\mu_m;f;x,e_m])
= \left[\mu_m;\gamma_{\mathcal N}\left(\Theta(H)(f);\lambda_{\vec a,x^{-1}(1)}^{-1},\dots,\lambda_{\vec a,x^{-1}(m)}^{-1}\right);x,e_m\right]\ .
\]
Since $H$ is a map of internal presheaves, we obtain
\[
\begin{split}
H([\mu_m;f^x;e_m])
&= \left[\mu_m;\gamma_{\mathcal N}\left(\Theta(H)(f);\lambda_{\vec a,x^{-1}(1)}^{-1},\dots,\lambda_{\vec a,x^{-1}(m)}^{-1}\right)^x;e_m\right] \\
&= \left[\mu_m;\gamma_{\mathcal N}\left(\Theta(H)(f)^x;\lambda_{\vec a,1}^{-1},\dots,\lambda_{\vec a,m}^{-1}\right);e_m\right] \\
&= [\mu_m;\Theta(H)(f)^x;e_m]\circ\lambda_{\vec a}^{-1}
\end{split}
\]
and so $\Theta(H)(f^x)=\Theta(H)(f)^x$.

We extend $\Theta$ to an actual functor
\[
\Theta:
\mathbf{Oper}^{\mathsf{alg}}_{\mathcal G}(\mathcal M\wr\widetilde{\mathbb E}_{\mathcal G},\mathcal N\wr\widetilde{\mathbb E}_{\mathcal G})
\to\mathbf{MultCat}_{\mathcal G}(\mathcal M,\mathcal N)
\]
as follows: note that, in view of \cref{lem:inrpresh-locfull}, for $1$-morphisms $H,K:\mathcal M\wr\widetilde{\mathbb E}_{\mathcal G}\to\mathcal N\wr\widetilde{\mathbb E}_{\mathcal G}\in\mathbf{Oper}^{\mathsf{alg}}_{\mathcal G}$, a $2$-morphism $\xi:H\to K$ is nothing but a natural transformation over $\widetilde{\mathbb E}_{\mathcal G}$.
We set
\begin{equation}
\label{eq:prf:multcat-oper-eq:Theta2}
\Theta(\xi):=\left\{\xi_a:H(a)\to K(a)\right\}_{a\in\mathcal M}\ .
\end{equation}
To see \eqref{eq:prf:multcat-oper-eq:Theta2} forms a multinatural transformation $\Theta(H)\to\Theta(K)$, notice that, for each $\vec a=a_1\dots a_n\in\mathcal M\wr\widetilde{\mathbb E}_{\mathcal G}$, the naturality of $\xi$ implies the square
\[
\vcenter{
  \xymatrix{
    H(a_1\dots a_n) \ar[r]^{\xi_{\vec a}} \ar[d]_{H((\widehat\rho_i)_{\vec a})} & K(a_1\dots a_n) \ar[d]^{K((\widehat\rho_i)_{\vec a})} \\
    H(a_i) \ar[r]^{\xi_{a_i}} & K(a_i) }}
\]
is commutative for each $1\le i\le n$.
Computing the compositions, one obtains
\[
\xi_{\vec a} = \lambda^{(K)}_{\vec a}\circ [\mathrm{id};\xi_{a_1},\dots,\xi_{a_n};e_n]\circ\lambda^{(H)-1}_{\vec a}\ ,
\]
where $\lambda^{(H)}$ and $\lambda^{(K)}$ are the ones in \cref{lem:adapter-isom} for functors $H$ and $K$ respectively.
Then, the multinaturality of \eqref{eq:prf:multcat-oper-eq:Theta2} is straightforward.

We verify $\Theta$ is actually an inverse to the functor \eqref{eq:multcat-oper-eq:homcat}.
If $F:\mathcal M\to\mathcal N$ is a $\mathcal G$-symmetric multifunctor, then
the $\mathcal G$-symmetric multifunctor $\Theta(\widetilde F^{\mathcal G})$ is exactly $F$ itself since $\lambda_{\vec a}$ is trivial in this case.
On the other hand, in view of \eqref{eq:adapter-isom:sq} and \eqref{eq:prf:multcat-oper-eq:natlambda}, the family $\lambda=\{\lambda_{\vec a}\}_{\vec a}$ forms a natural isomorphism $\widetilde{\Theta(H)}^{\mathcal G}\cong H$.
The uniqueness of $\lambda$ implies it is natural with respect to $H$.
Hence, \eqref{eq:multcat-oper-eq:homcat} is an equivalence of categories, and we have finished the proof of the part \ref{req:multcat-oper-eq:ff}.

Finally, we show the part \ref{req:multcat-oper-eq:surj}.
Let $q:\mathcal C\to\widetilde{\mathbb E}_{\mathcal G}$ be a category of algebraic $\mathcal G$-operators.
By virtue of \cref{lem:Gsym-transfer}, in order to see $\mathcal C$ lies in the essential image of $(\blank)\wr\widetilde{\mathbb E}_{\mathcal G}$, it suffices to show there is a multicategory $\mathcal M$ together with an equivalence $\mathcal M\wr\widetilde{\mathbb E}_{\mathcal G}\xrightarrow\sim\mathcal C$ in the $2$-category $\mathbf{Oper}'_{\mathcal G}$.
For each finite word $\vec W=W_1\dots W_n$ of objects in $\mathcal C$ with, say, $q(W_i)=\llangle k_i\rrangle$, the property \ref{cond:catalgop:prod} of categories of algebraic $\mathcal G$-operators allows us to take an object $\varpi(\vec W)\in\mathcal C_{\llangle k_1+\dots+k_n\rrangle}$ together with coCartesian morphisms
\[
(\widehat\rho^{(\vec k)}_i)_{\vec W}:\varpi(\vec W)\to W_i
\]
covering the inert morphism $\rho^{(\vec k)}_i:\llangle k_1+\dots+k_n\rrangle\to\llangle k_i\rrangle\in\nabla$.
In the following argument, we fix such data for each $\vec W$.
Note that the coincidence of the symbol $\varpi$ here and in \cref{lem:oper-concat} is intentional; for $\vec V=V_1\dots V_n$ with $V_i\in\mathcal C$ and $q(V_i)=\llangle l_i\rrangle$, and for active morphisms $\nu_i:\llangle k_i\rrangle\to\llangle l_i\rrangle\in\nabla$, we have a map
\[
\varpi:\prod_{i=1}^n\mathcal C(W_i,V_i)_{\nu_i}\to\mathcal C(\varpi(W_1\dots W_n),\varpi(V_1\dots V_n))_{\nu_1\diamond\dots\diamond\nu_n}
\]
so that
\[
(\widehat\rho^{(\vec l)}_i)_{\vec V}\circ\varpi(f_1,\dots,f_n)
=f_i\circ(\widehat\rho^{(\vec k)}_i)_{\vec W}\ .
\]
Put $\underline{\mathcal C}:=\mathcal C_{\llangle 1\rrangle}$.
Note that, for $\vec X^{(i)}=X^{(i)}_1\dots X^{(i)}_{k_i}$ with $X^{(i)}_j\in\underline{\mathcal C}$, \cref{lem:oper-Gsym} asserts that there is a unique isomorphism
\[
\theta:\varpi(\varpi(\vec X^{(1)})\dots\varpi(\vec X^{(n)}))
\cong \varpi(\vec X^{(1)}\dots\vec X^{(n)})
\in\mathcal C_{\llangle k_1+\dots+k_n\rrangle}
\]
which makes the square below commutes:
\[
\vcenter{
  \xymatrix{
    \varpi(\varpi(\vec X^{(1)})\dots\varpi(\vec X^{(n)})) \ar[r]^-{\theta}_-{\cong} \ar[d]_{\widehat\rho^{(\vec k)}_i} & \varpi(\vec X^{(1)}\dots\vec X^{(n)}) \ar[d]^{\widehat\rho_{k_1+\dots+k_{i-1}+j}} \\
    \varpi(\vec X^{(i)}) \ar[r]^{\widehat\rho_j} & X^{(i)}_j }}
\quad.
\]
The property \ref{cond:catalgop:plim} also implies that, for morphisms $f^{(i)}_j:X^{(i)}_j\to Y^{(i)}_j\in\underline{\mathcal C}$ for $1\le i\le n$ and $1\le j\le k_n$, the following square is also commutative:
\begin{equation}
\label{eq:prf:multcat-oper-eq:thetanat}
\vcenter{
  \xymatrix{
    \varpi(\varpi(\vec X^{(1)}),\dots,\varpi(\vec X^{(n)})) \ar[d]_{\varpi(\varpi(\vec f^{(1)}),\dots,\varpi(\vec f^{(n)}))} \ar[r]^-\theta & \varpi(\vec X^{(1)}\dots\vec X^{(n)}) \ar[d]^{\varpi(f^{(1)}_1,\dots,f^{(1)}_{k_1},\dots,f^{(n)}_{k_n})} \\
    \varpi(\varpi(\vec Y^{(1)}),\dots,\varpi(\vec Y^{(n)})) \ar[r]^-\theta & \varpi(\vec Y^{(1)}\dots\vec Y^{(n)}) }}
\end{equation}
In addition, the uniqueness of $\theta$ guarantees that it makes the diagram below commute:
\begin{equation}
\label{eq:prf:multcat-oper-eq:thetaassoc}
\vcenter{
  \xymatrix@!C=4pc@R=1.5pc{
    &\varpi(\varpi(\varpi(\vec X^{(1,1)})\dots\varpi(\vec X^{(1,r_1)}))\dots\varpi(\varpi(\vec X^{(n,1)})\dots\varpi(\vec X^{(n,r_n)}))) \ar[dl]_{\theta} \ar[ddr]^{\varpi(\theta,\dots,\theta)} & \\
    \varpi(\varpi(\vec X^{(1,1)})\dots\varpi(\vec X^{(1,r_1)})\dots\varpi(\vec X^{(n,r_n)})) \ar[ddr]_{\theta} && \\
    && \varpi(\varpi(\vec X^{(1,1)}\dots\vec X^{(1,r_1)})\dots\varpi(\vec X^{(n,r_n)})) \ar[dl]^{\theta} \\
    &\varpi(\vec X^{(1,1)}\dots\vec X^{(1,r_1)}\dots\vec X^{(n,r_n)}) & }}
\end{equation}

We now define a multicategory $\mathcal M_{\mathcal C}$ so that
\begin{itemize}
  \item objects are those in $\underline{\mathcal C}$;
  \item for $X,X_1,\dots,X_m\in\underline{\mathcal C}$, we set
\[
\mathcal M_{\mathcal C}(X_1\dots X_m;X)
:= \mathcal C(\varpi(X_1\dots X_m),X)_{\mu_m}\ ;
\]
  \item the composition is given by
\[
\begin{split}
\gamma:&\mathcal M_{\mathcal C}(X_1\dots X_n;X)\times\prod_{i=1}^n\mathcal M_{\mathcal C}(\vec X^{(i)};X_i) \\
&\cong\mathcal C(\varpi(X_1\dots X_n),X)_{\mu_n}\times\prod_{i=1}^n\mathcal C(\varpi(\vec X^{(i)}),X_i)_{\mu_{k_i}} \\
&
\begin{multlined}
\xrightarrow{\mathrm{id}\times\varpi} \mathcal C(\varpi(X_1\dots X_n),X)_{\mu_n} \\
\phantom{\mathrm{id}\times\varpi}\times\mathcal C(\varpi(\varpi(\vec X^{(1)})\dots\varpi(\vec X^{(n)})),\varpi(X_1\dots X_n))_{\mu_{\vec k}}
\end{multlined}
\\
& \xrightarrow{\mathrm{comp.}} \mathcal C(\varpi(\varpi(\vec X^{(1)})\dots\varpi(\vec X^{(n)})),X)_{\mu_{k_1+\dots+k_n}} \\
& \xrightarrow{\theta^{-1\ast}}\mathcal C(\varpi(\vec X^{(1)}\dots\vec X^{(n)}),X)_{\mu_{k_1+\dots+k_n}}\ ;
\end{split}
\]
in other words, we have
\[
\gamma(f;f_1,\dots,f_n)
= f\circ\varpi(f_1,\dots,f_n)\circ\theta^{-1}\ .
\]
\end{itemize}
The associativity of the composition is verified as follows: take morphisms $f\in\mathcal M_{\mathcal C}(X_1\dots X_n;X)$, $f_i\in\mathcal M_{\mathcal C}(X^{(i)}_1\dots X^{(i)}_{k_i};X_i)$, and $f^{(i)}_j\in\mathcal M_{\mathcal C}(\vec X^{(i,j)};X^{(i)}_j)$ for $1\le i\le n$ and $1\le j\le k_i$.
Then, thanks to the commutative squares \eqref{eq:oper-concat:comp}, \eqref{eq:prf:multcat-oper-eq:thetanat}, and \eqref{eq:prf:multcat-oper-eq:thetaassoc}, we have
\[
\begin{split}
&\gamma(f;\gamma(f_1;f^{(1)}_1,\dots,f^{(1)}_{k_1}),\dots,\gamma(f_n;f^{(n)}_1,\dots,f^{(n)}_{k_n})) \\
&= f\circ\varpi(f_1\circ\varpi(f^{(1)}_1,\dots,f^{(1)}_{k_1})\circ\theta^{-1},\dots,f_n\circ\varpi(f^{(n)}_1,\dots,f^{(n)}_{k_n})\circ\theta^{-1})\circ\theta^{-1} \\
&=
\begin{multlined}[t]
f\circ\varpi(f_1,\dots,f_n) \\
\circ\varpi(\varpi(f^{(1)}_1,\dots,f^{(1)}_{k_1}),\dots,\varpi(f^{(n)}_1,\dots,f^{(n)}_{k_n}))\circ\varpi(\theta,\dots,\theta)^{-1}\circ\theta^{-1}
\end{multlined}
\\
&= \gamma(f;f_1,\dots,f_n)\circ\varpi(f^{(1)}_1,\dots,f^{(1)}_{k_1},\dots,f^{(n)}_{k_n})\circ\theta^{-1} \\
&= \gamma(\gamma(f;f_1,\dots,f_n);f^{(1)}_1,\dots,f^{(1)}_{k_1},\dots,f^{(n)}_{k_n})\ ,
\end{split}
\]
which implies the associativity of the composition.
The unitality is obvious so that $\mathcal M_{\mathcal C}$ is actually a multicategory.

We define a functor $P:\mathcal M_{\mathcal C}\wr\widetilde{\mathbb E}_{\mathcal G}\to\mathcal C$ as follows: for each object $X_1\dots X_m\in\mathcal M_{\mathcal C}\wr\widetilde{\mathbb E}_{\mathcal G}$, put
\[
P(X_1\dots X_m)
:= \varpi(X_1\dots X_m)\ .
\]
As for a morphism $[\varphi;f_1,\dots,f_n;x]:X_1\dots X_m\to Y_1\dots Y_n$, taking the factorization $\varphi=\mu\rho$ with $\mu$ active and $\rho$ inert, we set
\begin{equation}
\label{eq:prf:multcat-oper-eq:Pmor}
P([\varphi;f_1,\dots,f_n;x])
:= \varpi(f_1,\dots,f_n)\circ\theta^{-1}\circ\widehat{[\rho,x]}^{}_{X_1\dots X_m}\ ,
\end{equation}
where $\widehat{[\rho,x]}^{}_{X_1\dots X_m}:\varpi(X_1\dots X_m)\to\varpi(X_{x^{-1}(1)}\dots X_{x^{-1}(m)})$ is the coCartesian lift of $[\rho,x]$ given in \cref{lem:oper-Gsym}.
The functoriality of $P$ directly follows from the uniqueness of each morphisms in the right hand side of \eqref{eq:prf:multcat-oper-eq:Pmor}.
It is clear that $P$ preserves coCartesian lifts of inert morphisms, so $P$ is a $1$-morphism in $\mathbf{Oper}'_{\mathcal G}$.
In addition, the property \ref{cond:catalgop:prod} of categories of algebraic $\mathcal G$-operators implies $P$ is essentially surjective.
Hence, in order to see $P$ is an equivalence, it remains to show it is fully faithful.
Consider an arbitrary morphism in $\mathcal C$ of the form
\[
h:\varpi(X_1\dots X_m)\to\varpi(Y_1\dots Y_n)
\]
covering $[\varphi,x]:\llangle m\rrangle\to\llangle n\rrangle\in\widetilde{\mathbb E}_{\mathcal G}$.
If $\varphi=\mu\rho$ is the factorization with $\mu$ active and $\rho$ inert, say $\delta$ is the section of $\rho$, the universal property of the coCartesian morphism $\widehat{[\rho,x]}^{}_{X_1\dots X_m}$ given in \cref{lem:oper-Gsym} implies that there is a unique morphism $h':\varpi(X_{x^{-1}(\delta(1))}\dots X_{x^{-1}(\delta(m))})\to\varpi(Y_1\dots Y_n)\in\mathcal C$ covering $\mu$ such that
\[
h = h'\circ\widehat{[\rho,x]}_{X_1\dots X_m}
\]
In view of \cref{lem:oper-concat}, $h'$ can be uniquely written as $h'=\varpi(h_1,\dots,h_m)$.
This implies $P$ is fully faithful, and this completes the proof.
\end{proof}

\bibliographystyle{plain}
\bibliography{mybiblio}

\begin{thebibliography}{10}

\bibitem{BataninMarkl2014}
M.~Batanin and M.~Markl.
\newblock Crossed interval groups and operations on the hochschild cohomology.
\newblock {\em Journal of Noncommutative Geometry}, pages 655--693, 2014.

\bibitem{Berger2007}
C.~Berger.
\newblock Iterated wreath product of the simplex category and iterated loop
  spaces.
\newblock {\em Advances in Mathematics}, 213(1):230--270, 2007.

\bibitem{BlackwellKellyPower1989}
R.~Blackwell, G.~M. Kelly, and A.J. Power.
\newblock Two-dimensional monad theory.
\newblock {\em Journal of Pure and Applied Algebra}, 59(1):1--41, 1989.

\bibitem{Bor94I}
F.~Borceux.
\newblock {\em Handbook of Categorical Algebra 1}.
\newblock Number~50 in Encyclopedia of Mathematics and its Applications.
  Cambridge University Press, 1994.

\bibitem{CG13}
A.~S. Corner and N.~Gurski.
\newblock Operads with general groups of equivariance, and some $2$-categorical
  aspects of operads in cat.
\newblock arXiv:1312.5910, 2013.

\bibitem{FL91}
Z.~Fiedorowicz and J-L Loday.
\newblock Crossed simplicial groups and their associated homology.
\newblock {\em Transactions of the American Mathematical Society},
  326(1):57--87, 1991.

\bibitem{FreydKelly1972}
P.J. Freyd and G.M. Kelly.
\newblock Categories of continuous functors, {I}.
\newblock {\em Journal of Pure and Applied Algebra}, 2(3):169--191, 1972.
\newblock Erratum ibid. 4(2014):121.

\bibitem{Gur15}
N.~Gurski.
\newblock Operads, tensor products, and the categorical borel construction.
\newblock arXiv:1508.04050, 2015.

\bibitem{Johnstone1977}
P.~T. Johnstone.
\newblock {\em Topos Theory}.
\newblock LMS Monograph 10. Academic Press, London, 1977.

\bibitem{Kra87}
R.~Krasauskas.
\newblock Skew-simplicial groups.
\newblock {\em Lithuanian Mathematical Journal}, 27(1):47--54, 1987.
\newblock Translated from Litovski\u\i{} Matematichenski\u\i{} Sbornik
  (Lietuvos Matematikos Rinkinys), 27(1):89--99.

\bibitem{Lack2010}
S.~Lack.
\newblock A 2-categories companion.
\newblock In {John C.} Baez and {J. Peter} May, editors, {\em Towards higher
  categories}, IMA volumes in mathematics and its applications, pages 105--191.
  Springer, 2010.

\bibitem{Lur14}
J.~Lurie.
\newblock Higher algebra.
\newblock see author's webpage, September 2014.

\bibitem{McL98}
S.~MacLane.
\newblock {\em Categories for the Working Mathematician}.
\newblock Number~5 in Graduate Texts in Mathematics. Springer-Verlag, second
  ed. edition, 1998.

\bibitem{May1973}
J.~P. May.
\newblock {\em The geometry of iterated loop spaces}, volume 271 of {\em
  Lectures Notes in Mathematics}.
\newblock Springer-Verlag, Berlin-New York, 1972.

\bibitem{MayThomason1978}
J.~P. May and R.~Thomason.
\newblock The uniqueness of infinite loop space machines.
\newblock {\em Topology}, 17(3):205--224, 1978.

\bibitem{Rezk2010}
C.~Rezk.
\newblock A {C}artesian presentation of weak {$n$}-categories.
\newblock {\em Geometry \& Topology}, 14(1):521--571, 2010.

\bibitem{Segal1968}
G.~Segal.
\newblock Classifying spaces and spectral sequences.
\newblock {\em Publications Math\'ematiques de I'H\'ES}, 34:105--112, 1968.

\bibitem{Stasheff1963}
James~Dillon Stasheff.
\newblock Homotopy associativity of {$H$}-spaces. {I}, {II}.
\newblock {\em Transactions of the American Mathematical Society},
  108:293--312, 1963.

\bibitem{Wah01}
N.~Wahl.
\newblock {\em Ribbon braids and related operads}.
\newblock PhD thesis, University of Oxford, 2001.

\bibitem{YoshidaGrpop}
J.~Yoshida.
\newblock Group operads as crossed interval groups.
\newblock arXiv:1806.03012, 2018.

\bibitem{YoshidaLimcolim}
J.~Yoshida.
\newblock Limits and colimits of crossed groups.
\newblock arXiv:1802.06644, 2018.

\bibitem{Zha11}
W.~Zhang.
\newblock Group operads and homotopy theory.
\newblock arXiv:1111.7090, part of the Ph.D. thesis, 2011.

\end{thebibliography}
\end{document}